\documentclass[10pt,reqno]{amsart}

\usepackage{amssymb}
\usepackage{amsmath}
\usepackage{amsfonts, dsfont}
\usepackage{kotex}
\counterwithin{figure}{section}
\usepackage[colorlinks=true, pdfstartview=FitV, linkcolor=black, citecolor=black, urlcolor=black]{hyperref} 

\usepackage{amsthm} 
\usepackage[utf8]{inputenc} 
\usepackage{enumitem} 
\allowdisplaybreaks

\usepackage{graphicx}

\usepackage[utf8]{inputenc}
\usepackage{anyfontsize}

\usepackage{tikz}
\usetikzlibrary{
	calc,intersections,arrows.meta,patterns,
	decorations.pathmorphing,
	decorations.fractals
}

\DeclareGraphicsExtensions{.pdf,.png,.jpg}

\usepackage{color}

\theoremstyle{plain}

\newtheorem{thm}{Theorem}[section]
\newtheorem*{thm*}{Theorem}
\newtheorem{corollary}[thm]{Corollary}

\newtheorem{lemma}[thm]{Lemma}
\newtheorem{prop}[thm]{Proposition}
\newtheorem{statement}[thm]{Statement}
\newtheorem*{statement*}{Statement}

\theoremstyle{definition}   
\newtheorem{defn}[thm]{Definition}

\theoremstyle{remark}  
\newtheorem{remark}[thm]{Remark}
\newtheorem{example}[thm]{Example}

\newtheorem{innercustomgeneric}{\customgenericname}
\providecommand{\customgenericname}{}

\newcommand{\newcustomtheorem}[2]{\newenvironment{#1}[1]
	{\renewcommand\customgenericname{#2}
		\renewcommand\theinnercustomgeneric{##1}\innercustomgeneric}{\endinnercustomgeneric}}

\newcustomtheorem{customthm}{Theorem}

\newcommand\trho{\widetilde{\rho}}

\newcommand\bR{\mathbb{R}}
\newcommand\bC{\mathbb{C}}

\newcommand\bZ{\mathbb{Z}}

\newcommand\bS{\mathbb{S}}

\newcommand\bN{\mathbb{N}}

\newcommand\cA{\mathcal{A}}

\newcommand\cD{\mathcal{D}}
\newcommand\cF{\mathcal{F}}

\newcommand\cH{\mathcal{H}}
\newcommand\cI{\mathcal{I}}

\newcommand\cM{\mathcal{M}}

\newcommand\cO{\mathcal{O}}

\newcommand{\domain}{\mathcal{O}}

\newcommand\la{\langle}
\newcommand\ra{\rangle}
\newcommand\dd{\,\mathrm{d}}
\newcommand\ee{\mathrm{e}}

\newcommand{\mysection}[1]{\section{#1}
	\setcounter{equation}{0}}

\newcount\pgfdecorationorder
\pgfkeys{/pgf/decoration/.cd,
	Koch angle/.store in=\pgfkochangle, Koch angle=85,
	Koch order/.code={\global\pgfdecorationorder=#1}, Koch order=1
}
\pgfdeclaredecoration{Koch}{calculate}{
	\state{calculate}[width=0pt, next state=draw, persistent precomputation={
		\pgfmathparse{\pgfdecoratedinputsegmentlength/(2*(1+cos(\pgfkochangle)))}%
		\let\pgfkochsegmentlength=\pgfmathresult%
		\pgfmathparse{\pgfkochsegmentlength*sin(\pgfkochangle)}%
		\let\pgfkochy=\pgfmathresult%
		\pgfmathparse{\pgfkochsegmentlength*(1 + cos(\pgfkochangle))}%
		\let\pgfkochxa=\pgfmathresult%
		\pgfmathparse{\pgfkochsegmentlength*(1 + 2*cos(\pgfkochangle))}%
		\let\pgfkochxb=\pgfmathresult%
	}]{}
	\state{draw}[width=\pgfdecoratedinputsegmentlength]{
		\pgfpathmoveto{\pgfpointorigin}%
		\pgfpathlineto{\pgfqpoint{\pgfkochsegmentlength pt}{0pt}}%
		\pgfpathlineto{\pgfqpoint{\pgfkochxa pt}{\pgfkochy pt}}%
		\pgfpathlineto{\pgfqpoint{\pgfkochxb pt}{0pt}}%
		\pgfpathlineto{\pgfqpoint{\pgfdecoratedinputsegmentlength}{0pt}}%
	}
	\state{final}{
		\global\advance\pgfdecorationorder by -1\relax%
		\ifnum\pgfdecorationorder>0\relax%
		\pgfgetpath\decoratedpath%
		\pgfsetpath\empty%
		\begin{pgfdecoration}{{Koch}{\pgfdecoratedpathlength}}%
			\pgfsetpath\decoratedpath
		\end{pgfdecoration}%
		\fi
}}

\begin{document}

	\title[$L_p$-theory for Poisson's equation in non-smooth domains]
	{Weighted Sobolev space theory for Poisson's equation in non-smooth domains}

	\thanks{The author has been supported by a KIAS Individual Grant (MG095802) at Korea Institute for Advanced Study, and National Research Foundation of Korea (NRF-2019R1A5A1028324).}

	\author[J. Seo]{Jinsol Seo}
	\address[J. Seo]{School of Mathematics, Korea Institute for Advanced Study, 85 Hoegiro Dongdaemun-gu, Seoul 02455, Republic of Korea}
	\email{seo9401@kias.re.kr}


	
	\begin{abstract}
		We introduce a general $L_p$-solvability result for the Poisson equation in non-smooth domains $\Omega\subset \bR^d$, with the zero Dirichlet boundary condition.
		Our sole assumption on the domain $\Omega$ is the Hardy inequality: There exists a constant $N>0$ such that 
		$$
		\int_{\Omega}\Big|\frac{f(x)}{d(x,\partial\Omega)}\Big|^2\dd x\leq N\int_{\Omega}|\nabla f|^2 \dd x\quad\text{for any}\quad f\in C_c^{\infty}(\Omega)\,.
		$$
		To describe the boundary behavior of solutions in a general framework, we propose a weight system composed of a superharmonic function and the distance function to the boundary.
		Additionally, we explore applications across a variety of non-smooth domains, including convex domains, domains with exterior cone condition, totally vanishing exterior Reifenberg domains, and domains $\Omega\subset\bR^d$ for which the Aikawa dimension of $\Omega^c$ is less than $d-2$.
		Using superharmonic functions tailored to the geometric conditions of the domain, we derive weighted $L_p$-solvability results for various non-smooth domains and specific weight ranges that differ for each domain condition.
		Furthermore, we provide an application to the H\"older continuity of solutions.
	\end{abstract}
	
	\maketitle

	\setcounter{tocdepth}{3}
	
	\let\oldtocsection=\tocsection
	
	\let\oldtocsubsection=\tocsubsection
	
	\let\oldtocsubsubsection=\tocsubsubsectio
	
	\renewcommand{\tocsection}[2]{\hspace{0em}\oldtocsection{#1}{#2}}
	\renewcommand{\tocsubsection}[2]{\hspace{1em}\oldtocsubsection{#1}{#2}}
	\renewcommand{\tocsubsubsection}[2]{\hspace{2em}\oldtocsubsubsection{#1}{#2}}
	
	\tableofcontents
	
	\mysection{Introduction}\label{sec:Introduction}
	The Poisson equation is among the most classical and fundamental partial differential equations.
	$L_p$-theory for this equation in $\bR^d$ and $C^2$-domains has been developed over a long period of time, together with Schauder theory and $L_2$-theory.
	In particular, the theory has been extended in various directions, including equations with variable coefficients \cite{KK2007_variable, Krylov2007, Krylov2009}, nonlocal or nonlinear operators \cite{Iwaniec1983,Pseudodiff}, and non-smooth domains.
	
	Our primary focus is the Poisson equation on \textit{non-smooth} domains $\Omega$, with the zero Dirichlet boundary condition:
	\begin{alignat}{3}
		\Delta u=f&\quad\text{in}\,\,\,\Omega\quad&&;\quad u=0\quad \text{on}\,\,\,\partial\Omega\,. \label{ellip}
	\end{alignat}	
	Unweighted and weighted $L_p$-theories for this equation have been developed in various types of domains, including
	$C^1$-domains \cite{doyoon, KK2004}, Reifenberg domains \cite{Relliptic}, convex domains \cite{convexAdo,convexFromm}, Lipschitz domains \cite{kenig}, domains with Ahlfors regular boundary \cite{MBX2022}, domains with point singularities \cite{MR0492821, MNP}, and piecewise smooth domains \cite{BK2006,MR}.
	Despite extensive analyses of the Poisson equation across these domains, a comprehensive $L_p$-theory for general non-smooth domains remains elusive.
	
	This paper presents a general result for the weighted $L_p$-solvability for \eqref{ellip} in non-smooth domains.
	We consider domains $\Omega\subsetneq\bR^d$ admitting the Hardy inequality: There exists a constant $\mathrm{C}_0(\Omega)>0$ such that
	\begin{align}\label{hardy}
		\int_{\Omega}\Big|\frac{f(x)}{d(x,\partial\Omega)}\Big|^2\dd x\leq \mathrm{C}_0(\Omega)\int_{\Omega}|\nabla f(x)|^2 \dd x\quad\text{for all}\quad f\in C_c^{\infty}(\Omega)\,.
	\end{align}
	One of the notable sufficient conditions for \eqref{hardy} is the volume density condition:
	\begin{align}\label{230212413}
		\inf_{\substack{p\in\partial\Omega\\r>0}}\frac{\big|\Omega^c \cap B_r(p)\big|}{\big|B_r(p)\big|}>0
	\end{align}
	(see Remark \ref{24032030232}).
	We also use a class of superharmonic functions, called superharmonic Harnack functions, as a tool for constructing weight functions in our $L_p$-estimate.
	Roughly speaking, we establish the following result: For equation \eqref{ellip} in a domain $\Omega$ with \eqref{hardy}, each \textit{superharmonic Harnack function} $\psi$ yields a corresponding weighted $L_p$-solvability result for any $p\in(1,\infty)$.
	In this result, $\psi$ describes the boundary behavior of solutions.
	We apply our result to various types of non-smooth domains by constructing appropriate superharmonic functions.
	Detailed discussions of the main result and its applications are given in Sections \ref{230214201} and \ref{0003}, respectively.
		
	\subsection{Historical remarks and overview of the main results.}\label{230214201}
	\,\,
	
	\vspace{1mm}
	\noindent
	\textbf{Historical remarks on $L_p$-solvability in non-smooth domains.}
	Studies of $L_p$-theory for non-smooth domains have mainly focused on the individual analysis of specific domain classes.
	One of the most significant contributions to this line of research was made by Jerison and Kenig \cite{kenig} for Lipschitz domains.
	The authors proved the following results for domains $\Omega\subset \bR^d$, $d\geq 3$ (resp. $d=2$):
	\begin{enumerate}
		\item If $p\in [3/2,3]$ (resp. $p\in[4/3,4]$), then for any bounded Lipschitz domain $\Omega$, the Poisson equation \eqref{ellip} has a unique solution in $\mathring{L}^p_1(\Omega)$ whenever $f\in L^p_{-1}(\Omega)$.
		
		\item For each $p>3$ (resp. $p>4$), there exists a bounded Lipschitz domain $\Omega$ and $f\in C^{\infty}(\overline{\Omega})$ such that \eqref{ellip} has no solution in $\mathring{L}^p_1(\Omega)$.
	\end{enumerate}
	(For the definition of function spaces $\mathring{L}^p_1(\Omega)$ and $L^p_{-1}(\Omega)$, see Remark \ref{230215958}.)
	The first result establish a universal range of $p$ that ensures unique solvability in unweighted Sobolev spaces.
	However, the second result shows that it is \textit{impossible} to establish a general unique solvability theorem in $\mathring{L}^p_1(\Omega)$ that holds for all $p\in(1,\infty)$ and for all Lipschitz domains $\Omega$.
	Given these limitations in unweighted Sobolev spaces, we turn our attention to theories in weighted Sobolev spaces.
	
	Elliptic equations in smooth or polygonal cones (conic domains) have been extensively studied in the literature, as indicated in monographs \cite{BK2006,MNP,MR}.
	Here, 
	\begin{align}\label{240304303}
		\Omega:=\{r\sigma\,:\,r>0\,\,\,\,\text{and}\,\,\,\, \sigma\in\cM\}\quad(\cM\subset \bS^{d-1})
	\end{align}
	is called a smooth cone if $\cM$ is a smooth subdomain of $\bS^{d-1}$, and a polygonal cone if $\cM$ is a spherical polygon.
	For these domains, scholars have investigated weighted $L_p$-theories for elliptic equations, valid for all $p\in(1,\infty)$.
	In these theories, the weight system is composed of distance functions corresponding to the vertices and edges of the domain, and the admissible range of weights for unique solvability is closely related to the \textit{eigenvalues} of the spherical Laplacian on $\cM$.
	For example, consider the case where $\cM=\{\,(\cos\theta,\sin\theta)\,:\,0<\theta<\kappa\}\subset \bS^1$, $\kappa\in(0,2\pi)$, and let $\Omega\subset \bR^2$ be defined by \eqref{240304303}.
	For any $p\in(1,\infty)$ and $\frac{2}{p}-\frac{\pi}{\kappa}<\mu<\frac{2}{p}+\frac{\pi}{\kappa}$, we have the estimate
	$$
	\big\||x|^{-\mu}u\big\|_p+\big\||x|^{-\mu+1}Du\big\|_p+\big\||x|^{-\mu+2}D^2u\big\|_p\lesssim \big\||x|^{-\mu+2}\Delta u\big\|_p
	$$
	for $u\in C_c^{\infty}(\Omega)$ (see \cite[2.6.6. Example]{MR}).
	The value of $\mu$ describes the behavior of solutions near the vertex, and the quantity $\frac{\pi}{\kappa}$ in the range of $\mu$ is directly related to the first eigenvalue of $\frac{\mathrm{d}^2}{\mathrm{d}\theta^2}$ on $\cM$.
	 
	The aforementioned studies indicate that, to develop a general framework for the $L_p$-solvability of the Poisson equation in various non-smooth domains, it is necessary to adopt a weight system associated with the Laplace operator and the geometric features of each domain.
	Furthermore, this weight system enables us to describe the \textit{boundary behavior} of solutions.
	
	Numerous other notable works have addressed various non-smooth domains.
	Section \ref{0003} summarizes prior works relevant to several types of non-smooth domains and presents our result in each situation.
	Before introducing our result, we briefly discuss one of the primary methods of this paper.

			\vspace{1mm}
	\noindent\textbf{Remark on the localization argument.}
	One of our primary methods is the localization argument developed by Krylov \cite{Krylov1999-1}.
	In that work, the Poisson equation in the half space $\bR_+^d$ was studied, and one of the main results can be stated as follows:
	If $\frac{1}{p}<\mu<1+\frac{1}{p}$, then for any $u\in C_c^{\infty}(\bR_+^d)$ and $f_0,\,f_1,\,\ldots,\,f_d$ such that $\Delta u=f_0+\sum_{i\geq 1} D_if_i$, we have
	\begin{align}
		\|\rho^{-\mu}u\|_p+\|\rho^{1-\mu}Du\|_p\,&\lesssim\|\rho^{-\mu} u\|_p+\|\rho^{2-\mu}f_0\|_p+\sum_{i\geq 1} \|\rho^{1-\mu}f_i\|_p\label{220930316}\\
		&\lesssim \|\rho^{2-\mu}f_0\|_p+\sum_{i\geq 1} \|\rho^{1-\mu}f_i\|_p\,,\label{220930317}
	\end{align}
	where $\rho(x):=d(x,\partial\bR^d_+)$ is the boundary distance function on $\bR_+^d$.
	(For the motivation of such estimates, see the front of Section \ref{0042}.)
	The parameter $\mu$ describes the boundary behavior of solutions and their derivatives (for instance, when $\mu=1$).
	The range $\frac{1}{p}<\mu<1+\frac{1}{p}$ is sharp, as noted in \cite[Remark 4.3]{Krylov1999-1}.
	From a technical point of view, this range follows from the proof of \eqref{220930317} in which the weighted Hardy inequalities for $\bR_+$ and their sharp constants play crucial roles.
	On the other hand, to derive estimate \eqref{220930316}, the author applied a localization argument based on, among other things, results for the Poisson equation in the whole space $\bR^d$.
	We note that this localization argument is applicable to \textit{any domain $\Omega$ and any $\mu\in\bR$}, not just to $\bR^d_+$ and specific $\mu$, as shown in \cite{Kim2014,ConicPDE}.

	While Krylov \cite{Krylov1999-1} dealt with only the half space because of estimate \eqref{220930317}, Kim \cite{Kim2014} revealed a connection between the approach in \cite{Krylov1999-1} and the classical Hardy inequality \eqref{hardy} for non-smooth domains.
	Kim \cite{Kim2014} studied stochastic parabolic equations in non-smooth domains and obtained estimates of type \eqref{220930316} and \eqref{220930317} for bounded domains $\Omega$ admitting the Hardy inequality, instead of $\bR_+^d$.
	However, it should be noted that in \cite[Theorem 2.12]{Kim2014}, the range of $\mu$ for the solvability is restricted to around $\frac{2}{p}$, and this range is not specified; briefly speaking, the boundary behavior of solutions is not described sufficiently well (cf. Krylov's work on $\bR_+^d$ mentioned above).
	
		\vspace{1mm}
		\noindent
		\textbf{Overview of the main result.}
		Following \cite{Kim2014}, we restrict our attention to domains admitting the Hardy inequality.
		This choice is motivated by the observation that the Hardy inequality holds in various non-smooth domains (see \eqref{230212413}).
		
		A distinctive feature of the present paper is the utilization of \textit{superharmonic functions}.
		We employ superharmonic functions in conjunction with the Hardy inequality.
		This combination enables us to capture precisely the boundary behavior of solutions (see \eqref{2401301249} or  Theorem \ref{21.05.13.2}).
		Furthermore, we introduce the concepts of \textit{Harnack functions} and \textit{regular Harnack functions}, extending the localization argument developed in \cite{Krylov1999-1} to a broader class of weight functions.
		Consequently, we utilize superharmonic Harnack functions $\psi$ as weight functions; each such function is locally integrable and satisfies the following conditions:
		\begin{enumerate}
			\item $\Delta \psi\leq 0$ in the sense of distributions.
			
			\item $\psi>0$ and there exists a constant $N>0$ such that 
			$$
			\underset{B(x,\rho(x)/2)}{\mathrm{ess\,sup}}\,\psi\leq 	N\underset{B(x,\rho(x)/2)}{\mathrm{ess\,inf}}\,\psi\quad\text{for all}\,\,\,\,x\in\Omega\,,
			$$
			where $\rho(x):=\mathrm{dist}(x,\partial\Omega)$.
		\end{enumerate}
		
		Our main result (Theorem \ref{21.09.29.1}) contains the following estimate:
		\begin{itemize}
			\item[] Let $\Omega$ admit the Hardy inequality \eqref{hardy} and $\psi$ be a superharmonic Harnack function on $\Omega$.
			For any $1<p<\infty$ and $-\frac{1}{p}<\mu<1-\frac{1}{p}$, it holds that for any $u\in C_c^\infty(\Omega)$ and $f_0,\,f_1,\,\ldots,\,f_d$ such that $\Delta u=f_0+\sum_{i\geq 1}D_if_i$, we have
			\begin{align}\label{2401301249}
				\|\psi^{-\mu}\rho^{-2/p}u\|_p+\|\psi^{-\mu}\rho^{-2/p+1}Du\|_p\lesssim 	\|\psi^{-\mu}\rho^{-2/p+2}f_0\|_p+\sum_{i\geq 1}\|\psi^{-\mu}\rho^{-2/p+1}f_i\|_p\,.
			\end{align}
		\end{itemize}
		Here, the superharmonic Harnack function $\psi$ describes the boundary behavior of solutions.
		By applying the Sobolev-H\"older embedding theorem, we also derive pointwise estimates for solutions (see Theorem \ref{240307248} and Proposition \ref{220512537}).
		
		Our main result does not specify a particular superharmonic Harnack function $\psi$.
		The flexibility in the choice of $\psi$ is the primary advantage of our theorem, enabling applications in a wide range of non-smooth domains.
		We offer a non-trivial general example of $\psi$ related to the Green functions in Example \ref{220912411}.
		Additionally, in Sections \ref{app.} and \ref{app2.}, we explore the construction of suitable $\psi$ for various geometric domain conditions.
		The domain conditions we investigate include the following:
		\begin{enumerate}		
			\item Domains satisfying the exterior cone condition, and planar domains satisfying the exterior line segment condition;
			\item Convex domains;
			\item Domains satisfying the totally vanishing exterior Reifenberg condition;
			\item Domains $\Omega$ satisfying the volume density condition \eqref{230212413};
			\item Domains $\Omega\subset \bR^d$ for which the Aikawa dimension of $\Omega^c$ is less than $d-2$.
		\end{enumerate}		
		For a domain $\Omega$ under each condition above, we construct suitable superharmonic functions $\psi$ such that $\psi\simeq d(\cdot,\partial\Omega)^\alpha$ for some $\alpha\in\bR$.
		Notably, the range of $\alpha$ is different for each domain condition.
		We sequentially introduce simplified versions of our results for the aforementioned conditions in Sections \ref{0003}.1 - \ref{0003}.5, together with earlier works for each domain condition.
		
		We close this subsection with brief comments on possible extensions of the present framework, omitting a separate summary of the organization since the table of contents is provided at the beginning of the paper.
		
		The approach developed here also applies to linear evolution equations governed by the Laplace operator, including the classical heat equation, time-fractional heat equation, and stochastic heat equation.
		In particular, the constructions of superharmonic functions in Sections~\ref{app.} and~\ref{app2.} can be readily employed for these problems.
		Applications to the classical heat equation and the time-fractional heat equation are discussed in \cite{Seo202411} by the present author, while an application to the stochastic heat equation is left for future work.

		Regarding variable coefficients, several studies have investigated elliptic equations in Krylov-type weighted settings (see, \textit{e.g.}, \cite{DK2015, KK2004, KL2013}).
		However, since those results crucially rely on the flatness of half-spaces or $C^1$ domains, their methods do not extend directly to general domains with limited boundary regularity.
		On the other hand, the inclusion of lower-order terms with variable coefficients may be feasible, as can be inferred from the argument in \cite{KK2004}.

		\subsection{Summary of applications under various domain conditions}\label{0003}		
		In this subsection, we consider a domain $\Omega\subset \bR^d$, $d\geq 2$, and denote by $\rho(x):=d(x,\partial\Omega)$ the distance to the boundary.
		For $p\in (1,\infty)$, $\theta\in\bR$, and $n\in \{0,1,2,\ldots\}$, we define
		\begin{alignat*}{2}
				&\|f\|_{W_{p,\theta}^n(\Omega)}&&:=\sum_{k=0}^n\|\rho^kD^kf\|_{L_{p,\theta}(\Omega)}:=\sum_{k=0}^n\bigg(\int_\Omega\big|\rho(x)^{k}D^kf(x)\big|^p\rho(x)^\theta\dd x\bigg)^{1/p}\,,\\
				&\|f\|_{W_{p,\theta}^{-n}(\Omega)}&&:=\inf\bigg\{\sum_{|\alpha|\leq n}\|\rho^{-|\alpha|}f_\alpha\|_{L_{p,\theta}(\Omega)}\,:\,f=\sum_{|\alpha|\leq n}D^{\alpha}f_{\alpha}\bigg\}\,.
		\end{alignat*}
		For $n\in\bZ$, $W_{p,\theta}^{n}(\Omega)$ is defined as the set of all $f\in\cD'(\Omega)$ such that $\|f\|_{W_{p,\theta}^{n}(\Omega)}<\infty$.	
		
		\begin{remark}\label{230214208}
		The spaces $W_{p,\theta}^{n}(\Omega)$, $n\in\bZ$, are used only in this subsection.
		However, it coincides (up to equivalence of norms) with $H_{p,\theta+d}^{n}(\Omega)$ (see Lemma~\ref{220512433}), where $H_{p,\theta+d}^{n}(\Omega)$ is the function space introduced in Definition \ref{220610533}.
		\end{remark}
		
		For convenience, we define the following statement:
		\begin{statement}[$\Omega,p,\theta$]\label{230203621} Let $\lambda\geq 0$.
		For any $n\in\bZ$, if $f\in W_{p,\theta}^{n}(\Omega)$, then the equation $\Delta u-\lambda u=f$ has a unique solution $u$ in $W_{p,\theta+2p}^{n+2}(\Omega)$.
				Moreover, we have
				\begin{align}\label{220923207}
					\|u\|_{W_{p,\theta}^{n+2}(\Omega)}+\lambda \|u\|_{W_{p,\theta+2p}^{n}(\Omega)}\leq N\|f\|_{W_{p,\theta+2p}^{n}(\Omega)}\,,
				\end{align}
				where $N$ is independent of $f$, $u$, and $\lambda$.
		\end{statement}
				
		\vspace{1mm}\noindent
		\textbf{\ref{0003}.1. (Section \ref{0071}) Domains with exterior cone condition.}
		For $\delta\in[0,\pi/2)$ and $R>0$, $\Omega$ is said to satisfy the \textit{exterior $(\delta,R)$-cone condition} if for every $p\in\partial\Omega$, there exists a unit vector $e_p\in\bR^d$ such that
		$$
		\{x\in\bR^d\,:\,(x-p)\cdot e_p\geq |x-p|\cos\delta\,\,,\,\,|x-p|<R\}\subset \Omega^c\,.
		$$
		When $\delta=0$, this condition also known as the exterior $R$-line segment condition.
		Examples of this condition are given in Example \ref{220816845} and illustrated in Figure \ref{230212745}.
		
		Given $\delta>0$, we denote
		$$
		\lambda_{\delta}:=-\frac{d-2}{2}+\sqrt{\Big(\frac{d-2}{2}\Big)^2+\Lambda_{\delta}}\,,
		$$
		where $\Lambda_{\delta}>0$ is the first eigenvalue of Dirichlet spherical Laplacian on
		$$
		\{\sigma=(\sigma_1,\ldots,\sigma_d)\in \bS^{d-1}\,:\,\sigma_1>-\cos\delta\}\,.
		$$
		When $d=2$ and $\delta=0$, we set $\lambda_{\delta}=1/2$.
		Information on $\lambda_{\delta}$ is stated in \eqref{230212803} and Proposition \ref{230212757}.
		Note that $\lambda_\delta>0$ for all $\delta>0$, and if $d=2$, then $\lambda_\delta=\frac{\pi}{2(\pi-\delta)}\geq \frac{1}{2}$ for all $\delta\geq 0$.
		
		Our result also covers some unbounded domains, but here, we restrict our presentation to bounded domains.
		
		\begin{thm}[see Theorem \ref{221026914}]\label{2302141109}
			Let $\delta\in(0,\pi)$ if $d\geq 3$, and $\delta\in[0,\pi)$ if $d= 2$.
			Suppose that $\Omega\subset \bR^d$ is a bounded domain satisfying the $(\delta,R)$-exterior cone condition for some $R>0$.
			Then for any $p\in(1,\infty)$ and $\theta\in\bR$ satisfying 
			$$
			-\lambda_{\delta}(p-1)-2<\theta<\lambda_{\delta}-2\,,
			$$
			Statement \ref{230203621} $(\Omega,p,\theta)$ holds.
			In addition, $N$ in \eqref{220923207} depends only on $d$, $p$, $n$, $\theta$, $\delta$, $\mathrm{diam}(\Omega)/R$.
		\end{thm}
		
		The exterior cone condition is more general than the Lipschitz boundary condition.
		It should be noted, however, that Theorem \ref{221026914} and the work of Jerison and Kenig \cite[Theorems 1.1, 1.3]{kenig} (for Lipschitz domains) cannot be directly compared, as they address different aspects of the Poisson equation in non-smooth domains.
		While Theorem \ref{221026914} covers a \textit{broader} class of domains than \cite{kenig}, when the focus is restricted to Lipschitz domains, the results in \cite{kenig} are more general in terms of unweighted estimates of higher regularity.
		For a detailed comparison between \cite{kenig} and Theorem \ref{221026914}, the reader is referred to the following remark following remark on the relations between the function spaces $H_{p,\theta+d}^\gamma(\Omega)$ (see Remark \ref{230214208}) and the Sobolev spaces presented in \cite{kenig}:
		
		\begin{remark}\label{230215958}
			Let $\Omega$ be a bounded Lipschitz domain.
			We refer to the function space $L_s^p(\Omega)$ and $L_{s,\mathrm{o}}^p(\Omega)$ as introduced in \cite[Section 2]{kenig}, where $p\in(1,\infty)$ is the integrability parameter and $s\in\bR$ is the regularity parameter.
			For clarity, we use the notation $\mathring{L}_s^p(\Omega)$ to denote the space $L_{s,\mathrm{o}}^p(\Omega)$.
			Note that $L_k^p(\Omega)=W_p^k(\Omega)$ for $k\in \bN_0$. 
			The space $\mathring{L}_k^p(\Omega)$ is defined as the closure of  $C_c^{\infty}(\Omega)$ in $L_k^p(\Omega)$, and $L_{-k}^p(\Omega)$ is defined as the dual space of $\mathring{L}_{k}^{p/(p-1)}(\Omega)$.

			It directly follows from the definition that $H_{p,d}^0(\Omega)=L^p_0(\Omega)=L_p(\Omega)$.
			For $k\in\bN$, the weighted Hardy inequality for Lipschitz domains (see, \textit{e.g.}, \cite{MR3168477}) and the boundedness of $\Omega$ implies that $\|u\|_{H_{p,d-kp}^k(\Omega)}\simeq \|u\|_{L_k^p(\Omega)}$ for all $u\in C_c^{\infty}(\Omega)$.
			Since $C_c^{\infty}(\Omega)$ is dense in each of $H_{p,d-kp}^k(\Omega)$ and $\mathring{L}_k^p(\Omega)$ (see Lemma \ref{21.09.29.4}.(1) with $\Psi\equiv 1$), $H_{p,d-kp}^k(\Omega)$ coincides with $\mathring{L}_k^p(\Omega)$.
			The interpolation properties for $\mathring{L}_s^p(\Omega)$ and $H_{p,\theta}^{\gamma}(\Omega)$ (see \cite[Corollary 2.10]{kenig} and \cite[Proposition 2.4]{Lo1}, respectively) imply that
			$H^{s}_{p,d-sp}(\Omega)=\mathring{L}^p_{s}(\Omega)$ for all $s>0$.
			By taking duals (see Lemma \ref{21.09.29.4}.(2)), we also have 
			$H^{-s}_{p,d+sp}(\Omega)=L^p_{-s}(\Omega)$ for all $s>0$.			
		\end{remark}
		
		\vspace{1mm}\noindent
		\textbf{\ref{0003}.2. (Section \ref{convex}) Convex domain.}
		
		\begin{thm}[see Theorem \ref{2208131026}]\label{221005646}
			Let $d\geq 2$ and suppose that $\Omega$ is a convex domain (not necessarily bounded).
			For any $p\in(1,\infty)$ and $\theta\in\bR$ satisfying
			$$
			-p-1<\theta<-1\,,
			$$
			Statement \ref{230203621} $(\Omega,p,\theta)$ holds.
			In addition, $N$ in \eqref{220923207} depends only on $d$, $p$, $n$, $\theta$.
			In particular, $N$ is independent of $\Omega$.
		\end{thm}
		
		Adolfsson \cite{convexAdo} and Fromm \cite{convexFromm} established the solvability of the Poisson equation in bounded convex domains.
		Regarding unweighted estimates for higher regularity, their results are more general than Theorem \ref{2208131026}.
		However, Theorem \ref{2208131026} deals with convex domains that are not necessarily bounded, and further establishes solvability results in \textit{weighted} Sobolev spaces.
		When comparing these results with Theorem \ref{2208131026}, it is helpful to note Remark \ref{230215958} and the fact that bounded convex domains are Lipschitz domains (see, \textit{e.g.}, \cite[Corollary 1.2.2.3]{nonsmoothGris}).
		
		Combining the results of Theorem~\ref{2208131026} with \cite[Theorem 3.2.1.2]{nonsmoothGris} may yield results similar to \cite[Corollary 1]{convexFromm}; however, we do not pursue this direction in this paper.

		\vspace{1mm}\noindent
		\textbf{\ref{0003}.3. (Section \ref{ERD}) Totally vanishing exterior Reifenberg condition.}
		This part introduces the \textit{totally vanishing exterior Reifenberg} condition (abbreviated as `$\langle\mathrm{TVER}\rangle$'), which is a generalization of the concept of bounded vanishing Reifenberg domains introduced below \eqref{22.02.26.41}.
		
		To clarify the main point of $\langle\mathrm{TVER}\rangle$ presented in Definition \ref{2209151117}.(3), we introduce a simplified variant of this concept in Definition \ref{221013228}, denoted by $\langle\mathrm{TVER}\rangle^\ast$.
		Note that $\langle\mathrm{TVER}\rangle^\ast$ is a sufficient condition for $\langle\mathrm{TVER}\rangle$.
		In Figure \ref{230212856}, we illustrate the differences among the vanishing Reifenberg condition, $\langle\mathrm{TVER}\rangle^\ast$ in Definition \ref{221013228}, and $\langle\mathrm{TVER}\rangle$ in Definition \ref{2209151117}.(3).

		\begin{defn}\label{221013228}
			We say that $\Omega$ satisfies $\langle\mathrm{TVER}\rangle^\ast$ if for any $\delta\in(0,1)$, there exist $R_{0,\delta},\,R_{\infty,\delta}>0$ such that the following holds: For every $p\in\partial \Omega$ and $r>0$ with $r\leq R_{0,\delta}$ or $r\geq R_{\infty,\delta}$, there exists a unit vector $e_{p,r}\in\bR^d$ such that
			\begin{align}\label{230203624}
				\Omega\cap B_r(p)\subset \{x\in B_r(p)\,:\,(x-p)\cdot e_{p,r}<\delta r\}\,.
			\end{align}
		\end{defn}
		
		As shown in Example \ref{220910305}, $\langle\mathrm{TVER}\rangle^\ast$ is satisfied by the following classes of bounded domains: vanishing Reifenberg domains, $C^1$-domains, domains with the exterior ball condition, and finite intersections of them.
		Furthermore, several unbounded domains also satisfy $\langle\mathrm{TVER}\rangle^\ast$ (see Proposition \ref{220918248}).
				
		\begin{thm}[see Theorem \ref{22.07.17.109}]\label{230214437}
			Suppose that $\Omega$ satisfies $\langle\mathrm{TVER}\rangle^\ast$.
			For any $p\in(1,\infty)$ and $\theta\in\bR$ satisfying
			$$
			-p-1<\theta<-1\,,
			$$
			Statement \ref{230203621} $(\Omega,p,\theta)$ holds.
			In addition, $N$ in \eqref{220923207} depends only on $d$, $p$, $n$, $\theta$, $\big\{R_{0,\delta}/R_{\infty,\delta}\big\}_{\delta\in(0,1]}$.
		\end{thm}

		The Poisson equation in bounded vanishing Reifenberg domains has been investigated in the literature, in the works of Byun and Wang \cite{Relliptic}, Choi and Kim \cite{Reifweight2}, and Dong and Kim \cite{DongKim}.
		These studies focused on elliptic equations with variable coefficients, and also provide weighted $L_p$-estimates for Muckenhoupt $A_p$-weight functions.
		However, these studies mostly dealt with bounded vanishing Reifenberg domains.
		In contrast to these works, Theorem \ref{230214437} considers domains satisfying $\langle\mathrm{TVER}\rangle^\ast$, thereby including bounded vanishing Reifenberg domains.

		\vspace{1mm}\noindent
		\textbf{\ref{0003}.4. (Section \ref{fatex}) Domains with fat exterior.}
		Let $\Omega$ be a domain satisfying the \textit{capacity density condition}:
		\begin{align}\label{2302101253}
			\inf_{\substack{p\in\partial\Omega\\r>0}}\frac{\mathrm{Cap}\left(\Omega^c\cap \overline{B}_r(p),B_{2r}(p)\right)}{\mathrm{Cap}\left(\overline{B}_r(p),B_{2r}(p)\right)}\geq \epsilon_0>0\,,
		\end{align}
		where $\mathrm{Cap}(K,U)$ denotes the $L_2$-capacity of $K$ relative to $U$, as defined in \eqref{230324942}.
		Condition \eqref{2302101253} has been studied in the literature, including \cite{aikawa2002, Aikawa2009, AA, KilKos1994, kinnunen2021, lewis}.
		It is worth noting that the volume density condition \eqref{230212413} is a sufficient condition for \eqref{2302101253} (see Remark \ref{24032030232}).
			
		In Section \ref{fatex}, we consider another condition equivalent to \eqref{2302101253}, called the local harmonic measure decay condition.
		For clarity, we present several corollaries in place of the main result (Theorem \ref{22.02.19.3}).

			\begin{thm}[see Corollary \ref{240120415} with Lemma \ref{240320301}]\label{240307248}
				Let $\Omega$ be a bounded domain satisfying \eqref{2302101253}.
				There exists $\alpha_0>0$ depending only on $d$, $N_0$, and $\epsilon_0$ (in \eqref{2302101253}) such that for any $\alpha\in(0,\alpha_0]$, the following holds:
				Let $\lambda\geq 0$ and $f_0,\,f_1,\,\ldots,\,f_d$ be measurable functions such that $|f_0|\lesssim \rho^{-2+\alpha}$ and $|f_1|,\,\ldots,\,|f_d|\lesssim \rho^{-1+\alpha}$.
				For any $\beta<\alpha$, the equation
				\begin{align}\label{240129716}
					\Delta u-\lambda u=f_0+\sum_{i\geq 1} D_i f_i\quad\text{in}\quad \Omega\quad;\quad u=0\quad\text{on}\quad \partial\Omega
				\end{align}
				has a unique solution $u$ in $C^{0,\beta}(\Omega)$.
				In addition, we have
				$$
				\sup_\Omega\rho^{-\beta}|u|+\|u\|_{C^{0,\beta}(\Omega)}\leq N \sup_{\Omega}\Big(\rho^{-2+\alpha}|f_0|+\sum_{i\geq 1}\rho^{-1+\alpha}|f_i|\Big)\,,
				$$
				where $N$ depends only on $d$, $|\Omega|$, $\epsilon_0$ (in \eqref{2302101253}), $\alpha$, $\beta$.
				\end{thm}

			\begin{remark}[see Remark \ref{240307323}]
			Theorem \ref{240307248} still holds for bounded domains $\Omega$ satisfying the following assumption, in place of \eqref{2302101253}:
			\begin{itemize}
				\item[] For any $F\in C(\partial\Omega)$, the Laplace equation 
				\begin{align*}
					\Delta u=0\quad\text{in}\,\,\,\Omega\quad;\quad u=F\quad\text{on}\,\,\,\partial \Omega
				\end{align*}
				has a unique classical solution $u\in C(\overline{\Omega})$.
				Additionally, there exists $\alpha_1\in(0,1)$ such that 
				$\|u\|_{C^{0,\alpha_1}(\Omega)}\leq N \|F\|_{C^{0,\alpha_1}(\partial\Omega)}$, where $N$ is a constant independent of $u$ and $F$.
			\end{itemize}
			Under this revised assumption, $\alpha_0$ in Theorem \ref{240307248} can be chosen as $\alpha_1$ in the revised assumption.
			\end{remark}

			We also prove an unweighted $L_p$-theory result for \eqref{240129716}, where $p$ is close to $2$.
			Although similar results have been established in the literature (see the discussion preceding Corollary~\ref{230210356}), we present the following theorem to emphasize the applicability of our main result:

		\begin{thm}[see Corollary \ref{230210356}]\label{2304141256}
			Let $\Omega$ satisfy \eqref{2302101253}, and let
			\begin{align*}
				\lambda\geq 0\quad \text{if}\quad D_{\Omega}:=\sup_{x\in\Omega}d(x,\partial\Omega)<\infty\,\,\,,\,\,\,\text{and}\quad \lambda> 0\quad \text{if}\quad D_{\Omega}=\infty\,.
			\end{align*}
			Then there exists $\epsilon\in(0,1)$ depending only on $d$ and $\epsilon_0$ (in \eqref{2302101253}) such that for any $p\in(2-\epsilon,2+\epsilon)$, the following holds:
			For any $f_0,\,f_1,\,\ldots,\,f_d\in L_p(\Omega)$, equation \eqref{240129716}
			has a unique solution $u$ in $\mathring{W}^{1}_{p}(\Omega)$ ($:=$ the closure of $C_c^{\infty}(\Omega)$ in $W_p^1(\Omega)$).
			Moreover, we have
			\begin{align*}
				\begin{split}
					\|\nabla u\|_{p}+\big(\lambda^{1/2}+D_\Omega^{-1}\big)\|u\|_{p}\lesssim_{d,p,\epsilon_0} \min\big(\lambda^{-1/2},D_\Omega\big)\|f_0\|_{p}+\sum_{i\geq 1}\|f^i\|_p\,.
				\end{split}
			\end{align*}
		\end{thm}
		
		\noindent
		\textbf{\ref{0003}.5. (Section \ref{0062}) Domains with thin exterior.}
		For a closed set $E\subset \bR^d$, the \textit{Aikawa dimension} of $E$, denoted by $\dim_{\cA}(E)$, is defined as the infimum of $\beta\geq 0$ such that
		$$
		\sup_{p\in\Omega^c,r>0}\frac{1}{r^{\beta}}\int_{B(p,r)}\frac{1}{d(x,E)^{d-\beta}}\dd x\leq A_{\beta}<\infty\,,
		$$		
		with the convention $0^{-1}=\infty$.
		We consider a domain $\Omega$ for which $\dim_{\cA}(\Omega^c)<d-2$.
		A relation between the Aikawa dimension, the Hausdorff dimension, and the Assouad dimension is discussed in Remark \ref{22.02.24.1}.
		
		\begin{thm}[see Theorem \ref{22.02.19.300}]
			Let $d\geq 3$ and let $\Omega\subset \bR^d$ satisfy $\dim_{\cA}(\Omega^c)=:\beta_0<d-2$.
			For any $p\in(1,\infty)$ and $\theta\in\bR$ satisfying 
			$$
			-d+\beta_0<\theta<(p-1)(d-\beta_0)-2p\,,
			$$
			Statement \ref{230203621} $(\Omega,p,\theta)$ holds.
			In addition, $N$ in \eqref{220923207} depends only on $d$, $p$, $n$, $\theta$, $\beta_0$, $\{A_{\beta}\}_{\beta>\beta_0}$.
			
		\end{thm}
	
		\subsection{Notation}\label{0004}
		
		\begin{itemize}
			\item We use $:=$ to denote a definition.

			\item  The letter $N$ denotes a finite positive constant which may have different values along the argument  while the dependence  will be informed;  $N=N(a,b,\cdots)$ means that this $N$ depends only on the parameters inside the parentheses.
			
			\item  For a list of parameters $L$, $A\lesssim_{L} B$ means that $A\leq N(L)B$, and $A\simeq _{L}B$ means that $A\lesssim_{L} B$ and $B\lesssim_{L} A$.
			
			\item $a \vee b :=\max\{a,b\}$, $a \wedge b :=\min\{a,b\}$. 
			
			\item For a Lebesgue measurable set $E\subset \bR^d$, $|E|$ denotes the Lebesgue measure of $E$.
			
			\item $\bN_0:=\bN\cup\{0\}$, $\bR^d_+:=\{(x^1,\ldots,x^d)\in\bR^d: x^1>0\}$, $\bR_+:=\bR_+^1$, and $\bS^{d-1}:=\big\{x\in\bR^d\,:\,|x|=1\big\}$.
			In addition, for $p\in\bR^d$ and $r>0$,  $B_r(p):=B(p,r):=\big\{x\in\bR^d\,:\,|x|<r\big\}$, and $B_r:=B_r(0)$.
						
			\item A non-empty connected open set is called a domain.
			
			\item For sets $E,\,F\subset \bR^d$, $d(x,E):=\inf_{y\in E}|x-y|$ and $d(E,F):=\inf_{x\in E}d(x,F)$.
			For a fixed open set $\domain\subset \bR^d$, we usually denote $\rho(x):=d(x,\partial\domain)$ when there is no confusion.

			\item For a set $E\subset \bR^d$, $1_E$ denotes the function defined as $1_E(x)=1$ for $x\in E$, and $1_E(x)=0$ for $x\notin E$.
			For a function $f$ defined in $E$, $f1_E$ denotes the function defined as $\big(f1_E\big)(x)=f(x)$ if $x\in E$, and $\big(f1_E\big)(x)=0$ if $x\neq E$.

			\item  $\mathrm{supp}(f)$ denotes the support of the function $f$ defined as the closure of $\{x\,:\,f(x)\neq 0\}$.

			\item For an open set $\domain\subseteq\bR^d$, $C^{\infty}_c(\domain)$ is the space of infinitely differentiable functions $f$ for which $\mathrm{supp}(f)$ is a compact subset of $\domain$.
			Also, $C^{\infty}(\domain)$ denotes the the space of infinitely differentiable functions in $\domain$.
			
			\item For an open set $\cO\subseteq \bR^d$, $\cD'(\cO)$ denotes the set of all distributions on $\cO$, which is the dual of $C_c^{\infty}(\Omega)$.
			For $f\in \cD'(\cO)$, the expression $\la f,\varphi\ra $, $\varphi\in C^{\infty}_c(\domain)$ denote the evaluation of $f$ with the test function $\varphi$.

			\item  For  any multi-index $\alpha=(\alpha_1,\ldots,\alpha_d)$, $\alpha_i\in \{0\}\cup \bN$, we  denote $|\alpha|:=\sum_{i=1}^d \alpha_i$.
			For a function $f$ defined on an open set $\domain\subset \bR^d$,  $f_{x^i}:=D_if:=\frac{\partial f}{\partial x^i}$, and $D^{\alpha}f(x):=D^{\alpha_d}_d\cdots D^{\alpha_1}_1f(x)$.
			For the second order derivatives we denote $D_jD_if$ by $D_{ij}f$. We often use the notation 
			$|gf_x|$ for $\sum_{i=1}^d|g D_if|$, $|gf_{xx}|$ for $\sum_{i,j=1}^d|gD_{ij}f|$, and  $\big|gD^m f\big|$ for $\sum_{|\alpha|=k}|gD^\alpha f|$.
			We extend these notations to a sublinear function $\|\cdot\|:\cD'(\Omega)\rightarrow [0,+\infty]$; for example, $\|gf_x\|:=\sum_{i=1}^d\|g D_if\|$.

			\item $\Delta f:=\sum_{i=1}^d D_{ii}f$ denotes the Laplacian for a function $f$ defined on $\cO$.

			\item  For an open set $\domain\subseteq\bR^d$, $C(\domain)$ denotes the set of all continuous functions $f$ in $\domain$ such that $|f|_{C(\domain)}:=\sup_{\domain}|f|<\infty$.
			For $n\in\bN_0$, $C^n(\cO)$ denotes the set of all strongly $n$-times continuously differentiable function $f$ on $\cO$ such that
			$\|f\|_{C^n(\cO)}:=\sum_{k=0}^n|D^kf|_{C(\domain)}<\infty\,.$
			For $\alpha\in (0,1]$, $C^{n,\alpha}(\cO)$ denotes the set of all $f\in C^n(\cO)$ such that $	\|f\|_{C^{n,\alpha}(\cO)}:=\|f\|_{C^n(\cO)}+[f]_{C^{n,\alpha}(\domain)}<\infty$, where $[f]_{C^{n,\alpha}(\domain)}:=\sup_{x\neq y\in \domain} \frac{|D^nf(x)-D^nf(y)|}{|x-y|^{\alpha}}$.
			For any set $E\subset \bR^d$, we define the space $C^{0,\alpha}(E)$ in the same way.

			\item  Let $(A, \cA, \mu)$ be a measure space. For a a measurable function $f:A\rightarrow [-\infty,\infty]$, $\mathrm{ess\,sup}_A\,f$ is defined as the infimum of $a\in [-\infty,\infty]$ for which $\mu\big(\{x\in A\,:\,f(x)> a\}\big)=0$, and  $\mathrm{ess\,inf}_A\,f:=-\mathrm{ess\,sup}_A\,(-f)$.

			\item Let $\cO\subseteq \bR^d$ be an open set. For $p\in[1,\infty]$, $L_p(\domain)$ is the set of all measurable functions $f$ on $\domain$ such that $\|f\|_p:=\big(\int_\domain |f|^p\dd x\big)^{1/p}<\infty$ if $p<\infty$, and $\|f\|_{\infty}:=\mathrm{ess\,sup}_A\,|f|<\infty$ if $p=\infty$.
			For $n\in \bN_0$,  $W^n_p(\domain):=\{f: \sum_{|\alpha|\le n}\|D^\alpha f\|_p<\infty\}$, the Sobolev space.

			\item Let $\cO\subseteq \bR^d$ be an open set.
			For $X(\cO)=L_p(\cO)$ or $C^n(\cO)$ or $C^{n,\alpha}(\cO)$, $X_{\mathrm{loc}}(\cO)$ denotes the set of all function $f$ on $\cO$ such that $f\zeta\in X(\cO)$ for all $\zeta\in C_c^{\infty}(\cO)$.
			Especially, if $f\in L_{1,\mathrm{loc}}(\Omega)$, then $f$ is said to be locally integrable in $\Omega$.	
		
		\end{itemize}
				
		\mysection{Key estimates for the Poisson equation}\label{0030}

		This section is devoted to estimating the zeroth-order term of solutions to the Poisson equation \eqref{ellip} on a domain that admits the Hardy inequality \eqref{hardy}.
		In the main theorem, Theorem \ref{21.05.13.2}, superharmonic functions serve as weight functions.
		We begin with the definition and elementary properties of superharmonic functions.
		
		\begin{defn}\label{21.01.19.1}
			\,\,
			
			\begin{enumerate} 
				\item A function $\phi\in L_{1,\mathrm{loc}}(\Omega)$ is said to be \textit{superharmonic} if $\Delta \phi\leq 0$ in the sense of distributions on $\Omega$, \textit{i.e.}, for any nonnegative $\zeta\in C_c^{\infty}(\Omega)$, 
				$$
				\int_{\Omega}\phi\,\Delta \zeta\, \dd x\leq 0\,.
				$$
				
				\item A function $\phi:\Omega\rightarrow (-\infty,+\infty]$ is called a \textit{classical superharmonic function} if the following conditions are satisfied:
				\begin{enumerate}
					\item $\phi$ is lower semicontinuous in $\Omega$.
					\item For any $x\in \Omega$ and $r>0$ satisfying $\overline{B}_r(x)\subset \Omega$, 
					\begin{align*}
						\phi(x)\geq \frac{1}{\big|B_r(x)\big|}\int_{B_r(x)}\phi(y)\dd y\,.
					\end{align*}
					\item $\phi\not\equiv +\infty$ on every connected component of $\Omega$.
				\end{enumerate}
			\end{enumerate}
		\end{defn}
		
		Recall that $\phi$ is said to be $harmonic$ if both $\phi$ and $-\phi$ are classical superharmonic functions.
		
		\begin{remark}\label{240316310}
			Equivalent definitions of classical superharmonic functions are given in \cite[Definition 3.1.2, Theorem 3.2.2]{AG}.
			In particular, if $\phi$ is a classical superharmonic function on a neighborhoodof each $x\in\Omega$, then $\phi$ is a classical superharmonic function on $\Omega$.
		\end{remark}
	
	\begin{lemma}\label{240315329}\,
			A function $\phi:\Omega\rightarrow [-\infty,+\infty]$ is superharmonic if and only if there exists a classical superharmonic function $\phi_0$ in $\Omega$ such that $\phi=\phi_0$ almost everywhere in $\Omega$.
		\end{lemma}
		The proof of this lemma can be found in \cite[Theorem 4.3.2]{AG} and \cite[Proposition 30.6]{TF} for the `if' and `only if' directions, respectively.
			
		\begin{lemma}\label{21.04.23.3}
			Let $\phi$ be a classical superharmonic function on $\Omega$.
			
			\begin{enumerate}
				\item If $\phi$ is twice continuously differentiable, then $\Delta\phi\leq 0$.
				
				\item $\phi$ is locally integrable in $\Omega$.
				
				\item For any compact set $K\subset \Omega$, $\phi$ has the minimum value on $K$.
				
				\item For $\epsilon>0$, define
				\begin{align}\label{21.04.23.1}
					\phi^{(\epsilon)}(x)=\int_{B_1(0)}\big(\phi 1_{\Omega}\big)(x-\epsilon y)\cdot N_0\,\ee^{-1/(1-|y|^2)}\dd y\,,
				\end{align}
				where $N_0:=\big(\int_{B_1}\ee^{-1/(1-|y|^2)}\dd y\big)^{-1}$.
				Then for any compact set $K\subset \Omega$ and $0<\epsilon<d(K,\Omega^c)$, the following hold:
				\begin{enumerate}
					\item $\phi^{(\epsilon)}$ is infinitely differentiable in $\bR^d$.
					\item $\phi^{(\epsilon)}$ is a classical superharmonic function on $K^{\circ}$.
					\item For any $x\in K$, $\phi^{(\epsilon)}(x)\nearrow \phi(x)$ as $\epsilon\searrow 0$.
				\end{enumerate}
			\end{enumerate}
		\end{lemma}
		Regarding this lemma, (1)–(3) follow from Definition \ref{21.01.19.1} and Lemma \ref{240315329}, while (4) is given in \cite[Theorem 3.3.3]{AG}.
		
		\begin{lemma}\label{21.04.23.5}
			Let $\phi$ be a positive superharmonic function on $\Omega$ and let $\phi^{(\epsilon)}$ denote the function defined in \eqref{21.04.23.1}.
			\begin{enumerate}
				\item For any $c\leq 1$, $\phi^c$ is locally integrable in $\Omega$.
				
				\item If $f\in L_1(\Omega)$ and  $\mathrm{supp}(f)$ is a compact subset of $\Omega$, then for any $c\in\bR$,
				\begin{align}\label{220429449}
					\lim_{\epsilon\rightarrow 0}\int_{\Omega}|f|\big(\phi^{(\epsilon)}\big)^c\dd x= \int_{\Omega}|f|\phi^c\dd x\,.
				\end{align}
				
				\item If $f\in L_\infty(\Omega)$ and $\mathrm{supp}(f)$ is a compact subset of $\Omega$, then for any $c\leq 1$,
				\begin{align*}
					\lim_{\epsilon\rightarrow 0}\int_{\Omega}f\big(\phi^{(\epsilon)}\big)^c\dd x= \int_{\Omega}f\phi^c\dd x\,.
				\end{align*}
			\end{enumerate}
		\end{lemma}
		
		\begin{proof}
			(1) Let $K$ be a compact subset of $\Omega$.
			If $c\in(0,1]$, then by Lemma \ref{21.04.23.3}.(2),
			$$
			\int_K\phi^c\dd x\leq |K|^{1-c}\Big(\int_K\phi\dd x\Big)^c<\infty\,.
			$$
			In addition, if $c\leq 0$, then by Lemma \ref{21.04.23.3}.(3), $\max_K(\phi^c)= \big(\min_K\phi\big)^c<\infty$.
			
			(2) Take a bounded open set $U$ such that $\text{supp}(f)\subset U$ and $\overline{U}\subset\Omega$.
			Consider only $\epsilon\in (0,d(\mathrm{supp}(f),U^c)$.
			If $c\geq 0$, then \eqref{220429449} follows from Lemma \ref{21.04.23.3}.(4) and the monotone convergence theorem.			
			If $c<0$, then $|f|\big(\phi^{(\epsilon)}\big)^c\leq \big(\min_{\overline{U}}\phi\big)^c|f|$, and therefore \eqref{220429449} follows from the Lebesgue dominated convergence theorem.

			(3) Since $f\in L_\infty(\Omega)$, (1) of this lemma implies that $f\phi^c\in L_1(\Omega)$.
			The proof is completed by (2) of this lemma for $\max(f,0)$ and $\max(-f,0)$ instead of $f$.
		\end{proof}

		We now present the key lemma of this section.
		
		\begin{lemma}\label{03.30}
			Let $p\in(1,\infty)$ and $c<1$, and suppose that $u\in C(\Omega)$ satisfies the following conditions:
			\begin{equation}\label{22.01.25.2}
				\begin{gathered}
					\mathrm{supp}(u)\,\,\text{is a compact subset of}\,\,\,\Omega\,,\\
					u\in C_{\mathrm{loc}}^2\big(\{x\in\Omega\,:\,u(x)\neq 0\}\big)\,\,\,,\,\,\,\text{and}\quad \int_{\{u\neq 0\}}|u|^{p-1}|D^2u|\dd x<\infty\,.
				\end{gathered}
			\end{equation}
			Let $\phi$ be a positive superharmonic function on a neighborhood of $\mathrm{supp}(u)$.
			\begin{enumerate}
				\item If $\phi$ is twice continuously differentiable, then
				\begin{align}\label{230121211}
					\int_{\Omega}|u|^p\phi^{c-2}|\nabla \phi|^2 \dd x \leq \Big(\frac{p}{1-c}\Big)^2\int_{\Omega\cap\{u\neq 0\}}|u|^{p-2}|\nabla u|^2\phi^c \dd x\,.
				\end{align}
				
				\item If we additionally assume that $c\in(-p+1,1)$ and $(\Delta u)1_{\{u\neq 0\}}$ is bounded, then
				\begin{align}\label{22062422511}
					\int_{\Omega\cap\{u\neq 0\}}|u|^{p-2}|\nabla u|^2\phi^c \dd x\leq N\int_{\Omega\cap\{u\neq 0\}}(-\Delta u)\cdot u|u|^{p-2}\phi^c \dd x\,,
				\end{align}
				where $N=N(p,c)>0$.
			
				\item If the Hardy inequality \eqref{hardy} holds in $\Omega$, then
				\begin{align}\label{22062422512}
					\int_{\Omega}|u|^p\phi^c\rho^{-2}\dd x\leq N\int_{\Omega\cap\{u\neq 0\}}|u|^{p-2}|\nabla u|^2 \phi^c \dd x\,,
				\end{align}
				where $N=N(p,c,\mathrm{C}_0(\Omega))>0$.
			\end{enumerate}
		\end{lemma}
		
		Lemma \ref{03.30} is primarily used for $u\in C_c^{\infty}(\Omega)$.
		However, we employ condition \eqref{22.01.25.2} to establish Lemma \ref{21.05.25.3}, which is a crucial lemma for the existence of solutions in the main theorem (Theorem \ref{21.09.29.1}).
		To handle condition \eqref{22.01.25.2}, we prove the following results stated in Lemma \ref{21.04.23.4}: If $u\in C(\bR^d)$ satisfies \eqref{22.01.25.2}, then
		$|u|^{p/2-1}u\in W_2^1(\bR^d)$ and $|u|^p\in W_1^2(\bR^d)$, with
		\begin{align}\label{240123758}
			\begin{gathered}
		D_i(|u|^{p/2-1}u)=\frac{p}{2}|u|^{p/2-1}(D_iu)1_{\{u\neq 0\}}\,\,,\,\,\,\,  D_i\big(|u|^p\big)=p|u|^{p-2}uD_iu 1_{\{u\neq 0\}}\,\,,\\
		D_{ij}\big(|u|^p\big)=\big(p|u|^{p-2}uD_{ij}u+p(p-1)|u|^{p-2}D_iuD_ju\big)\,1_{\{u\neq 0\}}\,.
		\end{gathered}
		\end{align}
		
		\begin{proof}[Proof of Lemma \ref{03.30}]
			By Lemma \ref{240315329}, we may assume that $\phi$ is a classical superharmonic function on a neighborhood of $\text{supp}(u)$.
			In this proof, all of the integrations by parts are based on \eqref{240123758}.

			(1) Recall that $\phi$ is twice continuously differentiable in a neighborhood of $\mathrm{supp}(u)$.
			Integrate by parts to obtain
			\begin{align}
					&(1-c)\int_{\Omega}|u|^p\phi^{c-2}|\nabla \phi|^2\dd x=-\int_{\Omega}|u|^p\nabla\phi\cdot\nabla (\phi^{c-1})\,\dd x\nonumber\\
					=\,&p\int_{\Omega\cap\{u\neq 0\}}|u|^{p-2}u\,\phi^{c-1}(\nabla u\cdot\nabla \phi)\dd x+\int_{\Omega}|u|^p\phi^{c-1} \Delta \phi\, \dd x\label{220530132}\\
					\leq\,& p\,\Big(\int_{\Omega\cap\{u\neq 0\}} |u|^{p-2}|\nabla u|^2\phi^c\dd x\Big)^{1/2}\Big(\int_{\Omega} |u|^p\phi^{c-2}|\nabla \phi|^2\dd x \Big)^{1/2}\,,\nonumber
			\end{align}
			where the last inequality follows from the H\"older inequality and that $\Delta\phi\leq 0$ on $\{u\neq 0\}$.
			Since the first term of \eqref{220530132} is finite, we obtain \eqref{230121211}.
			The proof of (1) is completed.
			
			Although we do not assume that $\phi$ is infinitely smooth in (2) and (3),
			it suffices to consider the case where $\phi$ is smooth on its domain.
			This is because, if \eqref{22062422511} and \eqref{22062422512} hold for $\phi^{(\epsilon)}$ instead of $\phi$, for all sufficiently small $\epsilon>0$,
			then they also hold for $\phi$ by Lemma \ref{21.04.23.5}.
			Note that if $0<\epsilon<d\big(\mathrm{supp}(u),\partial\Omega\big)$, then $\phi^{(\epsilon)}$ is a positive superharmonic function on a neighborhood of $\mathrm{supp}(u)$ (see Lemma \ref{21.04.23.3}).
			In addition, $|u|^{p-2}|\nabla u|^21_{\{u\neq 0\}}$ and $|u|^p\rho^{-2}$ are integrable (see Lemma \ref{21.04.23.4}), and $-\Delta u\cdot u|u|^{p-2}1_{\{u\neq 0\}}$ in \eqref{22062422511} is bounded.
			Therefore, in the proof of (2) and (3), we may assume that $\phi$ is infinitely smooth.
			
			(2) \textbf{Case 1: $0\leq c<1$.}
			Integrate by parts to obtain
			\begin{align*}
				\int_{\Omega}-\Delta u\cdot u|u|^{p-2}\phi^c \dd x=\,& (p-1)\int_{\Omega\cap\{u\neq 0\}}|u|^{p-2}|\nabla u|^2\phi^c\dd x-\frac{1}{p}\int_{\Omega}|u|^p\Delta(\phi^c)\dd x\,.
			\end{align*}
			Since
			\begin{align*}
				\Delta (\phi^c)=c\,\phi^{c-1}\Delta\phi+c(c-1)\phi^{c-2}|\nabla\phi|^2\leq 0\qquad\text{on}\quad \text{supp}(u)\,,
			\end{align*}
			\eqref{22062422511} is obtained.
			
			\textbf{Case 2: $-p+1<c<0$.}
			By integration by parts, the H\"older inequality, and \eqref{230121211}, we have
			\begin{align*}
				&\int_{\Omega}-\Delta u\cdot u|u|^{p-2}\phi^c \dd x\\
				=\,&(p-1)\int_{\Omega}|u|^{p-2}|\nabla u|^2\phi^c\dd x+c\int_{\Omega}(\nabla u)\cdot(\nabla\phi) u|u|^{p-2}\phi^{c-1}\dd x\\
				\geq\,& (p-1)\int_{\Omega}|u|^{p-2}|\nabla u|^2\phi^c\dd x\\
				&+c\left(\int_{\Omega\cap\{u\neq 0\}}|u|^{p-2}|\nabla u|^2\phi^c\dd x\cdot\int_{\Omega}|u|^p\phi^{c-2}|\nabla \phi|^2\dd x\right)^{1/2}\\
				\geq\,& \frac{p+c-1}{1-c}\int_{\Omega}|u|^{p-2}|\nabla u|^2\phi^c\dd x\,.
			\end{align*}
			
			(3) Note that our assumption of the Hardy inequality \eqref{hardy} implies that the inequality in \eqref{hardy} also holds for $f\in W^1_2(\Omega)$ whose support is a compact subset of $\Omega$.

			Since $\phi$ is assumed to be positive and smooth on a neighborhood of $\text{supp}(u)$, 
			it follows from Lemma \ref{21.04.23.4} that $|u|^{p/2-1}u\phi^{c/2}$ belongs to $W^1_2(\Omega)$, and
			$$
			\nabla\big(|u|^{p/2-1}u\phi^{c/2}\big)=\frac{p}{2}|u|^{p/2-1}(\nabla u)1_{\{u\neq 0\}}\phi^{c/2}+\frac{c}{2}|u|^{p/2}\phi^{c/2-1}\nabla\phi\,.
			$$
			Therefore, by the Hardy inequality and \eqref{230121211}, we have
			\begin{align*}
				\int_{\Omega}\big||u|^{p/2-1}u\phi^{c/2}\big|^2\rho^{-2}\dd x\lesssim_{p,c}\,&\mathrm{C}_0(\Omega) \int_{\Omega}\Big(|u|^{p-2}|\nabla u|^2\phi^c1_{\{u\neq 0\}}+|u|^p\phi^{c-2}|\nabla\phi|^2\Big)\dd x\\
				\lesssim_{p,c}\,& \mathrm{C}_0(\Omega) \int_{\Omega\cap\{u\neq 0\}}|u|^{p-2}|\nabla u|^2\phi^c \dd x.
			\end{align*}
		\end{proof}

		\begin{thm}\label{21.05.13.2}
			Let $\Omega$ admit the Hardy inequality \eqref{hardy}.
			For any $p\in(1,\infty)$, $c\in (-p+1,1)$, and positive superharmonic function $\phi$ in $\Omega$, the following holds:
			If $u\in C(\Omega)$ satisfies \eqref{22.01.25.2} and $(\Delta u)1_{\{u\neq 0\}}$ is bounded, then for any $\lambda\geq 0$,
			\begin{align*}
				\int_{\Omega}|u|^{p}\phi^c\rho^{-2}\dd x\leq N\int_{\Omega}|\Delta u-\lambda u|^p\phi^c\rho^{2p-2}\dd x\,,
			\end{align*}
			where $N=N(p,c,\mathrm{C}_0(\Omega))$.
		\end{thm}
		\begin{proof}
			Since $\lambda\geq 0$, Lemma \ref{03.30} implies
			\begin{align}\label{22.04.18.1}
				\begin{split}
					\int_{\Omega}|u|^{p}\phi^c\rho^{-2}\dd x&\leq N\int_{\Omega}(-\Delta u)\cdot u |u|^{p-2}1_{\{u\neq 0\}}\phi^c\dd x\\
					&\leq N\int_{\Omega}(-\Delta u+\lambda u)\cdot u |u|^{p-2}1_{\{u\neq 0\}}\phi^c\dd x\,,
				\end{split}
			\end{align}
			where $N=N(p,c,\mathrm{C}_0(\Omega))>0$.
			Since $\phi^c\rho^{-2}$ is locally integrable in $\Omega$ (see Lemma \ref{21.04.23.5}.(1)), the first term in \eqref{22.04.18.1} is finite. 
			By the H\"older inequality, the proof is completed.
		\end{proof}

		\begin{lemma}[Existence of a very weak solution]\label{21.05.25.3}
			Suppose that \eqref{hardy} holds in $\Omega$.
			Then for any $\lambda\geq 0$ and $f\in C_c^{\infty}(\Omega)$, there exists a measurable function $u:\Omega\rightarrow \bR$ satisfying the following:
			\begin{enumerate}
				\item $u\in L_{1,\mathrm{loc}}(\Omega)$.
				
				\item $\Delta u-\lambda u=f$ in the sense of distributions on $\Omega$, \textit{i.e.}, for any $\zeta\in C_c^{\infty}(\Omega)$,
				\begin{align}\label{230328849}
					\int_{\Omega}u\big(\Delta \zeta -\lambda \zeta\big)\dd x=\int_{\Omega} f \zeta \dd x\,.
				\end{align}
				
				\item For any $p\in(1,\infty)$, $c\in(-p+1,1)$, and any positive superharmonic function $\phi$ in $\Omega$,
				\begin{align}\label{220613103}
					\begin{split}
						\int_{\Omega}|u|^p\phi^c\rho^{-2}\dd x\leq N\int_{\Omega}|f|^p\phi^c\rho^{2p-2}\dd x
					\end{split}
				\end{align}
				where $N=N(p,c,\mathrm{C}_0(\Omega))>0$.
			\end{enumerate}
		\end{lemma}

		\begin{proof}
			Take infinitely smooth and bounded open sets $\Omega_n$, $n\in\bN$, such that
			$$
			\text{supp}(f)\subset \Omega_1\,\,,\quad \overline{\Omega_n}\subset \Omega_{n+1}\,\,,\quad \bigcup_{n}\Omega_n=\Omega
			$$
			(see, \textit{e.g.}, \cite[Proposition 8.2.1]{DD_2008}).
			For arbitrary $h\in C_c^{\infty}(\Omega_1)$ and $n\in\bN$, by $R_{\lambda,n}h$ we denote the classical solution $H\in C^{\infty}(\overline{\Omega_n})$ of the equation
			$$
			\Delta H-\lambda H=h1_{\Omega_1}\quad \text{on}\,\,\Omega_n\quad;\quad H|_{\partial\Omega_n}\equiv 0\,.
			$$
			Note that $\overline{\Omega_n}$ is compactly contained in $\Omega$, $R_{\lambda,n}h\in C^{\infty}(\overline{\Omega_n})$, and $R_{\lambda,n}h|_{\partial\Omega_n}\equiv 0$.
			Therefore $\big(R_{\lambda,n}h\big)1_{\Omega_n}$ is continuous in $\Omega$ and satisfies \eqref{22.01.25.2}.
			By Theorem \ref{21.05.13.2}, for any $p\in (1,\infty)$, $c\in(-p+1,1)$, and positive superharmonic function $\phi$ in $\Omega$, we have
			\begin{align}\label{220610424}
				\int_{\Omega}\big|\big(R_{\lambda,n}h\big)1_{\Omega_n}\big|^{p}\phi^{c}\rho^{-2}\dd x\leq N(p,c,\mathrm{C}_0(\Omega)) \int_{\Omega}|h|^p\phi^{c}\rho^{2p-2}\dd x\,.
			\end{align}
			Note that $N$ in \eqref{220610424} is independent of $n$.
			
			Take $F\in C_c^{\infty}(\Omega_1)$ such that $F\geq |f|$, and put $f_1:=f-F$ and $f_2:=-F$ so that $f_1,\,f_2\leq 0$, and $f_1-f_2=f$.
			
			For $v_n:=\big(R_{\lambda,n}f_1\big)1_{\Omega_n}$, the maximum principle implies that
			$0\leq v_n\leq v_{n+1}$ in $\Omega$.
			We define $v(x):=\lim_{n\rightarrow \infty}v_n(x)$.
			By applying the monotone convergence theorem to \eqref{220610424} with $(h,\phi,p,c):=(f_1,1_{\Omega},2,0)$, we obtain that $\int_{\Omega}|v|^2\rho^{-2}\dd x\lesssim \int_{\Omega}|f_1|^2\rho^{2}\dd x$,
			which implies that $v\in L_{1,\mathrm{loc}}(\Omega)$.
			
			We next claim that for any $\zeta\in C_c^{\infty}(\Omega)$,
			\begin{align}\label{230125936}
				\int_{\Omega}v\big(\Delta \zeta -\lambda \zeta\big)\dd x=\int_{\Omega} f_1 \zeta \dd x\,.
			\end{align}
			Fix $\zeta\in C_c^{\infty}(\Omega)$, and take large enough $N\in\bN$ such that $\mathrm{supp}(\zeta)\subset \Omega_N$.
			Then for any $n\geq N$, the definition of $v_n=R_{\lambda,n}f_1$ implies that \eqref{230125936} holds for $v_n$ instead of $v$.
			Since $0\leq v_n\leq v$ and $v\in L_{1,\mathrm{loc}}(\Omega)$, the Lebesgue dominated convergence theorem yields \eqref{230125936}.
			
			By the same argument, $w:=\lim_{\substack{n\rightarrow \infty}}\big(R_{\lambda,n}f_2\big)1_{\Omega_n}$
			belongs to $L_{1,\mathrm{loc}}(\Omega)$, and \eqref{230125936} holds for $(w,f_2)$ instead of $(v,f_1)$.
			
			Put $u:=v-w=\lim_{n\rightarrow \infty}\big(R_{\lambda,n}f\big)1_{\Omega_n}$, where the limit exists almost everywhere in $\Omega$.
			Then $u\in L_{1,\mathrm{loc}}(\Omega)$, and $u$ satisfies \eqref{230328849}.
			In addition, by applying Fatou's lemma to \eqref{220610424} with $h:=f$, we obtain \eqref{220613103}.
		\end{proof}

		\mysection{Weighted Sobolev spaces and solvability of the Poisson equation}\label{0040}
		In this section, we focus on the Poisson equation
		$$
		\Delta u-\lambda u=f\quad(\lambda\geq 0)
		$$
		in an open set $\Omega\subset \bR^d$ admitting the Hardy inequality, within the framework of the weighted Sobolev space $\Psi H_{p,\theta}^{\gamma}(\Omega)$ introduced in Definition \ref{220610533}.
		It is worth noting that the zero Dirichlet condition ($u|_{\partial\Omega}=0$) is naturally encoded in $\Psi H_{p,\theta}^{\gamma}(\Omega)$, as $C_c^{\infty}(\Omega)$ is dense in $\Psi H_{p,\theta}^{\gamma}(\Omega)$ (see Lemma \ref{21.09.29.4}).
		
		This section is organized as follows:
		In Section \ref{0041}, we present the notions of \textit{Harnack function} and \textit{regular Harnack function}.
		Section \ref{0042} introduces the weighted Sobolev space $\Psi H_{p,\theta}^{\gamma}(\Omega)$. 
		In Section \ref{0043}, we prove the main theorem of this section (Theorem \ref{21.09.29.1}), using the results of Section \ref{0030} and extending the localization argument of \cite{Krylov1999-1} to the setting of $\Psi H_{p,\theta}^{\gamma}(\Omega)$.

		\subsection{Harnack function and regular Harnack function}\label{0041}
		
		\begin{defn}\label{21.10.14.1}\,
			
			\begin{enumerate}
				\item We call a measurable function $\psi:\Omega\rightarrow \bR_+$ a \textit{Harnack function}, if there exists a constant $C=:\mathrm{C}_1(\psi)>0$ such that
				\begin{align*}
					\underset{B(x,\rho(x)/2)}{\mathrm{ess\,sup}}\,\psi\leq C\underset{B(x,\rho(x)/2)}{\mathrm{ess\,inf}}\,\psi\quad\text{for all}\,\,x\in\Omega\,.
				\end{align*}
				
				\item We call a function $\Psi\in C^{\infty}(\Omega)$ a \textit{regular Harnack function}, if $\Psi>0$ and there exists a sequence of constants $\{C^{(k)}\}_{k\in\bN}=:\mathrm{C}_2(\Psi)$ such that for each $k\in\bN$,
				\begin{align*}
					|D^k\Psi|\leq C^{(k)}\,\rho^{-k}\Psi\quad\text{on}\quad\Omega\,.
				\end{align*}
				
				\item Let $\psi$ be a measurable function and $\Psi$ be a regular Harnack function on $\Omega$. We say that $\Psi$ is a \textit{regularization} of $\psi$, if there exists a constant $C=:\mathrm{C}_3(\psi,\Psi)>0$ such that
				$$
				C^{-1}\Psi\leq\psi\leq C\,\Psi\quad\text{almost everywhere in}\,\,\Omega. 
				$$
			\end{enumerate}
		\end{defn}

		A relation between the notions of Harnack functions and regular Harnack functions is proved in Lemma \ref{21.05.27.3}.
		
		\begin{remark}
			We introduced the notion of the  Harnack function to enable a localization argument (see Lemma \ref{21.05.13.8}). 
			Independently of this, an earlier work \cite{VM} investigated the relation between the boundary behavior of continuous Harnack functions and the quasihyperbolic distance.
		\end{remark}

		\begin{example}\label{21.05.18.2}\,
			
			\begin{enumerate}
				\item For any $E\subset \Omega^c$, the function $x\mapsto d(x,E)$ is a Harnack function on $\Omega$.
				Additionally, $\mathrm{C}_1\big(d(\,\cdot\,,E)\big)$ can be chosen to be $3$.
				
				\item Let $\Psi\in C^{\infty}(\Omega)$ satisfy $\Psi>0$ and $\Delta\Psi=-\Lambda\Psi$ for some constant $\Lambda\geq 0$.
				We claim that $\Psi$ is a regular Harnack function on $\Omega$, and $\mathrm{C}_2(\Psi)$ can be chosen to depend only on $d$.
				To observe this, for a fixed $x_0\in\Omega$, put
				$$
				u(t,x):=\ee^{-\Lambda \rho(x_0)^2 t}\Psi\big(x_0+\rho(x_0)x\big)
				$$
				so that $u_t=\Delta u$ on $\bR\times B_1(0)$.
				The interior estimates and the parabolic Harnack inequality (see, \textit{e.g.}, \cite[Theorem 2.3.9]{Krylov2008} and \cite[Theorem 7.10]{EvansPDE}, respectively) imply that for any $k\in\bR$,
				$$
				\rho(x_0)^k|D^k\Psi(x_0)|=|D^k_x u(0,0)|\lesssim_{k,d}\|u\|_{L_2((-1/4,0]\times B_{1/2}(0))}\lesssim_d u(1,0)\leq \Psi(x_0)\,.
				$$
				
				\item The multivariate Fa\'a di Bruno formula (see, \textit{e.g.}, \cite[Theorem 2.1]{FDB}) implies the following:
				\begin{itemize}
					\item[] Let $U\subset \bR^d$ and  $V\subset\bR$ be open sets and $f:U\rightarrow V$ and $l:V\rightarrow \bR$ be smooth functions. For any multi-index $\alpha$,
					\begin{align*}
						\big|D^{\alpha}(l\circ f)\big|\leq N(d,\alpha)\sum_{k=1}^{|\alpha|}\Big(\big|\big(D^kl\big)\circ f\big|\sum_{\substack{\beta_1+\ldots+\beta_k=\alpha\\|\beta_i|\geq 1}}\,\prod_{i=1}^k|D^{\beta_i}f|\Big)\,.
					\end{align*}
				\end{itemize}
				This inequality implies that for any regular Harnack function $\Psi$ in $\Omega$, and any $\sigma\in\bR$, $\Psi^\sigma$ is also a regular Harnack function on $\Omega$, and $\mathrm{C}_2(\Psi^\sigma)$ can be chosen to depend only on $d,\, \sigma,\,\mathrm{C}_2(\Psi)$.
				
				\item If $\Psi$ and $\Phi$ are regularizations of $\psi$ and $\phi$, respectively, then $\Psi\Phi$, $\Psi+\Phi$, and $\frac{\Psi\Phi}{\Psi+\Phi}$ are regularizations of $\psi\phi$, $\max (\psi,\phi)$, and $\min (\psi,\phi)$, respectively.
			\end{enumerate}
		\end{example}

		\begin{lemma}\label{21.11.16.1}
			A measurable function $\psi:\Omega\rightarrow \bR_+$ is a Harnack function if and only if there exist $r\in(0,1)$ and $N_{r}>0$ such that
			\begin{align*}
				\underset{B(x,r\rho(x))}{\mathrm{ess\,sup}}\,\psi\leq N_{r}\underset{B(x,r\rho(x))}{\mathrm{ess\,inf}}\,\psi\quad\text{for all}\,\,x\in\Omega.
			\end{align*}
			In this case, $\mathrm{C}_1(\psi)$ and $N_r$ depend only on each other and $r$.
		\end{lemma}
			\begin{proof}
			We only need to prove that for fixed constants $r_0,\,r\in(0,1)$ and $\widetilde{N}\geq 1$, if 
			\begin{align}\label{220530319}
				\begin{split}
					\text{if}\quad \underset{B(x,r_0\rho(x))}{\mathrm{ess\,sup}}\psi\leq\,& \widetilde{N}\underset{B(x,r_0\rho(x))}{\mathrm{ess\,inf}}\psi\quad \forall\,\,x\in\Omega\,,\\
					\text{then}\quad \underset{B(x,r\rho(x))}{\mathrm{ess\,sup}}\psi\leq\,& \widetilde{N}^{2K+1}\underset{B(x,r\rho(x))}{\mathrm{ess\,inf}}\psi\quad \forall\,\,x\in\Omega\,,
				\end{split}
			\end{align}
			where $K$ is the smallest integer such that $K\geq \frac{r}{(1-r)r_0}$.

			If $r\leq r_0$, then there is nothing to prove.
			Consider the case $r>r_0$.
			For $x\in\Omega$, we denote $B(x)=B\big(x,r_0\rho(x)\big)$.
			For fixed $x_0\in\Omega$ and $y\in \overline{B}\big(x_0,r\rho(x_0)\big)$, put $x_i:=(1-\frac{i}{K})x_0+\frac{i}{K}y$, $i=1,\,\ldots,\,M$.
			One can observe that $|x_{i-1}-x_i|\leq r_0\rho(x_i)$, and therefore $x_{i-1}\in B(x_i)$. This implies that $B(x_{i-1})\cap B(x_i)\neq \emptyset$, and hence
			\begin{align}\label{22.03.02.4}
				\underset{B(x_{i})}{\mathrm{ess\,sup}}\,\psi \leq \widetilde{N}\, \underset{B(x_{i})}{\mathrm{ess\,inf}}\,\psi \leq \widetilde{N}\, \underset{B(x_{i-1})\cap B(x_i)}{\mathrm{ess\,inf}}\,\psi \leq \widetilde{N}\,\underset{B(x_{i-1})}{\mathrm{ess\,sup}}\,\psi\,.
			\end{align}
			By applying \eqref{22.03.02.4} for $i=1,\,\ldots,\,K$, we obtain that $
			\mathrm{ess\,sup}_{B(y)}\,\psi\leq \widetilde{N}^K\mathrm{ess\,sup}_B(x)\,\psi\,.
			$
			Since $B(x_0,r\rho(x_0))$ is contained in a finite union of elements in $\big\{B(y)\,:\,y\in \overline{B}(x_0,r\rho(x_0))\big\}$, we have
			\begin{align*}
				\underset{B(x_0,r\rho(x_0))}{\mathrm{ess\,sup}}\,\psi\leq \widetilde{N}^K 		\underset{B(x)}{\mathrm{ess\,sup}}\,\psi=\widetilde{N}^K \underset{B(x_0,r_0\rho(x_0))}{\mathrm{ess\,sup}}\,\psi\,.
			\end{align*}
			The same argument implies that 
			$$
			\underset{B(x_0,r_0\rho(x_0))}{\mathrm{ess\,inf}}\,\psi\leq \widetilde{N}^k\underset{B(x_0,r\rho(x_0))}{\mathrm{ess\,inf}}\,\psi\,.
			$$
			Consequently, we have
			$$
		\underset{B(x_0,r\rho(x_0))}{\mathrm{ess\,sup}}\,\psi \leq \widetilde{N}^K \underset{B(x_0,r_0\rho(x_0))}{\mathrm{ess\,sup}}\,\psi\leq \widetilde{N}^{K+1} \underset{B(x_0,r_0\rho(x_0))}{\mathrm{ess\,inf}}\,\psi\leq\widetilde{N}^{2K+1} \underset{B(x_0,r\rho(x_0))}{\mathrm{ess\,inf}}\,\psi\,,
			$$
			where the second inequality is implied by the assumption in \eqref{220530319}.
			\end{proof}

		\begin{remark}\label{22.02.17.5}
			Let $\psi$ be a Harnack function on $\Omega$. Since $\psi\in L_{1,\mathrm{loc}}(\Omega)$, almost every point in $\Omega$ is a Lebesgue point of $\psi$.
			If $x\in\Omega$ is a Lebesgue point of $\psi$, then for any $r\in(0,1)$
			$$
			\underset{B(x,r\rho(x))}{\mathrm{ess\,inf}}\,\psi\leq \psi(x)\leq \underset{B(x,r\rho(x))}{\mathrm{ess\,sup}}\,\psi\,.
			$$
			By Lemma \ref{21.11.16.1}, we obtain that for almost every $x\in\Omega$ and for any $r\in(0,1)$, there exists $N_r>0$ depending only on $\mathrm{C}_1(\psi)$ and $r$ such that
			$$
			N_r^{-1}\underset{B(x,r\rho(x))}{\mathrm{ess\,sup}}\,\psi\leq \psi(x)\leq N_r\underset{B(x,r\rho(x))}{\mathrm{ess\,inf}}\,\psi\,.
			$$
		\end{remark}

		\begin{lemma}\label{21.05.27.3}
			\,\,
			
			\begin{enumerate}
				\item If $\psi$ is a Harnack function, then there exists a regularization of $\psi$.
				For such a regularization of $\psi$, denoted by $\widetilde{\psi}$, $\mathrm{C}_2(\widetilde{\psi})$ and $\mathrm{C}_3(\psi,\widetilde{\psi})$ can be chosen to depend only on $d$ and $\mathrm{C}_1(\psi)$.
				
				\item If $\Psi$ is a regular Harnack function, then it is also a Harnack function and $\mathrm{C}_1(\Psi)$ can be chosen to depend only on $d$ and $\mathrm{C}_2(\Psi)$.
			\end{enumerate}
		\end{lemma}
		This lemma implies that a measurable function is a Harnack function if and only if it has a regularization.
		\begin{proof}[Proof of Lemma \ref{21.05.27.3}]
			\,
			
			(1) Let $\psi$ be a Harnack function on $\Omega$.
			Take $\zeta\in C_c^{\infty}(B_1)$ such that $\zeta\geq 0$ and $\int_{B_1}\zeta dx=1$.
			For $i=1,\,2,\,3$ and $k\in\bZ$, put
			\begin{align*}
				U_{i,k}=\{x\in\Omega\,:\,2^{k-i}<\rho(x)<2^{k+i}\}\quad\text{and}\quad \zeta_k(x)=\frac{1}{2^{(k-4)d}}\zeta\Big(\frac{x}{2^{k-4}}\Big)\,.
			\end{align*}
			Note that for each $i$,
			\begin{align}\label{220604957}
				\text{$\big\{U_{i,k}\big\}_{k\in\bZ}$ is a locally finite cover of $\Omega$, and}\,\,\,\,\sum_{k\in\bZ}1_{U_{i,k}}\leq 2i\,.
			\end{align}
			For each $k\in\bZ$, put
			$$
			\Psi_k(x):=\big(\psi 1_{U_{2,k}}\big)\ast \zeta_k(x):=\int_{B(x,2^{k-4})}\big(\psi 1_{U_{2,k}}\big)(y)\zeta_k(x-y)\dd y\,,
			$$
			so that $\Psi_k\in C^{\infty}(\Omega)$.
			
			If $x\in U_{1,k}$, then $B(x,2^{k-4})\subset B(x,\rho(x)/2)\cap U_{2,k}$.
			Therefore we have
			\begin{align}\label{22.02.17.3}
				\Psi_k\geq \Big(\underset{B(x,\rho(x)/2)}{\mathrm{ess\,inf}}\psi\Big)1_{U_{1,k}}(x)\,.
			\end{align}
			If $x\in U_{3,k}$, then $B(x,2^{k-4})\subset B(x,\rho(x)/2)$, and if $x\notin U_{3,k}$, then $ B(x,2^{k-4})\cap U_{2,k}=\emptyset$.
			Therefore we have
			\begin{align}\label{22.02.17.4}
				\Psi_k(x)\leq \Big(\underset{B(x,\rho(x)/2)}{\mathrm{ess\,sup}}\psi\Big)1_{U_{3,k}}(x)\,.
			\end{align}
			By \eqref{22.02.17.3}, \eqref{22.02.17.4}, and Remark \ref{22.02.17.5}, we obtain that
			\begin{align}\label{2206292521}
				N^{-1}\psi(x)1_{U_{1,k}}(x)\leq \Psi_k(x)\leq N\psi(x)1_{U_{3,k}}(x)
			\end{align}
			for almost every $x\in\Omega$, where $N=N(\mathrm{C}_1(\psi))$.
			Moreover,
			\begin{equation}\label{2206292522}
			\begin{alignedat}{2}
					|D^{\alpha}\Psi_k(x)|\,&\leq \|D^{\alpha}\zeta_k\|_{\infty}\int_{B(x,2^{k-4})}\psi 1_{U_{2,k}}\dd y&&\\
					&\leq 2^{-|\alpha|k}\bigg(\underset{B(x,\rho(x)/2)}{\mathrm{ess\,sup}}\psi\bigg)1_{U_{3,k}}(x)&&\lesssim_N \rho(x)^{-|\alpha|}\psi(x)1_{U_{3,k}}(x)
			\end{alignedat}
		\end{equation}
			for almost every $x\in\Omega$, where $N=N(d,\alpha,\mathrm{C}_1(\psi))$.
			By \eqref{220604957}, \eqref{2206292521}, and \eqref{2206292522}, we obtain that  $\Psi:=\sum_{k\in\bZ}\Psi_k$ belongs to $C^{\infty}(\Omega)$, and
			\begin{align}\label{230328229}
				\Psi\simeq_{\mathrm{C}_1(\psi)}\psi\,\,\,\,\text{and}\,\,\,\, |D^\alpha\Psi|\leq\sum_{k\in\bZ}|D^{\alpha}\Psi_k|\lesssim_N\rho^{-|\alpha|}\psi
			\end{align}
			almost everywhere in $\Omega$, where $N=N(d,\alpha,\mathrm{C}_1(\psi))$.
			By \eqref{230328229}, the proof is completed.
			
			(2) Let $x,y\in\Omega$ satisfy $|x-y|<\rho(x)/2$.
			For $r\in[0,1]$, put $x_r=(1-r)x+ry$, so that 
			$x_r\in B\big(x,\rho(x)/2\big)$ and $\rho(x_r)\geq \rho(x)-|x-x_r|\geq|x-y|$.
			Then we have
			\begin{align*}
				\Psi(x_r)\,&\leq \Psi(x_0)+|x-y|\int_0^r\big|(\nabla\Psi)(x_{t})\big|\dd t\\
				&\leq \Psi(x_0)+N_0|x-y|\int_0^r\rho(x_{t})^{-1}\Psi(x_t)\dd t\leq \Psi(x_0)+N_0\int_0^r\Psi(x_t)\dd t\,,
			\end{align*}
			where $N_0=N(d,\mathrm{C}_2(\Psi))>0$.
			Applying Gr\"onwall's inequality, we obtain
			$$
			\Psi(y)=\Psi(x_1)\leq \ee^{N_0}\Psi(x_0)=\ee^{N_0}\Psi(x)\,.
			$$
			
			For any $x\in\Omega$, if $y\in B(x,\rho(x)/3)$, then $|x-y|<\min\big(\rho(x),\rho(y)\big)/2$.
			Therefore we have
			$$
			\ee^{-N_0}\underset{B(x,\rho(x)/3)}{\mathrm{ess\,sup}}\Psi(y)\leq \Psi(x)\leq \ee^{N_0}\underset{B(x,\rho(x)/3)}{\mathrm{ess\,inf}}\Psi(y)\,,
			$$
			and by Lemma \ref{21.11.16.1}, the proof is completed.
		\end{proof}

		\subsection{Weighted Sobolev spaces and regular Harnack functions}\label{0042}
		In this subsection, we introduce the Krylov-type weighted Sobolev space $H_{p,\theta}^\gamma(\Omega)$, and generalize it through the use of regular Harnack functions.
		We denote the resulting generalized weighted Sobolev space by $\Psi H_{p,\theta}^\gamma(\Omega)$.
		
		The space $H_{p,\theta}^{\gamma}(\Omega)$ was first introduced by Krylov~\cite{Krylov1999-0,Krylov1999-1,Krylov2001} for $\Omega=\bR^d_+$, thereby initiating the $L_p$-theory for elliptic and parabolic equations in $H_{p,\theta}^{\gamma}(\bR^d_+)$.
		This development was motivated by stochastic parabolic equations, aiming to control the boundary behavior of derivatives of solutions, as discussed in detail in \cite{Krylov1994, Krylov1999-1}.
		Subsequently, Lototsky \cite{Lo0,Lo1} systematically extended the framework of $H_{p,\theta}^{\gamma}(\Omega)$ to general domains $\Omega$.
							
		We first recall the definition of the Bessel potential space on $\bR^d$.
		For $p\in(1,\infty)$ and $\gamma\in\bR$, $H_p^{\gamma}=H_p^{\gamma}(\bR^d)$ denotes the space of Bessel potentials with the norm
		\begin{align*}
			\|f\|_{H_p^{\gamma}}:=\big\|(1-\Delta)^{\gamma/2}f\big\|_{L_p(\bR^d)}:=\Big\|\cF^{-1}\big[(1+|\xi|^2)^{\gamma/2}\cF(f)(\xi)\big]\Big\|_p\,,
		\end{align*}
		where $\cF$ is the Fourier transform and $\cF^{-1}$ is the inverse Fourier transform.
		If $\gamma\in\bN_0$, then $H_p^{\gamma}$ coincides with the Sobolev space 
		\begin{align*}
			W_p^{\gamma}(\bR^d):=\bigg\{f\in\cD'(\bR^d)\,:\,\sum_{k=0}^\gamma\int_{\bR^d}|D^kf|^p\dd x<\infty\bigg\}\,.
		\end{align*}
		
		We next recall $H_{p,\theta}^{\gamma}(\Omega)$ and introduce its generalization $\Psi H_{p,\theta}^{\gamma}(\Omega)$.
		It is worth mentioning in advance that for $\gamma\in\bN_0$, the space $\Psi H_{p,\theta}^{\gamma}(\Omega)$ coincides with
		\begin{align*}
			\bigg\{f\in\cD'(\Omega)\,:\,\sum_{k=0}^{\gamma}\int_{\Omega}\big|\rho^kD^kf\big|^p\Psi^p\rho^{\theta-d}\dd x<\infty\bigg\}\,,
		\end{align*}
		where $\rho(x):=d(x,\partial\Omega)$ (see Lemma \ref{220512433}).
		In the remainder of this subsection, we assume that
		\begin{align*}
			p\in(1,\infty)\,,\,\,\,\gamma,\,\theta\in\bR\,,\,\,\,\text{$\Psi$ is a regular Harnack function on $\Omega$}\,.
		\end{align*}
		By $\trho$ we denote the regularization of $\rho(\,\cdot\,):=d(\,\cdot\,,\partial\Omega)$ constructed in Lemma \ref{21.05.27.3}.(1).
		Recall that for each $k\in\bN_0$, there exists a constant $N_k=N(d,k)>0$ such that
		\begin{align*}
			\trho\simeq_{N_0}\rho\quad\text{and}\quad|D^k\trho\,|\leq N_k\rho^{\,1-k}\quad \text{on}\quad \Omega\,.
		\end{align*}
		To define the weighted Sobolev spaces, fix $\zeta_0\in C_c^{\infty}(\bR_+)$ such that
		\begin{align*}
			\text{supp}(\zeta_0)\subset [\ee^{-1},\ee]\quad,\quad\zeta_0\geq 0\quad,\quad \sum_{n\in\bZ}\zeta_0(\ee^{n}t)=1\quad\text{for all}\,\,t\in\bR_+\,.
		\end{align*}
		For $x\in\bR^d$ and $n\in\bZ$, put
		\begin{align}\label{230130543}
			\zeta_{0,(n)}(x):=\zeta_0\big(\ee^{-n}\trho(x)\big)1_{\Omega}(x)\,,
		\end{align}
		so that
		\begin{align}\label{230130542}
			\begin{split}
				&\sum_{n\in\bZ}\zeta_{0,(n)}\equiv 1\quad\text{on}\,\,\Omega\,\,,\quad \text{supp}(\zeta_{0,(n)})\subset \{x\in\Omega\,:\,\ee^{n-1}\leq \trho(x)\leq \ee^{n+1}\}\,\,,\\
				&\qquad\qquad  \zeta_{0,(n)}\in C^{\infty}(\bR^d)\,\,,\,\,\,\,\text{and}\quad |D^{\alpha}\zeta_{0,(n)}|\leq N(d,\alpha,\zeta)\,\ee^{-n|\alpha|}\,.
			\end{split}
		\end{align}

		\begin{defn}\label{220610533}\,\,
			
			\begin{enumerate}
				\item By $H_{p,\theta}^{\gamma}(\Omega)$ we denote the class of all distributions $f\in\cD'(\Omega)$ such that
				\begin{align*}
					\|f\|^p_{H^{\gamma}_{p,\theta}(\Omega)}:=\sum_{n\in\bZ}\ee^{n\theta}\big\|\big(\zeta_{0,(n)}f\big)(\ee^n\cdot)\big\|_{H^{\gamma}_p(\bR^d)}^p<\infty\,.
				\end{align*}
				
				\item By $\Psi H_{p,\theta}^{\gamma}(\Omega)$ we denote the class of all distributions $f\in\cD'(\Omega)$ such that $f=\Psi g$ for some $g\in H_{p,\theta}^{\gamma}(\Omega)$. 
				The norm in $\Psi H_{p,\theta}^{\gamma}(\Omega)$ is defined by
				$$
				\|f\|_{\Psi H_{p,\theta}^{\gamma}(\Omega)}:=\|\Psi^{-1}f\|_{H_{p,\theta}^{\gamma}(\Omega)}\,.
				$$
			\end{enumerate}
		\end{defn}
		
		We also denote
		$$
		L_{p,\theta}(\Omega):=H_{p,\theta}^{0}(\Omega)\quad\text{and}\quad \Psi L_{p,\theta}(\Omega):=\Psi H_{p,\theta}^{0}(\Omega)\,.
		$$

		In the rest of this subsection, we collect properties of $H_{p,\theta}^\gamma(\Omega)$ and $\Psi H_{p,\theta}^\gamma(\Omega)$.
		As $\Psi H_{p,\theta}^\gamma(\Omega)$ is a variant of $H_{p,\theta}^\gamma(\Omega)$, we derive properties of $\Psi H_{p,\theta}^\gamma(\Omega)$ based on those of $H_{p,\theta}^\gamma(\Omega)$.
		Note that we cite the properties of $H_{p,\theta}^\gamma(\Omega)$ from \cite{Lo1} as refined versions.
		Specifically, in Lemma \ref{21.05.20.3} and the proof of Lemma \ref{21.09.29.4}.(2), the constants in their estimates are independent of $\Omega$.
		The validity of these refined estimates is supported by the proof in \cite{Lo1}, with complete details given in \cite[Appendix A.1]{Seo202304}.
		
		The spaces $H_{p,\theta}^{\gamma}(\Omega)$ and $\Psi H_{p,\theta}^{\gamma}(\Omega)$ are independent of the choice of $\zeta_0$ (see Lemma \ref{21.05.20.3}.(2)).
		Therefore, we ignore the dependence on $\zeta_0$.
		We denote
		\begin{align*}
			\cI=\{d,\,p,\,\gamma,\,\theta\}\quad\text{and}\quad \cI'=\{d,\,p,\,\gamma,\,\theta,\,\mathrm{C}_2(\Psi)\}\,,
		\end{align*}
		where $\mathrm{C}_2(\Psi)$ is the sequence of constants in Definition \ref{21.10.14.1}.(2).

		\begin{lemma}[see \cite{Lo1} or Proposition A.3 in \cite{Seo202304}]\label{21.05.20.3}\,\,
			
			\begin{enumerate}
				\item For any $s<\gamma$,
				$$
				\|f\|_{H_{p,\theta}^{s}(\Omega)}\lesssim_{\cI,s}\|f\|_{H_{p,\theta}^{\gamma}(\Omega)}\,.
				$$
				
				\item For any $\eta\in C_c^{\infty}(\bR_+)$,
				\begin{align*}
					\sum_{n\in\bZ}\ee^{n\theta}\|\eta\big(\ee^{-n}\trho(\ee^n\cdot)\big)f(\ee^n\cdot)\|^p_{H^{\gamma}_{p}}\lesssim_{\cI,\eta}\|f\|_{H^{\gamma}_{p,\theta}(\Omega)}^p.
				\end{align*}
				If $\eta$ additionally satisfies
				\begin{align*}
					\inf_{t\in\bR_+}\bigg[\sum_{n\in\bZ}\eta(\ee^nt)\bigg]>0\,,
				\end{align*}
				then
				\begin{align*}
					\|f\|_{H^{\gamma}_{p,\theta}(\Omega)}^p\lesssim_{\cI,\eta} \sum_{n\in\bZ}\ee^{n\theta}\big\|\eta\big(\ee^{-n}\trho(\ee^n\cdot)\big)f(\ee^n\cdot)\big\|^p_{H^{\gamma}_{p}}\,.
				\end{align*}
				
				\item For any $s\in\bR$,
				$$
				\|\trho^{\,s} f\|_{H^{\gamma}_{p,\theta}(\Omega)}\simeq_{\cI,s} \|f\|_{H^{\gamma}_{p,\theta+sp}(\Omega)}\,.
				$$
				
				\item For any multi-index $k\in\bN$,
				\begin{align}\label{240121500}
				\|f\|_{H^{\gamma}_{p,\theta}(\Omega)}\simeq_{\cI,k}\sum_{i=0}^k\|D^if\|_{H^{\gamma-k}_{p,\theta+ip}(\Omega)}.
				\end{align}
				In particular, $\|D^kf\|_{H^{\gamma-k}_{p,\theta+kp}(\Omega)}\lesssim_{\cI,k} \|f\|_{H^{\gamma}_{p,\theta}(\Omega)}$.
				
				\item Let $k\in\bN_0$ such that $|\gamma|\leq k$. If $a\in C^k_{\mathrm{loc}}(\Omega)$ satisfies
				\begin{align*}
					|a|_{k}^{(0)}:=\sup_{\Omega}\sum_{|\alpha|\leq k}\rho^{|\alpha|}|D^{\alpha}a|<\infty\,\,,
				\end{align*}
				then
				$$
				\|af\|_{H^{\gamma}_{p,\theta}(\Omega)}\lesssim_{\cI}|a|_{k}^{(0)}\|f\|_{H^{\gamma}_{p,\theta}(\Omega)}.
				$$
			\end{enumerate}
		\end{lemma}

		\begin{remark}
			Lemma \ref{21.05.20.3} also holds if $f$ is replaced by $\Psi^{-1}f$.
			Therefore, all of the assertions in Lemma \ref{21.05.20.3}, except Lemma \ref{21.05.20.3}.(4), remain valid when $H_{\ast,\ast}^{\ast}(\Omega)$ is replaced by $\Psi H_{\ast,\ast}^{\ast}(\Omega)$.
		\end{remark}

		\begin{lemma}\label{21.09.29.4}\,\,
			
			\begin{enumerate}
				\item $C_c^{\infty}(\Omega)$ is dense in $\Psi H_{p,\theta}^{\gamma}(\Omega)$.
				
				\item  $\Psi H_{p,\theta}^{\gamma}$ is a reflexive Banach space with the dual $\Psi^{-1}H_{p',\theta'}^{-\gamma}(\Omega)$, where
				\begin{align}\label{240121447}
				\frac{1}{p}+\frac{1}{p'}=1\quad\text{and}\quad \frac{\theta}{p}+\frac{\theta'}{p'}=d\,.
				\end{align}
				Moreover, for any $f\in\cD'(\Omega)$, we have
				$$
				\|f\|_{\Psi H_{p,\theta}^{\gamma}(\Omega)}\simeq_{\cI'}\sup_{g\in C_c^{\infty}(\Omega),g\neq 0}\frac{|\langle f,g\rangle| }{\|g\|_{\Psi^{-1}H_{p',\theta'}^{-\gamma}(\Omega)}}\,.
				$$
				
				\item For any $k,\,l\in\bN_0$,
				$$
				\|\big(D^k\Psi\big) D^lf\|_{H^{\gamma}_{p,\theta}(\Omega)}\leq_{\cI',l,k}\|\Psi f\|_{H^{\gamma+l}_{p,\theta-(k+l)p}(\Omega)}
				$$
				
				\item Let $\Phi$ be a regular Harnack function on $\Omega$, and suppose that there exists a constant $N_0>0$ such that $\Psi\leq N_0\Phi$ in $\Omega$.
				Then,
				$$
				\|\Psi f\|_{H^{\gamma}_{p,\theta}(\Omega)}\leq N\|\Phi f\|_{H^{\gamma}_{p,\theta}(\Omega)}.
				$$
				where $N=N(\cI',\mathrm{C}_2(\Phi),N_0)$.
					
				\item Let $p'\in(1,\infty)$ and $\gamma',\,\theta'\in\bR$, and let $\Psi'$ be a regular Harnack function on $\Omega$, 
				if $f\in \Psi H_{p,\theta}^{\gamma}(\Omega)\cap \Psi'H_{p',\theta'}^{\gamma'}(\Omega)$, then there exists $\{f_n\}_{n\in\bN}\subset  C_c^{\infty}(\Omega)$ such that
				$$
				\|f-f_n\|_{\Psi H_{p,\theta}^{\gamma}(\Omega)}+\|f-f_n\|_{\Psi' H_{p',\theta'}^{\gamma'}(\Omega)}\rightarrow 0\quad\text{as}\,\,\,n\rightarrow\infty\,.
				$$
			\end{enumerate}
			
		\end{lemma}
		\begin{proof} 
			(1), (2) When $\Psi\equiv 1$, the results can be found in \cite{Lo1} (see also \cite[Proposition A.2]{Seo202304}).
			Since the map $f\mapsto \Psi^{-1}f$ is an isometric isomorphism from $\Psi H_{p,\theta}^{\gamma}(\Omega)$ to $H_{p,\theta}^{\gamma}(\Omega)$, the assertion follows immediately.
			
			(3) Since $\Psi$ and $\trho$ are regular Harnack functions, we obtain that for any $k,\,m\in\bN_0$,
			$$
			\Big|\frac{D^k\Psi}{\trho^{\,-k}\Psi}\Big|^{(0)}_m\leq N(d,k,m,\mathrm{C}_2(\Psi))\,.
			$$
			By Lemmas \ref{21.05.20.3}.(5) and (3), we have
			\begin{align}\label{21.11.16.2}
				\|\big(D^k\Psi\big) f\|_{H^{\gamma}_{p,\theta}(\Omega)}\lesssim_{\cI',k}\|\trho^{\,-k}\Psi f\|_{H^{\gamma}_{p,\theta}(\Omega)}\lesssim_{\cI',k}\|\Psi f\|_{H^{\gamma}_{p,\theta-kp}(\Omega)}\,.
			\end{align}
			Therefore, it suffices to prove that for any $l\in\bN$,
			$$
			\|\Psi D^lf\|_{H^{\gamma}_{p,\theta}(\Omega)}\lesssim_{\cI',l}\|\Psi f\|_{H^{\gamma+l}_{p,\theta-lp}(\Omega)}\,.
			$$
			Recall that $\Psi^{-1}$ is a regular Harnack function, and $\mathrm{C}_2(\Psi^{-1})$ can be chosen to depend only on $\mathrm{C}_2(\Psi)$ and $d$.
			It follows from Leibniz's rule, \eqref{21.11.16.2}, and Lemma \ref{21.05.20.3}.(4) and (1) that
			\begin{alignat*}{3}
				\|\Psi D^l(\Psi^{-1}\Psi f)\|_{H^{\gamma}_{p,\theta}(\Omega)}&\lesssim_{d,l}\,&&\sum_{n=0}^l\|\Psi D^{l-n}(\Psi^{-1})\cdot D^n(\Psi f)\|_{H^{\gamma}_{p,\theta}(\Omega)}\\
				&\lesssim_N &&\sum_{n=0}^l\| D^n(\Psi f)\|_{H^{\gamma}_{p,\theta-(l-n)p}(\Omega)}\lesssim_N\|\Psi f\|_{H^{\gamma+l}_{p,\theta-lp}(\Omega)}\,.
			\end{alignat*}
			
			(4) For any $k\in\bN_0$, 
			$$
			|\Psi\Phi^{-1}|^{(0)}_k\leq N(d,k,\mathrm{C}_2(\Psi),\mathrm{C}_2(\Phi), N_0)\,.
			$$
			Therefore, it follows from Lemma \ref{21.05.20.3}.(5) that
			$$
			\|\Psi f\|_{H_{p,\theta}^{\gamma}(\Omega)}=\|\Psi \Phi^{-1}(\Phi f)\|_{H_{p,\theta}^{\gamma}(\Omega)}\lesssim_N \|\Phi f\|_{H_{p,\theta}^{\gamma}(\Omega)}\,.
			$$
			
			(5) It directly follows from Lemma \ref{22.04.11.3}.
		\end{proof}

		\begin{remark}\label{21.10.05.2}
			It follows from Lemma \ref{21.09.29.4}.(4) that for regular Harnack functions $\Psi$ and $\Phi$, if 
			$N^{-1}\Phi\leq \Psi\leq N\Phi$
			for some constant $N>0$,
			then $\Psi H_{p,\theta}^{\gamma}(\Omega)$ coincides with $\Phi H_{p,\theta}^{\gamma}(\Omega)$.
			Therefore, applying Lemma \ref{21.05.20.3}.(3), we obtain that if $\Psi$ is a regularization of $\rho^\sigma$ ($\sigma\in\bR$), then $\Psi H_{p,\theta}^{\gamma}(\Omega)= H_{p,\theta-\sigma p}^{\gamma}(\Omega)$.
		\end{remark}

		\begin{lemma}\label{220512433}
			Let $f\in\cD'(\Omega)$.
			
			\begin{enumerate} 
				\item If $\gamma\in\bN_0$, then 
				\begin{align*}
					\|f\|^p_{\Psi H_{p,\theta}^\gamma(\Omega)}\simeq_{\cI'}  \sum_{|\alpha|\leq \gamma}\int_{\Omega}\big|\rho^{|\alpha|}D^\alpha f\big|^p\Psi^{-p}\rho^{\theta-d}\dd x\,.
				\end{align*}
				
				\item For any $k\in\bN$, 
				\begin{align}\label{240121445}
					\|f\|_{\Psi H_{p,\theta}^{\gamma}(\Omega)}\simeq_{\cI',k} \inf\bigg\{\sum_{|\alpha|\leq k}\|f_{\alpha}\|_{\Psi H_{p,\theta-|\alpha|p}^{\gamma+k}(\Omega)}:\,f=\sum_{|\alpha|\leq k}D^{\alpha}f_{\alpha}\bigg\}\,.
				\end{align}
				In particular, when $\gamma=-\bN$, we have
				\begin{align*}
					\|f\|_{\Psi H_{p,\theta}^{\gamma}(\Omega)}\simeq_{\cI'} \inf\bigg\{\sum_{|\alpha|\leq -\gamma}\|f_{\alpha}\|_{\Psi L_{p,\theta-|\alpha|p}(\Omega)}:\,f=\sum_{|\alpha|\leq -\gamma}D^{\alpha}f_{\alpha}\bigg\}\,.
				\end{align*}					
			\end{enumerate}

		\end{lemma}
		\begin{proof}
			(1) By \eqref{240121500}, it suffices to prove the case of $\gamma=0$. 
			This case is proved by the following:
			\begin{align*}
			\|f\|_{\Psi L_{p,\theta}(\Omega)}^p:=&\sum_{n\in\bZ}\ee^{n\theta}\int_\Omega\big|\big(\zeta_{0,(n)}\Psi^{-1} f\big)(\ee^{n}x)\big|^p\dd x\\
			=&\int_\Omega\bigg(\sum_{n\in\bZ}\ee^{n(\theta-d)}|\zeta_{0,(n)}|^p\bigg)|f|^p\Psi^{-p}\dd x\simeq_{d,p,\theta} \int_\Omega\rho^{\theta-d}|f|^p\Psi^{-p}\dd x\,,
			\end{align*}
			where the last similarity is implied by properties of $\zeta_{0,(n)}$ (see \eqref{230130543} and \eqref{230130542}).
			
			(2)	Repeatedly applying Lemma \ref{22.02.16.1}, we obtain $\{f_{\alpha}\}_{|\alpha|\leq k}\subset \cD'(\Omega)$ such that
			\begin{align*}
				f=\sum_{|\alpha|\leq k}D^{\alpha}f_\alpha\quad \text{and}\quad \sum_{|\alpha|\leq k}\|f_\alpha\|_{\Psi H_{p,\theta-|\alpha|p}^{\gamma+k}(\Omega)}\lesssim_{\cI',k}\|f\|_{\Psi H_{p,\theta}^{\gamma}(\Omega)}\,.
			\end{align*}
			Therefore we obtain \eqref{240121445} where `$\simeq_{\cI,k}$' is replaced by `$\gtrsim_{\cI,k}$'.
			
			For the inverse inequality, let $f=\sum_{|\alpha|\leq n}D^{\alpha}f_{\alpha}$ where $f_{\alpha}\in H^{\gamma+n}_{p,\theta-|\alpha|p}(\Omega)$.
			It follows from Lemma \ref{21.09.29.4}.(2) and Lemmas \ref{21.09.29.4}.(3) and \ref{21.05.20.3}.(1) that for any $g\in C_c^{\infty}(\Omega)$,
			\begin{alignat*}{2}
				|\langle f,g\rangle |\,&= &&\bigg|\sum_{|\alpha|\leq n}\big\langle \Psi^{-1}f_{\alpha},\Psi D^{\alpha}g\rangle\bigg|\\
				&\lesssim_{\cI',n} &&\sum_{|\alpha|\leq n}\Big(\|\Psi^{-1}f_{\alpha}\|_{H_{\theta-|\alpha|p}^{\gamma+n}(\Omega)}\|\Psi D^{\alpha}g\|_{H_{p',\theta'+|\alpha|p'}^{-\gamma-n}(\Omega)}\Big)\\
				&\lesssim_{\cI',n} \Big(&&\sum_{|\alpha|\leq n}\|\Psi^{-1}f_{\alpha}\|_{H_{\theta-|\alpha|p}^{\gamma+n}(\Omega)}\Big)\|\Psi g\|_{H^{-\gamma}_{p',\theta'}(\Omega)}\,,
			\end{alignat*}
			where $p'$ and $\theta'$ are constants in \eqref{240121447}.
			By applying Lemma \ref{21.09.29.4}.(2), we have
			$$
			\| f\|_{\Psi H_{p,\theta}^{\gamma}(\Omega)}\lesssim_{\cI',n}\inf\Big\{\sum_{|\alpha|\leq n}\|  f_{\alpha}\|_{\Psi H^{\gamma+n}_{p,\theta-|\alpha|p}(\Omega)}:\,f=\sum_{|\alpha|\leq n}D^{\alpha}f_{\alpha}\Big\}\,.
			$$
			This completes the proof.
		\end{proof}
		
			We end this subsection with a Sobolev-H\"older embedding theorem for the spaces $\Psi H_{p,\theta}^{\gamma}(\Omega)$.
			For $k\in\bN_0$, $\alpha\in(0,1]$, and $\delta\in\bR$, we define the weighted H\"older norm
		\begin{align*}
			|f|^{(\delta)}_{k,\alpha}:=\sum_{i=0}^k\sup_{\Omega}\left|\rho^{\delta+i}D^if\right|+\sup_{\substack{x,y\in\Omega\\x\neq y}}\frac{\left|\big(\trho^{\delta+k+\alpha} D^kf\big)(x)-\big(\trho^{\delta+k+\alpha} D^kf\big)(y)\right|}{|x-y|^{\alpha}}\,.
		\end{align*}
		
		\begin{prop}\label{220512537}
			Let $k\in\bN_0$ and $\alpha\in(0,1]$.
			\begin{enumerate}
				\item For any $\delta\in\bR$,
				\begin{align*}
					\left|\Psi^{-1} f\right|^{(\delta)}_{k,\alpha}\simeq_{N} &\,\sum_{i=0}^k\sup_{x\in\Omega}\left|\Psi(x)^{-1}\rho(x)^{\delta+i} D^if(x)\right|\\
					&+\sup_{x\in\Omega}\bigg(\Psi^{-1}(x)\rho^{\delta+k+\alpha}(x)\sup_{y:0<|x-y|<\frac{\rho(x)}{2}}\frac{\left|D^kf(x)-D^kf(y)\right|}{|x-y|^{\alpha}}\bigg)\,,
				\end{align*}
				where $N=N(d,k,\alpha,\delta,\mathrm{C}_2(\Psi))$.
			
				\item
				If $\alpha\in(0,1)$ and $k+\alpha\leq \gamma-d/p$,
				then for any $f\in \Psi H^{\gamma}_{p,\theta}(\Omega)$,
				\begin{align*}
					\left|\Psi^{-1} f\right|^{(\theta/p)}_{k,\alpha}\lesssim_{\cI',k,\alpha} \|f\|_{\Psi H^{\gamma}_{p,\theta}(\Omega)}\,.
				\end{align*}
			\end{enumerate}
		\end{prop}
		\begin{proof}
			(1)	This result follows from a direct calculation and the definition of regular Harnack functions.			
			Hence, we omit the proof.
			
			(2)	We only need to prove the statement for $\Psi\equiv 1$, and the result for this case is covered in \cite[Theorem 4.3]{Lo1}. 
			We include a proof here for the reader's convenience.
			
			For $f\in H_{p,\theta}^{\gamma}(\Omega)$, the Sobolev embedding theorem implies
			\begin{align}\label{220614409}
				\left\|\big(f\zeta_{0,(n)}\big)(\ee^n\,\cdot\,)\right\|_{C^{k,\alpha}}\leq N\left\|\big(f\zeta_{0,(n)}\big)(\ee^n\,\cdot\,)\right\|_{H_p^{\gamma}}<\infty\,,
			\end{align}
			where $N=N(d,p,\gamma,k,\alpha)$.
			Hence $f$ belongs to $C^k_{\mathrm{loc}}(\Omega)$.
			For $x\in\Omega$, take $n_0\in\bZ$ such that $\ee^{n_0-1}\leq \rho(x)\leq \ee^{n_0}$.
			If $|x-y|<\frac{\rho(x)}{2}$, then $\ee^{n_0-2}\leq \rho(y)\leq \ee^{n_0+2}$.
			Take constants $A$ and $B$ depending only on $d$ such that $A^{-1}\rho\leq \trho \leq A\rho$, and $\sum_{|n|\leq B}\zeta_0\big(\ee^nt\big)\equiv 1$ for all $\frac{1}{A\ee^2}\leq t\leq A\ee^2$.
			Then we have
			$$
			\sum_{|n-n_0|\leq B}\zeta_{0,(n)}\equiv 1\quad \text{on}\quad U_{n_0}:=\left\{y\,:\,\ee^{n_0-2}\leq \rho(y)\leq \ee^{n_0+2}\right\}.
			$$
			By $B(x,\rho(x)/2)\subset U_{n_0}$ and \eqref{220614409}, we have
			\begin{alignat*}{2}
				&&&\sum_{i=0}^k\left(\rho(x)^{\theta/p+i}\left|D^if(x)\right|\right)+\rho(x)^{\theta/p+k+\alpha}\sup_{y:|y-x|<\frac{\rho(x)}{2}}\frac{\left|D^kf(x)-D^kf(y)\right|}{|x-y|^{\alpha}}\\
				&\lesssim_N &&\,\ee^{n_0\theta/p}\bigg(\sum_{i=0}^k\left|D^i\big(f(\ee^{n_0}\,\cdot\,)\big)(x)\right| \\
				&&&\qquad\qquad\qquad +\sup_{\ee^{-{n_0}}y\in U_{n_0}}\frac{\left|D^k\big(f(\ee^{n_0}\,\cdot\,)\big)(x)-D^k\big(f(\ee^{n_0}\,\cdot\,)\big)(y)\right|}{|x-y|^{\alpha}}\bigg)\\
				&\leq &&\sum_{|n-n_0|\leq B}\ee^{n_0\theta/p}\left\|(f\zeta_{0,(n)})(\ee^{n}\,\cdot\,)\right\|_{C^{k,\alpha}}\\
				&\lesssim_N &&\left(\sum_{n\in\bZ}\ee^{n\theta}\left\|(f\zeta_{0,(n)})(\ee^{n}\,\cdot\,)\right\|_{H_{p}^{\gamma}}^p\right)^{1/p}\,,
			\end{alignat*}
			where $N=N(d,p,\gamma,\theta,k,\delta)$.
			By (1) of this proposition, the proof is completed.
		\end{proof}

		\subsection{Solvability of the Poisson equation}\label{0043}
			The goal of this subsection is to prove the following theorem:
		
		\begin{thm}\label{21.09.29.1}
			Let $\Omega$ be an open set admitting the Hardy inequality \eqref{hardy} and $\psi$ be a superharmonic Harnack function on $\Omega$, with its regularization $\Psi$.
			Then for any $p\in(1,\infty)$, $\mu\in (-1/p,1-1/p)$, and $\gamma\in\bR$, the following assertion holds:
			For any $\lambda\geq 0$ and $f\in \Psi^{\mu}H_{p,d+2p-2}^{\gamma}(\Omega)$, the equation
			\begin{align}\label{220613956}
				\Delta u-\lambda u=f
			\end{align}
			has a unique solution $u$ in $\Psi^{\mu}H_{p,d-2}^{\gamma+2}(\Omega)$.
			Moreover, we have
			\begin{align}\label{220606433}
				\|u\|_{\Psi^{\mu}H_{p,d-2}^{\gamma+2}(\Omega)}+\lambda\|u\|_{\Psi^{\mu}H_{p,d+2p-2}^{\gamma}(\Omega)}\leq N\|f\|_{\Psi^{\mu}H_{p,d+2p-2}^{\gamma}(\Omega)},
			\end{align}
			where $N=N(d,p,\gamma,\mu,\mathrm{C}_0(\Omega),\mathrm{C}_2(\Psi),\mathrm{C}_3(\psi,\Psi))$.
		\end{thm}
		Here, $\mathrm{C}_0(\Omega)$ is the constant in \eqref{hardy}, and $\mathrm{C}_2(\Psi)$ and $\mathrm{C}_3(\psi,\Psi)$ are the constants in Definition \ref{21.10.14.1}.
		
		In Theorem~\ref{21.09.29.1}, trivial examples of $\psi$ and $\Psi$ are given by $\psi=\Psi=1_{\Omega}$.
		Another example of $\psi$ is given in Example \ref{220912411}, which is constructed  from the Green function and is valid for any domain admitting the Hardy inequality.

		\begin{remark}\label{220617}
			In Theorem \ref{21.09.29.1}, the spaces $\Psi^{\mu} H_{p,d-2}^{\gamma+2}(\Omega)$ and $\Psi^{\mu} H_{p,d+2p-2}^{\gamma}(\Omega)$ do not depend on the particular choice of $\Psi$ among regularizations of $\psi$ (see Remark \ref{21.10.05.2}).
			If we take $\Psi$ to be $\widetilde{\psi}$, the regularization of $\psi$ constructed in Lemma \ref{21.05.27.3}.(1),  
			then Theorem \ref{21.09.29.1} can be reformulated in terms of $\psi$.
			Indeed, $\mathrm{C}_2(\widetilde{\psi})$ and $\mathrm{C}_3(\psi,\widetilde{\psi})$ depend only on $d$ and $\mathrm{C}_1(\psi)$, and therefore the constant $N$ in \eqref{220606433} depends only on $d$, $p$, $\gamma$, $\mu$, $\mathrm{C}_0(\Omega)$, and $\mathrm{C}_1(\psi)$.
			Additionally, for the case $\gamma\in\bZ$, equivalent norms of $\widetilde{\psi}^{\,\mu}H_{p,d-2}^{\gamma+2}(\Omega)$ and $\widetilde{\psi}^{\,\mu}H_{p,d+2p-2}^{\gamma}(\Omega)$ are established in Lemma \ref{220512433}, and they can also be reformulated in terms of $\psi$.
		\end{remark}
		
		\begin{remark}
			If $\mu\notin (-1/p,1-1/p)$, then Theorem \ref{21.09.29.1} fails in general, as pointed out in \cite[Remark 4.3]{Krylov1999-1}.
			To illustrate this, consider the equation
			\begin{align}\label{240122407}
			\Delta u=f\quad \text{in}\quad \Omega:=(0,\pi)
			\end{align}
			and put $\psi(x)=\Psi(x)=\sin x$ and $\gamma=0$.
			
			Let $\mu\geq 1-1/p$, and let $f\in C_c^{\infty}(\Omega)$ with $f\leq 0$, so that $f\in \Psi^{\mu}L_{p,d+2p-2}(\Omega)$.
			We assume that there exists a solution $u_1\in \Psi^{\mu}H_{p,d-2}^2(\Omega)$ of \eqref{240122407}.
			Then this $u_1$ belongs to $H_{p,d-2}^2(\Omega)$.
			Let $u_0$ be the classical solution of \eqref{240122407} with the boundary condition $u(0)=u(\pi)=0$.
			Then $u_0\in H_{p,d-2}^2(\Omega)$.
			By Theorem \ref{21.09.29.1}, \eqref{240122407} has a unique solution, and therefore $u_0\equiv u_1$.
			However, $u_0\notin \Psi^{\mu}H_{p,d-2}^2(\Omega)$ for all $\mu\geq 1-1/p$, since $u_0\simeq \sin x$.
			This is a contradiction. 
			Therefore there exists no solution $u\in \Psi^{\mu}H_{p,d-2}^2(\Omega)$ of \eqref{240122407}.
			
			If $\mu< -1/p$, then both $0$ and $1_{\Omega}$ belong to $\Psi^{\mu}H^2_{p,d-2}(\Omega)$ (see Lemma \ref{220512433}).
			Therefore \eqref{240122407} with $f=0$ has at least two solutions in $\Psi^{\mu}H^2_{p,d-2}(\Omega)$.
			
			Consider the case $\mu=-1/p$.
			For $n\in\bN$, take $\zeta_n\in C_c^{\infty}(\Omega)$ such that
			$$
			1_{\left[\frac{2}{n},\pi-\frac{2}{n}\right]}\leq \zeta_n \leq 1_{\left[\frac{1}{n},\pi-\frac{1}{n}\right]}\quad\text{and}\quad \left|D^k\zeta_n\right|\leq N(k)n^{k}\,.
			$$
			By putting $u:=\zeta_n$, one sees that no constant $N$ satisfies \eqref{220606433}.
		\end{remark}
		
		\begin{example}\label{220912411}
			Let $\Omega\subset \bR^d$ be a domain admitting the Hardy inequality.
			We denote by $G_{\Omega}:\Omega\times \Omega\rightarrow [0,+\infty]$ the Green function associated with the Poisson equation. (For the definition and existence of $G_\Omega$, see \cite[Definition 4.1.3]{AG}, \cite[Theorems 4.1.2 and 5.3.8]{AG}, and \cite[Theorem 2]{AA}.)
			We claim that, for any fixed $x_0\in\Omega$, $\phi_0:=G_{\Omega}(x_0,\,\cdot\,)\wedge 1$ is a superharmonic Harnack function on $\Omega$.
			It is worth noting that $\phi_0$ is the smallest positive classical superharmonic function, up to constant multiples (see \cite[Lemma 4.1.8]{AG}).
			That is, if $\phi$ is a positive classical superharmonic function on $\Omega$, then there exists $N_0=N(\phi,\Omega,x_0)>0$ such that $\phi_0\leq N_0\phi$ in $\Omega$.
			
			Note that $G_{\Omega}(x_0,\,\cdot\,)$ is a positive classical superharmonic function on $\Omega$, while harmonic in $\Omega\setminus \{x_0\}$.
			This implies that $\phi_0$ is a classical superharmonic function on $\Omega$ (see Lemma \ref{21.05.18.1}.(1)).
			
			For $x\in\Omega$, set $B(x):=B\big(x,\rho(x)/8\big)$.
			If $|x-x_0|>\rho(x)/4$, then $G_{\Omega}(x_0,\,\cdot\,)$ is harmonic on $B\big(x,\rho(x)/4\big)$.
			By the Harnack inequality, we have
			\begin{align*}
				\sup_{B(x)}\phi_0=\Big(\sup_{y\in B(x)}G_{\Omega}(x_0,y)\Big)\wedge 1\lesssim_d \Big(\inf_{y\in B(x)}G_{\Omega}(x_0,y)\Big)\wedge 1=\inf_{B(x)}\phi_0\,.
			\end{align*}
			If $|x-x_0|\leq \rho(x)/4$, then $\rho(x)\leq \frac{4}{3}\rho(x_0)$, which implies that $B(x)\subset B\big(x_0,\rho(x_0)/2\big)$.
			By Lemma \ref{21.04.23.3}.(3), there exists $\epsilon_0\in(0,1]$ such that $G(x_0,\,\cdot\,)\geq \epsilon_0$ on $B\big(x_0,\rho(x_0)/2\big)$.
			Therefore, we have
			\begin{align*}
				\sup_{B(x)}\phi_0&\leq 1\leq \epsilon_0^{-1}\inf_{B(x)}\phi_0\,.
			\end{align*}
			Consequently, $\phi_0$ is a superharmonic Harnack function on $\Omega$.
		\end{example}
		
		To prove Theorem \ref{21.09.29.1}, we make use of Lemmas \ref{21.05.13.8} and \ref{21.11.12.1}.
		These lemmas are based on a localization argument, in which $\Omega$ is an arbitrary domain and $\Psi$ is an arbitrary regular Harnack function.
		Theorem \ref{21.09.29.1} will be proved after the proof of Lemma \ref{21.11.12.1}.
		
		\begin{lemma}[Higher-order estimates]\label{21.05.13.8}
			Let $p\in (1,\infty)$, $\gamma,\,s\in\bR$, $\theta\in\bR$, and $\Psi$ be a regular Harnack function.
			Then there exists a constant $N=N(d,p,\theta,\gamma,\mathrm{C}_2(\Psi),s)>0$ such that 
			the following assertion holds:
			Let $\lambda\geq 0$ and suppose that $u,\,f\in\cD'(\Omega)$ satisfy \eqref{220613956}.
			Then
			\begin{align}\label{21.09.30.1}
				\| u\|_{\Psi H^{\gamma+2}_{p,\theta}(\Omega)}+\lambda\| u\|_{\Psi  H^{\gamma}_{p,\theta+2p}(\Omega)}\leq N\left(\| u\|_{ \Psi H^{s}_{p,\theta}(\Omega)}+\|f\|_{\Psi H^{\gamma}_{p,\theta+2p}(\Omega)}\right)\,.
			\end{align}
		\end{lemma}
		
		\begin{proof}
			We denote $\Phi=\Psi^{-1}$ so that $\mathrm{C}_2(\Phi)$ depends only on $d$ and $\mathrm{C}_2(\Psi)$.
			
			\textbf{Step 1.} First, we consider the case $s\geq \gamma+1$. 
			One can certainly assume that 
			$$
			\|\Phi u\|_{ H^{s}_{p,\theta}(\Omega)}+\|\Phi f\|_{ H^{\gamma}_{p,\theta+2p}(\Omega)}<\infty\,,
			$$
			for if not, there is nothing to prove.
			Since
			\begin{align*}
				\|\Phi u\|_{H_{p,\theta}^{\gamma+1}(\Omega)}\lesssim_{d,p,s,\gamma} \|\Phi u\|_{H_{p,\theta}^{s}(\Omega)}
			\end{align*}
			(see Lemma \ref{21.05.20.3}.(1)), it suffices to prove for $s=\gamma+1$.
			Put
			\begin{align*}
				v_n(x)=\zeta_0\big(\ee^{-n}\trho(\ee^nx)\big)\Phi(\ee^nx) u(\ee^nx)\,.
			\end{align*}
			Since
			$$
			\sum_{n\in\bZ}\ee^{n\theta}\left\|v_n\right\|_{H_p^{\gamma+1}(\bR^d)}^p=\|\Phi u\|_{H_{p,\theta}^{\gamma+1}(\Omega)}^p<\infty\,,
			$$
			we have $v_n\in H_p^{\gamma+1}(\bR^d)$.
			Observe that
			\begin{align}\label{2301261222}
				\Delta v_n-\ee^{2n}\lambda v_n=\widetilde{f}_n\quad \text{in}\quad \bR^d\,,
			\end{align}
			where
			\begin{align*}
				\widetilde{f}_n(x):=\,&\ee^{2n}\zeta_{0,(n)}(\ee^n x)\big(\Phi f\big)(\ee^n x)-\ee^{2n}\zeta_{0,(n)}(\ee^n x)\big(\Phi\Delta u\big)(\ee^n x)+\Delta v_n(x)\\
				=\,&\Big[\,\ee^{2n}\zeta_{0,(n)}\Big(\Phi f
				+2\big(\nabla 	u\cdot\nabla\Phi\big)+(\Delta\Phi)u\Big)
				\\
				&\,\,\,\, +\ee^n\big(\zeta_0'\big)_{(n)}\Big(2\big(\nabla\trho\cdot \nabla(\Phi u)\big)+(\Delta\trho)\Phi u\Big)				+\big(\zeta_0''\big)_{(n)}|\nabla\trho|^2\Phi u
				\,\Big](\ee^nx)\,.
			\end{align*}
			Here, $\big(\zeta_0'\big)_{(n)}$ and $\big(\zeta_0''\big)_{(n)}$ are defined in the same way as in \eqref{230130543}. 
			Make use of Lemmas \ref{21.05.20.3}.(1) - (3) and \ref{21.09.29.4}.(3) to obtain
			\begin{align}
					&\sum_{n\in\bZ}\ee^{n\theta}\big\|\widetilde{f}_n\big\|_{H^{\gamma}_p(\bR^d)}^p\nonumber\\
					\lesssim_N &\left\|\Phi f\right\|^p_{H_{p,\theta+2p}^{\gamma}(\Omega)}
					+\left\|2\big(\nabla 	u\cdot\nabla\Phi\big)+(\Delta\Phi)u\right\|^p_{H_{p,\theta+2p}^{\gamma}(\Omega)}
					\label{22.04.15.1}\\
					&+\left\|2\big(\nabla\trho\cdot \nabla(\Phi u)\big)+(\Delta\trho)\Phi u\right\|^p_{H_{p,\theta+p}^{\gamma}(\Omega)}
					+\left\||\nabla\trho|^2\Phi u\right\|^p_{H_{p,\theta}^{\gamma}(\Omega)}
					\nonumber\\
					\lesssim_N & \|\Phi f\|^p_{H_{p,\theta+2p}^{\gamma}(\Omega)}
					+\|\Phi u\|^p_{H^{\gamma+1}_{p,\theta}(\Omega)}<\infty\,,\nonumber
			\end{align}
			where $N=N(d,p,\gamma,\theta,\mathrm{C}_2(\Psi))$.
			This implies that for any $n\in\bZ$, $\tilde{f}_n\in H^{\gamma}_p(\bR^d)$.
			
			By \eqref{2301261222} and that $v_n\in H_p^{\gamma+1}(\bR^d)$ and $\widetilde{f}_n\in H_p^{\gamma}(\bR^d)$, we have 
			\begin{align*}
				\begin{gathered}
					V_n:=(1-\Delta)^{\gamma/2}v_n\in H_p^1(\bR^d)\quad,\quad F_n:=(1-\Delta)^{\gamma/2}\widetilde{f}_n\in L_p(\bR^d)\,,\\
					\Delta V_n-(\ee^{2n}\lambda+1)V_n=F_n-V_n\,.
				\end{gathered}
			\end{align*}
			It is implied by classical results for the Poisson equation in $\bR^d$ (see, \textit{e.g.}, \cite[Theorem 4.3.8, Theorem 4.3.9]{Krylov2008})
			that
				\begin{align*}
					&\|v_n\|_{H_p^{\gamma+2}(\bR^d)}+\ee^{2n}\lambda\|v_n\|_{H_p^{\gamma}(\bR^d)}=\left\|V_n\right\|_{H_p^2(\bR^d)}+\ee^{2n}\lambda\left\|V_n\right\|_{L_p(\bR^d)}\nonumber\\
					\leq\,& \left\|\Delta V_n\right\|_{L_p(\bR^d)}+(\ee^{2n}\lambda+1)\left\|V_n\right\|_{L_p(\bR^d)}\lesssim_{d,p} \|F_n-V_n\|_{L_p(\bR^d)}\\
					\leq\,&\|\widetilde{f}_n\|_{H_p^{\gamma}(\bR^d)}+\|v_n\|_{H_p^{\gamma}(\bR^d)}\,.\nonumber
				\end{align*}
			Combine this with 
			\eqref{22.04.15.1} to obtain that
			\begin{align*}
					&\|\Phi u\|_{H^{\gamma+2}_{p,\theta}(\Omega)}^p+\lambda^p\|\Phi u\|_{H^{\gamma}_{p,\theta+2p}(\Omega)}^p= \sum_{n\in\bZ}\ee^{n\theta}\Big(\|v_n\|_{H^{\gamma+2}_p(\bR^d)}^p+(\ee^{2n}\lambda)^p\|v_n\|_{H^{\gamma}_p(\bR^d)}^p\Big)\\
					\lesssim_N\,& \sum_{n\in\bZ}\ee^{n\theta}\left(\|v_n\|_{H^{\gamma}_p(\bR^d)}^p+\|\tilde{f}_n\|_{H^{\gamma}_p(\bR^d)}^p\right)\lesssim_N\|\Phi u\|_{H^{\gamma+1}_{p,\theta}(\Omega)}^p+\|\Phi f\|^p_{H_{p,\theta+2p}^{\gamma}(\Omega)}\,.
			\end{align*}
			Therefore the case $s=\gamma+1$ is proved.
			Consequently, \eqref{21.09.30.1} holds for all $s\geq \gamma+1$.

			\textbf{Step 2.}
			For $s<\gamma+1$, take $k\in\bN$ such that $\gamma+1-k\leq s<\gamma+2-k$.
			By the result in Step 1, \eqref{21.09.30.1} holds for $(\gamma,s)$ replaced by $(\gamma,\gamma+1)$, $(\gamma-1,\gamma)$, ..., $(\gamma-k,\gamma+1-k)$.
			Therefore we have
			\begin{align*}
				&\|\Phi u\|_{H^{\gamma+2}_{p,\theta}(\Omega)}+\lambda\|\Phi u\|_{H^{\gamma}_{p,\theta+2p}(\Omega)}\lesssim_N \|\Phi u\|_{H^{\gamma+1}_{p,\theta}(\Omega)}+\|\Phi f\|_{H_{p,\theta+2p}^{\gamma}(\Omega)}\\
				\lesssim_N \,&\,\,\,\,\cdots\,\,\,\,\lesssim_N \|\Phi u\|_{H^{\gamma-k+1}_{p,\theta}(\Omega)}+\|\Phi f\|_{H_{p,\theta+2p}^{\gamma}(\Omega)}\,.
			\end{align*}
			Since $\|\Phi u\|_{H^{\gamma-k+1}_{p,\theta}(\Omega)}\lesssim \|\Phi u\|_{H^{s}_{p,\theta}(\Omega)}$ (see Lemma \ref{21.05.20.3}.(1)), the proof is completed.
		\end{proof}
		
		\begin{lemma}\label{21.11.12.1}
				Let $p\in (1,\infty)$ and $\theta\in\bR$, and let $\Psi$ be a regular Harnack function.
				Let $\lambda\geq 0$ and suppose that there exists $\gamma\in\bR$ such that the following holds: 
			\begin{itemize}
				\item[]For any $f\in \Psi H_{p,\theta+2p}^{\gamma}(\Omega)$, there exists a unique solution $u$ of equation \eqref{220613956} in $\Psi H_{p,\theta}^{\gamma+2}(\Omega)$.
				For this solution, we have
				\begin{align}\label{220606837}
					\|u\|_{\Psi H_{p,\theta}^{\gamma+2}(\Omega)}+\lambda\|u\|_{\Psi H_{p,\theta+2p}^{\gamma}(\Omega)}\leq N_{\gamma}\|f\|_{\Psi H_{p,\theta+2p}^{\gamma}(\Omega)}\,,
				\end{align}
				where $N_{\gamma}$ is a constant independent of $f$ and $u$.
			\end{itemize}
			Then for all $s\in\bR$, the following holds:
			\begin{itemize}
				\item[]
				For any $f\in \Psi H_{p,\theta+2p}^{s}(\Omega)$, there exists a unique solution $u$ of equation \eqref{220613956} in $\Psi H_{p,\theta}^{s+2}(\Omega)$.
				For this solution, we have
				\begin{align}\label{22.04.01.1}
					\|u\|_{\Psi H_{p,\theta}^{s+2}(\Omega)}+\lambda\|u\|_{\Psi H_{p,\theta+2p}^{s}(\Omega)}\leq N_s\|f\|_{\Psi H_{p,\theta+2p}^{s}(\Omega)}\,,
				\end{align}
				where $N_s=N(d,p,\gamma,\theta,\mathrm{C}_2(\Psi),N_{\gamma},s)$.
			\end{itemize}
		\end{lemma}
		
		\begin{proof}
			To prove the uniqueness of solutions, let us assume that $\overline{u}\in \Psi H_{p,\theta}^{s+2}(\Omega)$ satisfies $\Delta \overline{u}-\lambda \overline{u}=0$.
			By Lemma \ref{21.05.13.8}, $\overline{u}$ belongs to $\Psi H_{p,\theta}^{\gamma+2}(\Omega)$. 
			By the assumption of this lemma, the zero distribution is the unique solution in $\Psi H_{p,\theta}^{\gamma+2}(\Omega)$ for the equation $\Delta u-\lambda u=0$.
			Consequently, $\overline{u}$ is also the zero distribution, and the uniqueness of solutions is proved.
			Thus, it remains to show the existence of solutions and estimate \eqref{22.04.01.1}.
			
			\textbf{Step 1.} We first consider the case $s> \gamma$.
			Let $f\in\Psi H^{s}_{p,\theta+2p}(\Omega)$.
			Since $\Psi H^s_{p,\theta+2p}(\Omega)\subset \Psi H^{\gamma}_{p,\theta+2p}(\Omega)$, $f$ belongs to $\Psi H^{\gamma}_{p,\theta+2p}(\Omega)$, and hence there exists a solution $u\in \Psi H^{\gamma+2}_{p,\theta}(\Omega)$ of equation \eqref{220613956}.
			It follows from Lemma \ref{21.05.13.8}, \eqref{220606837}, and Lemma \ref{21.05.20.3}.(1) that
			\begin{align*}
				&\left\|u\right\|_{\Psi H^{s+2}_{p,\theta}(\Omega)}+\lambda\left\|u\right\|_{\Psi H^{s}_{p,\theta+2p}(\Omega)}\lesssim_N\left\|u\right\|_{\Psi H^{\gamma+2}_{p,\theta}(\Omega)}+\left\|f\right\|_{\Psi H^{s}_{p,\theta+2p}(\Omega)}\\
				\leq \,& N_{\gamma}\|f\|_{\Psi H_{p,\theta+2p}^{\gamma}(\Omega)}+\|f\|_{\Psi H_{p,\theta+2p}^s(\Omega)}\lesssim_N (N_{\gamma}+1)\left\|f\right\|_{\Psi H^{s}_{p,\theta+2p}(\Omega)}\,,
			\end{align*}
			where $N=N(d,p,\theta,\gamma,\mathrm{C}_2(\Psi),s)$.
			Therefore $u$ belongs to $\Psi H_{p,\theta}^{s+2}(\Omega)$, and the proof is completed.
			
			\textbf{Step 2.} Consider the case $s<\gamma$.
			Since the case $s\geq\gamma$ is proved in Step 1, by mathematical induction, it is sufficient to show that if this lemma holds for $s=s_0+1$, then this also holds for $s=s_0$.
			
			Let us assume that this lemma holds for $s=s_0+1$. 
			For $f\in \Psi H_{p,\theta+2p}^{s_0}(\Omega)$, by Lemma \ref{22.02.16.1}, there exists $f^0\in \Psi H_{p,\theta+2p}^{s_0+1}(\Omega)$ and $f^1,\,\ldots,\,f^d\in \Psi H_{p,\theta+p}^{s_0+1}(\Omega)$	such that $f=f^0+\sum_{i=1}^dD_if^i$ and
			\begin{align}\label{220506335}
				\left\|f^0\right\|_{\Psi H_{p,\theta+2p}^{s_0+1}(\Omega)}+\sum_{i=1}^d\left\|\trho^{\,-1}f^i\right\|_{\Psi H_{p,\theta+2p}^{s_0+1}(\Omega)}
				\leq N\|f\|_{\Psi H_{p,\theta+2p}^{s_0}(\Omega)}\,,
			\end{align}
			where $N=N(d,p,\theta,s_0,\mathrm{C}_2(\Psi))$.
			By the assumption that this lemma holds for $s=s_0+1$, there exist $v^0,\,\cdots,\,v^d\in \Psi H_{p,d-2}^{s_0+3}(\Omega)$ such that
			\begin{align*}
				\Delta v^0-\lambda v^0=f^0\quad\text{and}\quad \Delta v^i-\lambda v^i=\trho^{\,-1} f^i\quad\text{for }i=1,\,\ldots,\,d\,,
			\end{align*}
			and
			\begin{alignat}{2}
				&&&\sum_{i=0}^d\left(\left\|v^i\right\|_{\Psi H^{s_0+3}_{p,\theta}(\Omega)}+\lambda\left\|v^i\right\|_{\Psi H^{s_0+1}_{p,\theta+2p}(\Omega)}\right)\nonumber\\
				&\leq&&N_{s_0+1}\bigg(\left\|f^0\right\|_{\Psi H^{s_0+1}_{p,\theta+2p}(\Omega)}+\sum_{i=1}^d\left\|\trho^{\,-1} f^i\right\|_{\Psi H^{s_0+1}_{p,\theta+2p}(\Omega)}\bigg)\label{21.09.30.100}\\
				&\lesssim_N\,&&N_{s_0+1}\|f\|_{\Psi H_{p,\theta+2p}^{s_0}(\Omega)}\,,\nonumber
			\end{alignat}
			where the last inequality follows from \eqref{220506335}. 
			Set $v=v^0+\sum_{i=1}^dD_i\big(\trho v^i\big)$, and observe that
			$$
			\Delta v-\lambda v=f +\sum_{i=1}^dD_i\big(\Delta(\trho v^i)-\trho\Delta v^i\big)\,.
			$$
			By Lemmas \ref{21.05.20.3} and \ref{21.09.29.4}.(3), we have
			\begin{alignat*}{2}
			&\left\|D_i\left(\Delta(\trho v^i)-\trho\Delta v^i\right)\right\|_{\Psi H^{s_0+1}_{p,\theta+2p}(\Omega)}\lesssim_N \|\Delta(\trho v^i)-\trho\Delta v^i\|_{\Psi H^{s_0+2}_{p,\theta+p}(\Omega)}\\
				\leq \,& \big\|\big(D^2\trho\big) v^i\big\|_{\Psi H^{s_0+2}_{p,\theta+p}(\Omega)}+\big\|\big(D\trho\big) Dv^i\big\|_{\Psi H^{s_0+2}_{p,\theta+p}(\Omega)}\lesssim_N \|v^i\|_{\Psi H^{s_0+3}_{p,\theta}(\Omega)}<\infty\,,
			\end{alignat*}
			where $N=N(d,p,\theta,s_0,\mathrm{C}_2(\Psi))$.
			By the assumption that this lemma holds for $s=s_0+1$, there exists $w\in \Psi H^{s_0+3}_{p,\theta}(\Omega)$ such that
			$$
			\Delta w-\lambda w=\sum_{i=1}^dD_i\big(\Delta(\trho v^i)-\trho\Delta v^i\big)\quad(=\Delta v-\lambda v-f)\,.
			$$
			This $w$ satisfies
				\begin{alignat}{2}
					\|w\|_{\Psi H^{s_0+3}_{p,\theta}(\Omega)}+\lambda\|w\|_{\Psi H^{s_0+1}_{p,\theta+2p}(\Omega)}&\leq&& N_{s_0+1}\sum_{i=1}^d\left\|D_i\left(\Delta(\trho v^i)-\trho\Delta v^i\right)\right\|_{\Psi H^{s_0+1}_{p,\theta+2p}(\Omega)}\nonumber\\
					&\lesssim_N &&N_{s_0+1}\sum_{i=1}^d\|v^i\|_{\Psi H^{s_0+3}_{p,\theta}(\Omega)}\,.\label{21.09.30.200}
				\end{alignat}
			Set $u=v-w=v^0+\sum_{i=1}^d D_i(\trho v^i)-w$.
			Then $u$ satisfies $\Delta u-\lambda u = f$.
			Moreover, by \eqref{21.09.30.100} and \eqref{21.09.30.200}, we obtain \eqref{22.04.01.1} for $s=s_0$.
		\end{proof}

		\begin{proof}[Proof of Theorem \ref{21.09.29.1}]
			By Lemma \ref{21.11.12.1}, it suffices to prove for $\gamma=0$.
			
			\textbf{\textit{A priori} estimates.} Let $u\in\Psi^{\mu}H_{p,d-2}^2(\Omega)$ and $\Delta u-\lambda u\in\Psi^{\mu}L_{p,d+2p-2}(\Omega)$.
			By Lemma \ref{21.05.13.8}, we obtain
			\begin{align}\label{220506424}
				\begin{split}
					&\|u\|_{\Psi^{\mu}H_{p,d-2}^{2}(\Omega)}+\lambda\|u\|_{\Psi^{\mu}L_{p,d+2p-2}(\Omega)}\\
					\lesssim_N&\|u\|_{\Psi^{\mu}L_{p,d-2}(\Omega)}+\|\Delta u-\lambda u\|_{\Psi^{\mu}L_{p,d+2p-2}(\Omega)}<\infty\,,
				\end{split}
			\end{align}
			where $N=N(d,p,\mu,\mathrm{C}_2(\Psi))$.
			By \eqref{220506424} and Lemma \ref{21.09.29.4}.(5), whether $\lambda=0$ or $\lambda>0$, there exists $u_n\in C_c^{\infty}(\Omega)$ such that
			$$
			\lim_{n\rightarrow\infty}\Big(\|u-u_n\|_{\Psi^{\mu}H_{p,d-2}^{2}(\Omega)}+\lambda\|u-u_n\|_{\Psi^{\mu} L_{p,d+2p-2}(\Omega)}\Big) =0\,.
			$$
			This implies that
			$$
			\lim_{n\rightarrow \infty}\big\|\big(\Delta-\lambda\big)(u-u_n)\big\|_{\Psi^{\mu}L_{p,d+2p-2}(\Omega)}= 0\,.
			$$
			Since $\Psi$ is a regularization of the superharmonic Harnack function $\psi$, Theorem \ref{21.05.13.2} and Lemma \ref{220512433} imply
				\begin{align}\label{221213517}
				\begin{aligned}
					&\|u_n\|_{\Psi^{\mu}L_{p,d-2}(\Omega)}\simeq_N\int_{\Omega}|u_n|^p\psi^{-\mu p}\rho^{-2}\dd x\\
					\lesssim_{N}\,&\int_{\Omega}|\Delta u_n-\lambda u_n|^p \psi^{-\mu p}\rho^{2p-2} \dd x\simeq_N \|\Delta u_n-\lambda u_n\|_{\Psi^{\mu}L_{p,d+2p-2}(\Omega)}\,,
				\end{aligned}
				\end{align}
			where $N=N(d,p,\mu,\mathrm{C}_0(\Omega),\mathrm{C}_2(\Psi),\mathrm{C}_3(\psi,\Psi))$.
			By letting $n\rightarrow \infty$, we obtain \eqref{221213517} for $u$ instead of $u_n$.
			By combining this with \eqref{220506424}, we have
			\begin{align}\label{220506452}
					&\| u\|_{\Psi^{\mu}H_{p,d-2}^{2}(\Omega)}+\lambda\| u\|_{\Psi^{\mu}L_{p,d+2p-2}(\Omega)}\nonumber\\
					\lesssim_N\,&\|u\|_{\Psi^{\mu}L_{p,d-2}(\Omega)}+\|\Delta u-\lambda u\|_{\Psi^{\mu}L_{p,d+2p-2}(\Omega)}\\
					\lesssim_N\,&\|\Delta u-\lambda u\|_{\Psi^{\mu}L_{p,d+2p-2}(\Omega)}\,.\nonumber
			\end{align}
			Note that estimate \eqref{220506452} also implies uniqueness of solutions.
			
			\textbf{Existence of solutions.}
			Let $f\in \Psi^{\mu}L_{p,d+2p-2}(\Omega)$.
			Since $C_c^{\infty}(\Omega)$ is dense in $\Psi^{\mu}L_{p,d+2p-2}(\Omega)$, there exists $f_n\in C_c^{\infty}(\Omega)$ such that $f_n\rightarrow f$ in $\Psi^{\mu}L_{p,d+2p-2}(\Omega)$.
			Lemmas \ref{21.05.25.3} and \ref{220512433} yield that for each $n\in\bN$, there exists $u_n\in\Psi^{\mu}L_{p,d-2}^2(\Omega)$ such that $\Delta u_n-\lambda u_n=f_n$.
			By Lemma \ref{21.05.13.8}, $u_n\in\Psi^{\mu}H_{p,d-2}^2(\Omega)$.
			Since $f_n\rightarrow f$ in $\Psi^{\mu} L_{p,d+2p-2}(\Omega)$, it follows from \eqref{220506452} that
			$$
			\|u_n-u_m\|_{\Psi^{\mu}H_{p,d-2}^2(\Omega)}\leq N\|f_n-f_m\|_{\Psi^{\mu}L_{p,d+2p-2}}\rightarrow 0
			$$
			as $n,\,m\rightarrow \infty$.
			Therefore there exists $u\in\Psi^{\mu}H_{p,d-2}^2(\Omega)$ such that $u_n$ converges to $u$ in $\Psi^{\mu}H_{p,d-2}^2(\Omega)$.
			Since $u_n$ and $f_n$ converge to $u$ and $f$ in the sense of distributions, respectively (see Lemma \ref{21.09.29.4}.(2)), $u$ is a solution of equation \eqref{220613956}.
		\end{proof}
		
		We end this subsection with a global uniqueness result.
		
		\begin{thm}[Global uniqueness]\label{220530526}
			Let $\Omega$ be an open set admitting the Hardy inequality \eqref{hardy}.
			For each $i=1,\,2$, let $\Psi_i$ be a regularization of a superharmonic Harnack function, $p_i\in(1,\infty)$, $\gamma_i\in\bR$, and $\mu_i\in (-1/p_i,1-1/p_i)$.
			Let $f\in \bigcap_{i=1,2}\Psi_i^{\mu_i} H_{p_i,d+2p_i-2}^{\gamma_i}(\Omega)$, and for each $i=1,\,2$, let $u^{(i)}\in\Psi_i^{\mu_i} H_{p_i,d-2}^{\gamma_i+2}(\Omega)$ be a solution of \eqref{220613956}.
			Then $u^{(1)}=u^{(2)}$ in $\cD'(\Omega)$.
		\end{thm}
		
		\begin{proof}
			By Lemma \ref{21.09.29.4}.(5), there exist $\{f_n\}\subset  C_c^{\infty}(\Omega)$ such that
			$
			f_n\rightarrow f$ in $\bigcap_{i=1,2} \Psi_i^{\mu_i} H_{p_i,d+2p_i-2}^{\gamma_i}(\Omega)
			$.
			Since $\{f_n\}\subset  C_c^{\infty}(\Omega)$, Lemmas \ref{21.05.25.3} and \ref{220512433} yield that for each $n\in\bN$, there exists $u_n\in \bigcap_{i=1,2}\Psi_i^{\mu_i} L_{p_i,d-2}(\Omega)$ such that $\Delta u_n-\lambda u_n=f_n$.
			Lemma \ref{21.05.13.8} yields that $u_n\in \bigcap_{i=1,2}\Psi_i^{\mu_i} H_{p_i,d-2}^{\gamma_i+2}(\Omega)$.
			Since
			$$
			(\Delta-\lambda)\big(u_n-u^{(1)}\big)=(\Delta-\lambda)\big(u_n-u^{(2)}\big)=f_n-f\,,
			$$
			For each $i=1,\,2$, Theorem \ref{21.09.29.1} implies that $u_n\rightarrow u^{(i)}$ in $\Psi_i^{\mu_i} H_{p_i,d-2}^{\gamma_i+2}(\Omega)$, and by Lemma \ref{21.09.29.4}.(2),  this convergences also holds in $\cD'(\Omega)$.
			Therefore $u^{(1)}=u^{(2)}=\lim_{n\rightarrow \infty}u_n$ in $\cD'(\Omega)$.
		\end{proof}

		
		\mysection{Application I - Domain with fat exterior or thin exterior}\label{app.}
		
		In this section, we introduce applications of the results in Section \ref{0040} to domains satisfying fat exterior or thin exterior conditions.
		The notions of the fat exterior and thin exterior are closely related to the geometry of a domain $\Omega$, namely the Hausdorff dimension and the Aikawa dimension of $\Omega^c$.
		
		For a set $E\subset \bR^d$, the Hausdorff dimension of $E$ is defined as 
		$$
		\dim_{\cH}(E):=\inf\big\{\lambda\geq 0\,:\,H^{\lambda}_{\infty}(E)=0\big\}\,.
		$$
		Here,
		$$
		\cH_{\infty}^{\lambda}(E):=\inf \Big\{\sum_{i\in\bN}r_i^{\lambda}\,:\,E\subset \bigcup_{i\in\bN}B(x_i,r_i)\quad\text{where }x_i\in E\text{ and }r_i>0\Big\}\,.
		$$
		The Aikawa dimension of $E$, denoted by $\dim_{\cA}(E)$, is defined by the infimum of $\beta\geq 0$ for which
		$$
		\sup_{p\in E,\,r>0}\frac{1}{r^{\beta}}\int_{B_r(p)}\frac{1}{d(x,E)^{d-\beta}}\dd x< \infty\,,
		$$
		with the convention $\frac{1}{0}=+\infty$.
		
		\begin{remark}\label{22.02.24.1}\,\,
			\begin{enumerate}
				\item While the Aikawa dimension is defined via integration, it is equivalent to a dimension defined via a covering property, known as the Assouad dimension (see \cite[Theorem 1.1]{LT}).
				
				\item  For any $E\subset \bR^d$, $\dim_{\cH}(E)\leq \dim_{\cA}(E)$, although equality does not hold in general (see \cite[Section 2.2]{lehr}).
				However, if $E$ is Ahlfors regular, for example, if $E$ is self-similar, such as the Cantor set or the Koch snowflake, then $\dim_{\cH}(E)=\dim_{\cA}(E)$; see \cite[Lemma 2.1]{lehr} and \cite[Theorem 4.14]{Mattila}.
			\end{enumerate} 
		\end{remark}
		 
		Koskela and Zhong \cite{KZ} established the dimensional dichotomy results for domains admitting the Hardy inequality, in terms of the Hausdorff and Minkowski dimensions.
		Their result can be expressed in terms of the Hausdorff and Aikawa dimensions, as shown in \cite[Theorem 5.3]{lehr}.
		
		\begin{prop}[Theorem 5.3 of \cite{lehr}]\label{22.02.07.1}
			Suppose a domain $\Omega\subset \bR^d$ admits the Hardy inequality.
			Then there is a constant $\epsilon>0$ such that for each $p\in\partial\Omega$ and $r>0$, either
			$$
			\dim_{\cH}\big(\Omega^c\cap \overline{B}(p,4r)\big)\geq d-2+\epsilon\quad\text{or}\quad \dim_{\cA}\big(\Omega^c\cap \overline{B}(p,r)\big)\leq d-2-\epsilon\,.
			$$
		\end{prop}
	 
		We refer readers to \cite{kinnunen2021,ward} for a deeper discussion of this dimensional dichotomy. 
		
		In view of Proposition \ref{22.02.07.1}, we consider domains $\Omega\subset \bR^d$ that satisfy one of the following conditions:
		\begin{enumerate}
			\item(Fat exterior) There exist $\epsilon\in(0,1)$ and $c>0$ such that
			\begin{align}\label{220617253}
				\cH^{d-2+\epsilon}_{\infty}\big(\Omega^c\cap \overline{B}(p,r)\big)\geq c\,r^{d-2+\epsilon}\quad\text{for all }p\in \partial\Omega\,\,,\,\,\,\,r>0\,.
			\end{align}
			
			\item(Thin exterior) $\dim_{\cA}(\Omega^c)<d-2$.
		\end{enumerate}
		These two conditions have been studied extensively; we review previous studies on these conditions, in particular those related to the Hardy inequality, in Sections \ref{fatex} and \ref{0062}.
		
		In this section and Section \ref{app2.} we construct superharmonic functions in various domains $\Omega\subset \bR^d$ that are comparable to $d(\,\cdot\,,\partial\Omega)^{\alpha}$ for some $\alpha$.
		This type of superharmonic function ensures the validity of the following statement for all $p\in(1,\infty)$ and suitable $\theta$ (see Lemma \ref{220617557}):
		\begin{statement}[$\Omega,p,\theta$]\label{22.02.19.1}
			For any $\lambda\geq 0$ and $\gamma\in\bR$, if $f\in H_{p,\theta+2p}^{\gamma}(\Omega)$, then the equation
				\begin{align}\label{240319417}
					\Delta u-\lambda u=f
				\end{align}
				has a unique solution $u$ in $H^{\gamma+2}_{p,\theta}(\Omega)$.
				Moreover, we have
				\begin{align}\label{2205241155}
					\|u\|_{H^{\gamma+2}_{p,\theta}(\Omega)}+\lambda\|u\|_{H_{p,\theta+2p}^{\gamma}(\Omega)}\leq N_1\|f\|_{H_{p,\theta+2p}^{\gamma}(\Omega)}\,,
				\end{align}
				where $N_1$ is a constant independent of $u$, $f$, and $\lambda$.
			\end{statement}
		
		\begin{lemma}\label{220617557}
			Let $\Omega$ admit the Hardy inequality \eqref{hardy}, and suppose that there exists a superharmonic function $\psi$ and constants $\alpha\in\bR$ and $M>0$ satisfying
			\begin{align}\label{231215620}
				M^{-1}\rho^{\alpha}\leq \psi\leq M\rho^{\alpha}\,.
			\end{align}
			Then Statement \ref{22.02.19.1} $(\Omega,p,\theta)$ holds for all $p\in(1,\infty)$ and $\theta\in\bR$ with 
			\begin{alignat*}{2}
				d-2-(p-1)\alpha<&\,\,\theta<\quad d-2+\alpha&&\quad\text{if}\quad \alpha>0\,;\\
				d-2+\alpha\quad <&\,\,\theta<d-2-(p-1)\alpha&&\quad\text{if}\quad \alpha<0\,.
			\end{alignat*}
			Moreover, $N_1$ in \eqref{2205241155} depends only on $d,\,p,\,\gamma,\,\theta,\,\mathrm{C}_0(\Omega),\,\alpha$ and $M$ (in \eqref{231215620}).
		\end{lemma}
		\begin{proof}
		Observe that $\psi$ is a superharmonic Harnack function, and $\Psi:=\trho^{\,\alpha}$ is a regularization of $\psi$.
		For this $\Psi$, the constants $\mathrm{C}_2(\Psi)$ and $\mathrm{C}_3(\Psi,\psi)$ can be chosen to depend only on $d$, $\alpha$, and $M$.
		In addition, Lemma \ref{21.05.20.3}.(3) implies that for any $p\in(1,\infty)$ and $\gamma,\,\theta\in\bR$ there exists $N=N(d,p,\gamma,\alpha,\mu,M)$ such that $
		\|f\|_{\Psi^{\mu}H_{p,\theta}^{\gamma}(\Omega)}\simeq_N \|f\|_{H_{p,\theta-\alpha\mu}^{\gamma}(\Omega)}$ for all $f\in \cD'(\Omega)$.
		Therefore the proof is concluded by applying Theorem \ref{21.09.29.1} with $\Psi:=\trho^{\,\alpha}$.
		\end{proof}
	 	
		We collect basic properties of classical superharmonic functions, which are used in this section and Section \ref{app2.}. 
		\begin{lemma}\label{21.05.18.1}
			Let $\Omega$ be an open set in $\bR^d$.
			\begin{enumerate}
				\item Let $\phi_1,\,\phi_2$ be classical superharmonic functions in $\Omega$. Then $\phi_1\wedge \phi_2$ is also a classical superharmonic function on $\Omega$.
				
				\item Let $\{\phi_{\alpha}\}$ be a family of positive classical superharmonic functions in $\Omega$.
				Then $\phi:=\inf_{\alpha}\phi_{\alpha}$ is a superharmonic function on $\Omega$.
				
				\item For each $i=1,\,2$, let $\Omega_i$ be an open set in $\bR^d$ and $\phi_i$ be a classical superharmonic function in $\Omega_i$.
				Suppose that 
				\begin{alignat*}{2}
					&\liminf_{x\rightarrow x_1,x\in\Omega_2}\phi_2(x)\geq \phi_1(x_1)\quad &&\text{for all}\quad  x_1\in \Omega_1\cap \partial\Omega_2\,;\\
					&\liminf_{x\rightarrow x_2,x\in\Omega_1}\phi_1(x)\geq \phi_2(x_2)\quad &&\text{for all}\quad  x_2\in \Omega_2\cap \partial\Omega_1\,.
				\end{alignat*}
				Then the function
				\begin{align*}
					\phi(x):=
					\begin{cases}
						\phi_1(x)&\quad x\in\Omega_1\setminus\Omega_2\\
						\phi_1(x)\wedge \phi_2(x) &\quad x\in \Omega_1\cap \Omega_2\\
						\phi_2(x)&\quad x\in\Omega_2\setminus\Omega_1\\
					\end{cases}
				\end{align*}
				is also a classical superharmonic function in $\Omega_1\cup\Omega_2$.
				
			\end{enumerate}
		\end{lemma}
		For Lemma \ref{21.05.18.1}, (1) follows from the definition of classical superharmonic functions; (2) can be found in \cite[Theorem 3.7.5]{AG}; and (3) follows from \cite[Corollary 3.2.4]{AG}.

		\subsection{Domain with fat exterior: Harmonic measure decay property}\label{fatex}
		This subsection begins with a discussion of the relation between condition \eqref{220617253}, classical potential theory, and the Hardy inequality (see Lemma \ref{240320301} and Remark \ref{24032030232}).
		
		We first recall some notions from classical potential theory.
		For a bounded open set $U\subset\bR^d$, $d\geq 2$, and a bounded Borel function $F$ on $\partial U$, the Perron-Wiener-Brelot solution (abbreviated as `PWB solution') of the equation
		\begin{align}\label{2208021143}
			\Delta u=0\quad\text{in}\,\,\,U\quad;\quad u=F\quad\text{on}\,\,\,\partial U
		\end{align}
		is defined as
		\begin{equation}\label{2301051056}
			\begin{aligned}
				u(x):=\inf\Big\{\phi(x)\,:\,\text{$\phi$ is a superharmonic function on $U$ and}\quad &\\
				\text{$\underset{y\rightarrow z,y\in U}{\liminf}\,\phi(y)\geq F(z)$ for all $z\in\partial U$}&\,\Big\}\,.
			\end{aligned}
		\end{equation}
		This $u$ is harmonic on $U$. 
		However, $\lim_{y\rightarrow z}u(y)=F(z)$ does not hold in general even when $F\in C(\partial U)$.
		For basic properties of PWB solutions, we refer readers to \cite{AG}.
		
		For a Borel set $E\subset \partial U$, $w(\,\cdot\,,U,E)$ denotes the PWB solution $u$ of equation \eqref{2208021143} with $F:=1_E$. 
		This $w$ is called the \textit{harmonic measure} of $E$ over $U$.

		We fix an arbitrary open set $\Omega\subset \bR^d$ (not necessarily bounded), $d\geq 2$.
		For $p\in\partial\Omega$ and $r>0$, we denote
		$$
		w(\,\cdot\,,p,r)=w\big(\,\cdot\,,\Omega\cap B_r(p),\Omega\cap \partial B_r(p)\big)
		$$ 
		(see Figure \ref{230113736} below); note that $\Omega\cap \partial B_r(p)$ is a relatively open subset of $\partial\big(\Omega\cap B_r(p)\big)$.

		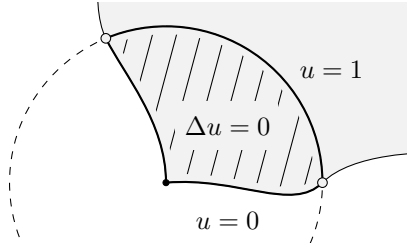
\begin{figure}[h]
			\begin{tikzpicture}[> = Latex]
				\begin{scope}
					\begin{scope}[scale=0.8]
						\clip (-4,-1.5) rectangle (4,2.5);
						
						\begin{scope}[shift={(0.5,0)}]
							\fill[gray!10] (-0.5,-0.5) .. controls +(0,1) and +(0.3,-0.6) ..(-1.5,1.9) .. controls +(-0.3,0.6) and +(0.2,-0.2) .. (-1.7,4) -- (4,4) -- (4,0) .. controls +(-0.2,0) and +(0.5,0.5) .. (2.1,-0.5) .. controls +(-0.5,-0.5) and +(1.2,0.1)..(-0.5,-0.5);
							
						\end{scope}

						\begin{scope}[shift={(0.5,0)}]
							\clip (-0.5,-0.5) .. controls +(0,1) and +(0.3,-0.6) ..(-1.5,1.9) arc (acos(-5/13):0:2.6)  .. controls +(-0.5,-0.5) and +(1.2,0.1)..(-0.5,-0.5);
							\foreach \i in {-4,-3.6,...,3}
							{\draw (\i,-2.8)--(\i+1.5,2.8);}
							\path[fill=gray!10] (-0.35,0.05) rectangle (1.35,0.85);
							
							\draw[gray!10,line width=6pt] (2.1,-0.5) arc (0:acos(-5/13):2.6);
							\draw[gray!10,line width=6pt] (2.1,-0.5) .. controls +(-0.5,-0.5) and  +(1.2,0.1)..	(-0.5,-0.5) .. controls +(0,1) and +(0.3,-0.6) ..(-1.5,1.9);
							
						\end{scope}
						
						\begin{scope}[shift={(0.5,0)}]
							\draw[dashed] (-0.5,-0.5) circle (2.6);
							\draw
							(-0.5,-0.5) .. controls +(0,1) and +(0.3,-0.6) ..(-1.5,1.9) .. controls +(-0.3,0.6) and +(0.2,-0.2) .. (-1.7,4);
							\draw
							(4,0) .. controls +(-0.2,0) and +(0.5,0.5) .. (2.1,-0.5) .. controls +(-0.5,-0.5) and  +(1.2,0.1)..	(-0.5,-0.5);

							\draw[ line width=0.8pt] (2.1,-0.5) arc (0:acos(-5/13):2.6);
							\draw[line width=0.8pt] (2.1,-0.5) .. controls +(-0.5,-0.5) and  +(1.2,0.1)..	(-0.5,-0.5) .. controls +(0,1) and +(0.3,-0.6) ..(-1.5,1.9);
							\draw[fill=gray!10] (2.1,-0.5) circle (0.08);
							\draw[fill=gray!10] (-1.5,1.9) circle (0.08);
							\draw[fill=black] (-0.5,-0.5) circle (0.05);
						\end{scope}

					\end{scope}
					
					\begin{scope}[shift={(0.4,0)}]
						\draw (0.42,0.35) node  {$\Delta u=0$} ;
						\draw (1.8,1.1) node (1) {$u=1$};
						\draw (0.42,-0.9) node (0) {$u=0$};
						
					\end{scope}

				\end{scope}
			\end{tikzpicture}
			\caption{$u:=w(\,\cdot\,,p,r)$}\label{230113736}
		\end{figure}
		
			For convenience, based on Lemma \ref{220621237}, we consider $w(\,\cdot\,,p,r)$ to be continuous in $\Omega\cap \overline{B}(p,r)$ with $w(x,p,r)=1$ for $x\in\Omega\cap \partial B(p,r)$.

		\begin{lemma}\label{220621237}\,\,
			
			\begin{enumerate}
				\item $w(\,\cdot\,,p,r)$ is harmonic in $\Omega\cap B_r(p)$ with values in $[0,1]$.
				
				\item For any $x_0\in \Omega\cap \partial B_r(p)$, $w(x,p,r)\rightarrow 1$ as $x\rightarrow x_0$ from within $x\in\Omega\cap B_r(p)$.
				
				\item For any $0<r<R$ and $N_0\geq 0$, if $w(\cdot,p,R)\leq N_0$ on $\Omega\cap \partial B_r(p)$, then $w(\cdot,p,R)\leq N_0 w(\cdot,p,r)$ in $\Omega\cap B_r(p)$.
			\end{enumerate}
		\end{lemma}
	
	\begin{proof}
	(1) and (2) are the basic properties of $w(\,\cdot\,,p,r)$ which can be found in \cite[Chapter 6]{AG}.
	Therefore we only prove (3).
	
		For convenience, denote $U_R:=\Omega\cap B_R(p)$ and $U_r:=\Omega\cap B_r(p)$, and consider $w(\cdot,p,R):=1_{\Omega\cap \partial B_R(p)}$ on $\partial U_R$.
		It follows from \cite[Theorem 6.3.6]{AG} that $w(x,p,R)|_{U_r}$ is the PWB solution of \eqref{2208021143} for $U:=U_r$ and $F:=w(\cdot,p,R)|_{\partial U_r}$.
		One can observe that 
		$$
		\partial U_r \setminus \big(\Omega\cap \partial B_r(p)\big) \subset (\partial\Omega)\cap B_R(p)\subset \partial U_R\,,
		$$
		which implies that $w(x,p,R)=1_{\Omega\cap \partial B_R(p)}(x)=0$ for $x\in \partial U_r \setminus \big(\Omega\cap \partial B_r(p)\big)$.
		Since $w(x,p,R)\leq N_0$ in $\Omega\cap \partial B_r(p)$, we have $w(\cdot,p,R)|_{\partial U_r}\leq N_01_{\Omega\cap \partial B_r(p)}$.
		By the definition of PWB solution \eqref{2301051056}, $w(\cdot,p,R)\leq N_0 w(\cdot,p,r)$ on $U_r:=\Omega \cap B_r(p)$.
	\end{proof}

		\begin{defn}\label{2301021117}
			A domain $\Omega$ is said to satisfy the \textit{local harmonic measure decay property} with exponent $\alpha>0$ (abbreviated as `$\mathbf{LHMD}(\alpha)$'), if there exists a constant $M_{\alpha}>0$ depending only on $\Omega$ and $\alpha$ such that
			\begin{align}\label{21.08.03.1111}
				w(x,p,r)\leq M_{\alpha}\left(\frac{|x-p|}{r}\right)^{\alpha}\quad\text{for all}\,\,\, x\in \Omega\cap B(p,r)
			\end{align}
			whenever $p\in\partial \Omega$ and $r>0$.
		\end{defn}
		
	\begin{remark}\label{240307323}
	The notion of $\mathbf{LHMD}$ is closely related to the H\"older continuity of PWB solutions.
	Let $\Omega$ be a bounded domain.
	For $F\in C(\partial \Omega)$, we denote by $H_\Omega F$ the PWB solution $u$ of equation \eqref{2208021143} with $U:=\Omega$.
	$H_\Omega F$ is called a classical solution if $\lim_{y\rightarrow z}H_\Omega F(y)=F(z)$ for all $z\in \partial\Omega$.
	Aikawa \cite[Theorem 2, Theorem 3]{aikawa2002} established the following results: Let $0<\alpha<1$.
	\begin{enumerate}
		\item If $H_\Omega F$ is the classical solution for any $F\in C(\partial\Omega)$ and 
		\begin{align}\label{2403051139}
		\sup_{F\in C^{0,\alpha}(\partial\Omega),F\not\equiv 0}\frac{\|H_\Omega F\|_{C^{0,\alpha}(\Omega)}}{\|F\|_{C^{0,\alpha}(\partial\Omega)}}<\infty\,,
		\end{align}
	then $\Omega$ satisfies $\mathbf{LHMD}(\alpha)$.
				
		\item Conversely, if $\Omega$ satisfies $\mathbf{LHMD}(\beta)$ for some $\beta>\alpha$, then $H_\Omega F$ is the classical solution for any $F\in C(\partial\Omega)$, and \eqref{2403051139} holds.
	\end{enumerate}
\end{remark}

		\begin{lemma}\label{22.02.18.4}
			Let $\Omega$ be a bounded domain, and let $\alpha>0$. 
			Suppose that there exist constants $r_0,\,\widetilde{M}\in(0,\infty)$ such that
			\begin{align}\label{21.08.03.1}
				w(x,p,r)\leq \widetilde{M}\left(\frac{|x-p|}{r}\right)^{\alpha}\quad\text{for all}\quad x\in \Omega\cap B(p,r)
			\end{align}
			whenever $p\in\partial\Omega$ and $0<r\leq r_0$.
			Then $\Omega$ satisfies $\mathbf{LHMD}(\alpha)$, where $M_{\alpha}$ in \eqref{21.08.03.1} depends only on $\alpha,\,\widetilde{M}$ and $\mathrm{diam}(\Omega)/r_0$.
		\end{lemma}
	
		\begin{proof}
			Let $p\in\partial\Omega$.
			If $r>\mathrm{diam}(\Omega)$, then $\Omega\cap \partial B(p,r)=\emptyset$, which implies that $w(\,\cdot\,,p,r)\equiv 0$.
			In addition, by the assumption of this lemma, we do not need to pay attention to the case of $r\leq r_0$.
			Therefore, we only consider the case of $r_0<r\leq \mathrm{diam}(\Omega)$.
			
			For $r_0<r\leq \mathrm{diam}(\Omega)$, it follows from Lemmas \ref{220621237}.(1) and (3) that $w(x,p,r)\leq 1$ in general, and $w(x,p,r)\leq w(x,p,r_0)$ if $|x-p|<r_0$.
			By \eqref{21.08.03.1} and that $r_0<r\leq \mathrm{diam}(\Omega)$, we have
			\begin{align*}
			w(x,p,r)\leq \max\big(\widetilde{M},1\big)\left(\frac{\mathrm{diam}(\Omega)}{r_0}\right)^{\alpha}\left(\frac{|x-p|}{r}\right)^{\alpha}\qquad \text{for all}\quad x\in\Omega\cap B(p,r)\,.
			\end{align*}
			The proof is completed.
			\end{proof}
		
		We finally introduce the relation between \eqref{220617253} and the local harmonic measure decay property.

		\begin{lemma}\label{240320301}
			Let $\Omega\subset\bR^d$ be a domain.
			\begin{enumerate}
				\item The following conditions are equivalent:
				\begin{enumerate}
					\item There exists $\epsilon>0$ such that the fat exterior condition \eqref{220617253} holds.
					
					\item There exists $\alpha>0$ such that $\mathbf{LHMD}(\alpha)$ holds.
					
					\item There exists $\epsilon_0>0$ such that
					\begin{align}\label{22.02.22.1}
						\inf_{p\in\partial\Omega,r>0}\frac{\mathrm{Cap}\big(\Omega^c\cap \overline{B}(p,r),B(p,2r)\big)}{\mathrm{Cap}\big( \overline{B}(p,r),B(p,2r)\big)}\geq \epsilon_0>0\,.
					\end{align}
				Here, $\mathrm{Cap}\big(K,B\big)$ is the capacity of a compact set $K\subset B$ relative to an open ball $B$, defined as follows: 
				\begin{align}\label{230324942}
					\mathrm{Cap}(K,B):=\inf\left\{\|\nabla f\|_2^2\,:\,f\in C_c^{\infty}(B)\,\,\,,\,\,\, f\geq 1\,\,\text{on}\,\,K\right\}\,.
				\end{align}
				\end{enumerate}
				In particular, constants $(c,\epsilon)$ in  \eqref{220617253}, $(\alpha,M_\alpha)$ in \eqref{21.08.03.1111}, and $\epsilon_0$ in \eqref{22.02.22.1} depend only on $d$ and on each other.
				
				\item If \eqref{22.02.22.1} holds, then $\Omega$ admits the Hardy inequality \eqref{hardy}, where $\mathrm{C}_0(\Omega)$ depends only on $d$ and $\epsilon_0$.
			\end{enumerate}
		\end{lemma}
		In this lemma, the equivalence between conditions (a) and (c) was established by Lewis \cite[Theorem 1]{lewis} and Aikawa \cite[Theorem B]{aikawa1997} (see, \textit{e.g.}, \cite[Theorem 7.22]{kinnunen2021} for a simplified version).
		Additionally, the equivalence between (b) and (c), as well as Lemma \ref{240320301}(2), were established by Ancona \cite[Lemma 3, Theorem 1]{AA}.
		
		\begin{remark}\label{24032030232}\,
			
		\begin{enumerate}
			\item \eqref{22.02.22.1} is called the capacity density condition.
				For domains $\Omega$ in $\bR^2$, \eqref{22.02.22.1} holds if and only if $\Omega$ admits the Hardy inequality \eqref{hardy} (see Ancona \cite[Theorem 2]{AA}).		
			
			\item A well-known sufficient condition for \eqref{22.02.22.1} is the volume density condition:
			\begin{align*}
				\inf_{p\in\partial\Omega,r>0}\frac{|\Omega^c\cap \overline{B}(p,r)|}{| \overline{B}(p,r)|}\geq \epsilon_1>0
			\end{align*}
			(see, \textit{e.g.}, \cite[Example 6.18]{kinnunen2021}).
			For a deeper discussion of the capacity density condition, we refer readers to \cite{kinnunen2021,Kinhardy,lewis} and the references therein.
		\end{enumerate}
		\end{remark}
		 
		In view of this discussion, we consider domains satisfying $\mathbf{LHMD}(\alpha)$ for some $\alpha>0$, instead of \eqref{220617253}.
		This condition is implied by geometric conditions introduced in Section \ref{app2.}, and the exponent $\alpha$ reflects the corresponding geometric condition; see Theorem \ref{22.02.18.3}.
		In the rest of this subsection, we construct appropriate superharmonic functions related to $\alpha$ (see Lemma \ref{220617557}).
		These constructions will play a crucial role in Section \ref{app2.}.

		\begin{thm}\label{21.11.08.1}
			Let $\Omega$ satisfy $\mathbf{LHMD}(\alpha)$, $\alpha>0$.
			Then for any $\beta\in(0,\alpha)$, there exists a superharmonic function $\phi$ in $\Omega$ such that 
			$$
			N^{-1}\rho(x)^\beta\leq \phi(x)\leq N\rho(x)^\beta
			$$
			for all $x\in\Omega$, where $N=N(\alpha,\beta,M_{\alpha})>0$.
		\end{thm}
		
		Before proving Theorem \ref{21.11.08.1}, we first state the following corollaries:
		
		\begin{thm}\label{22.02.19.3}
			Let $\Omega\subset\bR^d$ satisfy $\mathbf{LHMD}(\alpha)$, $\alpha>0$.
			For any $p\in(1,\infty)$ and $\theta\in\bR$ satisfying
			\begin{align}\label{240127120}
			d-2-(p-1)\alpha<\theta<d-2+\alpha\,,
			\end{align}
			Statement \ref{22.02.19.1} $(\Omega,p,\theta)$ holds.
			In addition, $N_1$ (in \eqref{2205241155}) depends only on $d$, $p$, $\gamma$, $\theta$, $\alpha$, $M_{\alpha}$.
		\end{thm}

		\begin{remark}
		The Poisson equation \eqref{240319417} does not explicitly include any boundary condition. 
		However, Theorem \ref{22.02.19.3} can be interpreted as implicitly enforcing the zero Dirichlet boundary condition $u|_{\partial\Omega}\equiv0$.
		This interpretation is supported by the following observations:
		\begin{itemize}
			\item $C_c^{\infty}(\Omega)$ is dense in $H^{\gamma}_{p,\theta+2p}(\Omega)$.
			\item For $f\in C_c^{\infty}(\Omega)$, the solution $u$ given by Theorem \ref{22.02.19.3} belongs to $H_{p,\theta}^{\gamma+2}(\Omega)$ for any $p\in(1,\infty)$, $\theta$ in \eqref{240127120}, and $\gamma\in\bR$ (see Theorem \ref{220530526}).
		\end{itemize}
		In addition, by taking appropriate $p$, $\theta$, and $\gamma>0$, it follows from Proposition \ref{220512537} that this $u$ is continuous in $\Omega$ and $u\rightarrow 0$ as $\rho(x)\rightarrow 0$.
		\end{remark}

		\begin{proof}[Proof of Theorem \ref{22.02.19.3}]
			Take $\beta\in(0,\alpha)$ such that
			$$
			d-2-(p-1)\beta<\theta<d-2+\beta.
			$$
			It follows from Theorem \ref{21.11.08.1} that there exists a superharmonic function $\phi$ such that $\phi\simeq_N\rho^\beta$, where $N=N(\alpha,\,\beta,\,M_{\alpha})$.
			Lemma \ref{240320301} yields that $\Omega$ admits the Hardy inequality \eqref{hardy}, where $\mathrm{C}_0(\Omega)$ can be chosen to depend only on $d,\alpha$ and $M_{\alpha}$ (in \eqref{21.08.03.1111}). 
			Therefore, the proof is completed by Lemma \ref{220617557}.
		\end{proof}

		\begin{proof}[Proof of Theorem \ref{21.11.08.1}]
			The following construction is a combination of \cite[Theorem 1]{AA} and \cite[Lemma 2.1]{KenigToro}.
			Recall that $M_{\alpha}$ is the constant in \eqref{21.08.03.1111}, and $\beta<\alpha$.
			Take $r_0\in(0,1)$ small enough to satisfy $M_{\alpha} r_0^{\alpha} < r_0^{\beta}$,  and take $\eta\in(0,1)$ small enough to satisfy
			$$
			(1-\eta)M_{\alpha} r_0^{\alpha}+\eta\leq r_0^{\beta}\,.
			$$
			For $w(x,p,r)$, we shall need only the following properties (see Lemma \ref{220621237} and Definition \ref{2301021117}):
			\begin{align*}
					&\text{$w(\cdot,p,r)$ is a classical superharmonic function on $\Omega\cap B_r(p)$}\,;\\
					&\text{$w(\cdot,p,r)=1$ in $\Omega\cap \partial B_r(p)$}\,;\\
					&\text{$0\leq w(\cdot,p,r)\leq M_{\alpha}r_0^{\alpha}$ in $\Omega\cap B(p,r_0r)$}\,.
			\end{align*}
			For $p\in\partial\Omega$ and $k\in\bZ$, put
			$$
			\phi_{p,k}(x)=r_0^{k\beta}\big((1-\eta)\, w(x,p,r_0^k)+\eta\big)\,.
			$$
			Then $\phi_{p,k}$ is a classical superharmonic function on $\Omega\cap B(p,r_0^k)$,
			\begin{alignat*}{2}
				&\quad\,\, \phi_{p,k}\leq  r_0^{(k+1)\beta}\qquad\,\,&&\text{on}\quad \Omega\cap \overline{B}(p,r_0^{k+1})\,,\\
				&\quad\,\, \phi_{p,k}=r_0^{k\beta}\qquad\qquad &&\text{on}\quad \Omega\cap \partial B(p,r_0^k)\,,\\
				&\eta\cdot r_0^{k\beta}\leq \phi_{p,k}\leq r_0^{k\beta}\quad\,\,&&\text{on}\quad \Omega\cap B(p,r_0^k)\,.
			\end{alignat*}
			For $p\in\partial\Omega$ and $x\in\Omega$, define
			$$
			\phi_p(x)=\inf\{\phi_{p,k}(x)\,:\,|x-p|< r_0^{k}\}.
			$$
			
			If we prove the following:
			\begin{align}
				&\text{$\phi_p$ is a classical superharmonic function on $\Omega$ ;}\label{221230124}\\
				&\eta|x-p|^{\beta}\leq \phi_p(x)\leq r_0^{-\beta}|x-p|^{\beta}\,,\label{221230125}
			\end{align}
			then $\phi:=\inf_{p\in\partial\Omega}\phi_p$ is superharmonic in $\Omega$ (see Lemma \ref{21.05.18.1}.(2)) and satisfies
			$$
			\eta \rho(x)^{\beta}\leq \phi(x) \leq r_0^{-\beta}\rho(x)^{\beta}\,.
			$$
			This completes the proof.

			\textbf{-} \eqref{221230124} \textbf{:} 
			We only need to prove that for each $k_0\in\bZ$, $\phi_p$ is a classical superharmonic function on $U_{k_0}:=\{x\in\Omega\,:\,r_0^{k_0+2}<|x-p|<r_0^{k_0}\}$ (see Remark \ref{240316310}).
			For $x\in U_{k_0}$, put
			\begin{align*}
				v_{p,k_0}(x)=
				\begin{cases}
					\phi_{p,k_0}(x) &\text{if}\quad r_0^{k_0+1}\leq |x-p|< r_0^{k_0}\\
					\phi_{p,k_0}(x)\wedge \phi_{p,k_0+1}(x)&\text{if}\quad r_0^{k_0+2}<|x-p|< r_0^{k_0+1}\,.
				\end{cases}
			\end{align*}
			Since $\phi_{p,k_0}\leq \phi_{p,k_0+1}$ in $\Omega\cap \partial B(p,r_0^{k_0+1})$, Lemma \ref{21.05.18.1}.(4) implies that $v_{p,k_0}$ is a classical superharmonic function on $U_{k_0}$.
			Observe that
			$$
			\phi_p(x)=v_{p,k_0}(x)\wedge \inf\{\phi_{p,k}(x)\,:\,k\leq k_0-1\}.
			$$
			Moreover, if $\eta\,r_0^{k\beta}\geq r_0^{k_0\beta}$ then
			$$
			v_{p,k_0}(x)\leq \phi_{p,k_0}(x)\leq r_0^{k_0\beta}\leq \eta\,r_0^{k\beta}\leq \phi_{p,k}(x)\,.
			$$
			Therefore
			\begin{align*}
				\phi_p(x)=v_{p,k_0}(x)\wedge \inf\{\phi_{p,k}(x)\,:\, k\leq k_0-1\quad\text{and}\quad \eta\, r_0^{k\beta}\leq r_0^{k_0\beta}\}\,,
			\end{align*}
			which implies that on $U_{k_0}$, $\phi_p$ is the minimum of finitely many classical superharmonic functions.
			Consequently, by Lemma \ref{21.05.18.1}.(1), $\phi_p$ is a classical superharmonic function on $U_{k_0}$.
			
			\textbf{-} \eqref{221230125} \textbf{:} 
			Let $x\in\Omega$ satisfy $r_0^{k_0+1}\leq |x-p|< r_0^{k_0}$, $k_0\in\bZ$.
			Since
			$$
			\phi_{p,k_0}(x)\leq r_0^{k_0\beta}\,\,,\quad\text{and}\quad \phi_{p,k}(x)\geq \eta r_0^{k\beta}\geq \eta r_0^{k_0\beta}\quad\text{for all}\quad k\leq k_0\,,
			$$
			we obtain that $\eta\,r_0^{k_0\beta}\leq \phi_p(x)\leq r_0^{k_0\beta}$.
			This implies \eqref{221230125}.
		\end{proof}
		
		We end this subsection by providing two corollaries of Theorem \ref{22.02.19.3}.
		
			\begin{corollary}\label{240120415}
			Let $\Omega$ satisfy $\mathbf{LHMD}(\alpha)$ with $\alpha\in(0,1]$, and assume that there exists $M\geq 0$ such that $\int_\Omega\rho(x)^M\dd x=:K_M<\infty$.
			Consider the equation
			\begin{align}\label{231228234}
				\Delta u-\lambda u=f_0+\sum_{i=1}^dD_if_i\quad \text{in}\,\,\,\, \Omega\quad;\quad u=0\quad \text{on}\,\,\,\,\partial\Omega\,,
			\end{align}
			where $f_0,\,f_1,\,\ldots,\,f_d$ are measurable functions in $\Omega$ such that
			\begin{align*}
				|f_0|\lesssim \rho^{-2+\alpha}\,\,,\,\,\,\,|f_1|+\cdots+|f_d|\lesssim \rho^{-1+\alpha}\,.
			\end{align*}
			Then for any $0<\beta<\alpha$, equation \eqref{231228234} has a unique solution $u$ in $C^{0,\beta}(\overline{\Omega})$.
			In addition, we have
			\begin{align}\label{231228253}
				\sup_\Omega\rho^{-\beta}|u|+[u]_{C^{0,\beta}(\Omega)}\lesssim_N \sup_{\Omega}\Big(\rho^{2-\alpha}|f_0|+\rho^{1-\alpha}|f_1|+\cdots +\rho^{1-\alpha}|f_d|\Big)=:N_F\,,
			\end{align}
			where $N$ depends only on $d$, $\alpha$, $M_\alpha$, $\beta$, $M$, and $K_M$.
		\end{corollary}
		
		\begin{proof}
			We first mention that the assumption $\int_\Omega\rho(x)^M\dd x<\infty$ implies that the function $\rho$ is bounded; moreover, $\lim_{|x|\rightarrow \infty}\rho(x)= 0$.
			This implies that if LHS in \eqref{231228253} is finite, then $u\in C^{0,\beta}(\overline{\Omega})$.
			
			\textbf{- Uniqueness of solutions.} 
			If $\Omega$ is bounded, then the uniqueness of solutions directly follows from the maximum principle.
			Consider the case when $\Omega$ is unbounded, and let $u\in C^{0,\beta}(\overline{\Omega})$ satisfies \eqref{231228234} for $f_0=\ldots=f_d\equiv 0$.
			Since $\lim_{|x|\rightarrow \infty}\rho(x)=0$, the conditions on $u$ imply that $\lim_{|x|\rightarrow \infty}u(x)=0$.
			Combining this with the maximum principle, we have
			$$
			\sup_{\Omega}|u|=\lim_{R\rightarrow \infty}\sup_{\Omega\cap B_R}|u|=\lim_{R\rightarrow\infty }\sup_{\partial (\Omega\cap B_R)}|u|=0\,.
			$$
			Therefore, the uniqueness of solutions is follows.
						
			\textbf{- Existence of solutions and \eqref{231228253}.}
			For $\beta\in(0,\alpha)$, put $p:=\frac{d+M}{\alpha-\beta}$ so that $\beta\leq 1-\frac{d}{p}$ and $\theta:=-p\beta$ satisfies \eqref{240127120}.
			Observe that
			\begin{align*}
			\|F\|_{H_{p,\theta+2p}^{-1}(\Omega)}^p&\lesssim_{p,d,\beta} \int_{\Omega}\bigg(|\rho^{2-\beta}f_0|^p+\sum_{i=1}^d|\rho^{1-\beta}f_i|^p\bigg)\rho^{-d}\dd x\\
			&\leq  \big(N_F\big)^p\int_\Omega\rho(x)^M\dd x<\infty\,,
			\end{align*}
			where the last inequality follows from that $p(\alpha-\beta)=d+M$.
			Theorem \ref{22.02.19.3} provides a solution $u\in H_{p,\theta}^1(\Omega)$ to equation \eqref{231228234} with 
			$$
			\|u\|_{H_{p,\theta}^1(\Omega)}\lesssim_{N} \|F\|_{H_{p,\theta+2p}^{-1}(\Omega)}^p\lesssim_{N}\Big(\int_\Omega\rho^M\dd x\Big)^{1/p}N_F\,,
			$$
			where $N=N(d,\alpha,M_\alpha,\beta)$.
			Proposition \ref{220512537} implies $|u|_{0,\beta}^{(-\beta)}\lesssim_{d,\alpha,\beta} \|u\|_{H_{p,\theta}^1(\Omega)}$, and therefore we obtain \eqref{231228253}.
			By the comment at the first in this proof, we have $u\in C^{0,\beta}(\overline{\Omega})$.
		\end{proof}
		
		The following corollary gives an unweighted $L_p$-solvability result when $p$ is close to $2$.
		We note that similar results for various equations can be found in the literature, such as \cite{KilKos1994}, utilizing the reverse H\"older inequality.
		We present Corollary \ref{230210356} here because its proof does not rely on the reverse H"older inequality; instead, it relies on the weighted solvability result (Theorem \ref{22.02.19.3}).
		This theorem also gives an $L_p$ estimate invariant under dilation (see \eqref{230203814}).
		
		We denote by $\mathring{W}_p^1(\Omega)$ the closure of $C_c^{\infty}(\Omega)$ in $W_p^1(\Omega)$.		
		
		\begin{corollary}\label{230210356}
			Let $\Omega$ satisfy \eqref{22.02.22.1} and 
			\begin{equation}\label{2304044112}
				\begin{aligned}
					\lambda\geq 0\quad \text{if}\quad D_{\Omega}<\infty\quad\text{and}\quad  	\lambda> 0\quad \text{if}\quad D_{\Omega}=\infty\,,
				\end{aligned}	
			\end{equation}
			where $D_{\Omega}:=\sup_{x\in\Omega}d(x,\partial\Omega)$.
			Then there exists $\epsilon\in(0,1)$ depending only on $d$, $\epsilon_0$ (in \eqref{22.02.22.1}) such that for any $p\in(2-\epsilon,2+\epsilon)$, the following holds:
			For any $f^0,\,\ldots,\,f^d\in L_p(\Omega)$, equation \eqref{231228234} has a unique solution $u$ in $\mathring{W}^{1}_{p}(\Omega)$.
			Moreover, we have
			\begin{align}\label{230203814}
				\begin{split}
					\|\nabla u\|_{p}+\big(\lambda^{1/2}+D_\Omega^{-1}\big)\|u\|_{p}\lesssim_{d,p,\epsilon_0} \min\big(\lambda^{-1/2},D_\Omega\big)\|f^0\|_{p}+\sum_{i=1}^d\|f^i\|_p\,.
				\end{split}
			\end{align}
		\end{corollary}

		\begin{proof}
			We begin by recalling the following two results under the capacity density condition \eqref{22.02.22.1}:
			\begin{itemize}
				\item[(a)] By Lemma \ref{240320301}.(1), there exists $\alpha\in(0,1)$ such that $\Omega$ satisfies $\mathbf{LHMD}(\alpha)$.
				By Theorem \ref{22.02.19.3}, Statement \ref{22.02.19.1} $(\Omega,p,d-p)$ holds for $p\in(2-\alpha_1,2+\alpha_1)$, and $N_1$ (in \eqref{2205241155}) depends only on $d,\,p,\,\gamma,\,\epsilon_1$.
				
				\item[(b)] It is implied by \cite[Theorem 1, Theorem 2]{lewis} (or see \cite[Theorem 3.7, Corollary 3.11]{Kinhardy}) that there exists $\alpha_2\in(0,1)$ depending only on $d$ and $\epsilon_0$ such that for any $p>2-\alpha_2$,
				\begin{align}\label{230213147}
					\int_{\Omega}\Big|\frac{u(x)}{\rho(x)}\Big|^p\dd x\leq N(d,p,\epsilon_0)\int_{\Omega}|\nabla u|^p\dd x\quad \forall\,\, u\in C_c^{\infty}(\Omega)\,.
				\end{align}
			\end{itemize}
			Put $0<\epsilon<\min(\alpha_1,\alpha_2)$ and consider $p\in (2-\epsilon,2+\epsilon)$.
			
			\textbf{Step 1. Uniqueness of solutions.}
			Since Statement \ref{22.02.19.1} $(\Omega,p,d-p)$ holds, it suffices to show that $\mathring{W}_p^1(\Omega)\subset H_{p,d-p}^1(\Omega)$.
			For any $u\in \mathring{W}_p^1(\Omega)$, there exists a sequence $\{u_n\}_{n\in\bN}\subset C_c^{\infty}(\Omega)$ such that $u_n\rightarrow u$ in $W_p^1(\Omega)$.
			Applying Fatou’s lemma to \eqref{230213147} with $u_n\in C_c^{\infty}(\Omega)$ in place of $u$, we obtain that \eqref{230213147} also holds for $u\in \mathring{W}_p^1(\Omega)$.
			This implies that $u\in H_{p,d-p}^1(\Omega)$.

			\textbf{Step 2. Existence of solutions and estimate \eqref{230203814}.}
				In this Step, we use Lemma \ref{220512433}.(1), $D_{\Omega}^{\,-1}\|u\|_p\leq \|\rho^{-1}u\|_p$, and $\|\rho f\|_p\leq D_{\Omega}\|f\|_p$, without mentioning. Additionally, we also use the fact that
			\begin{align}\label{240112315}
				\|u\|_{H_{p,d}^\gamma(\Omega)}\lesssim_{p,d} \|u\|_{H_{p,d-p}^{\gamma+1}(\Omega)}^{1/2}\|u\|_{H_{p,d+p}^{\gamma-1}(\Omega)}^{1/2}\,,
			\end{align}
			which follows from that
			\begin{align*}
				&\qquad \sum_{n\in\bZ}\ee^{nd}\|\big(\zeta_{0,(n)}u\big)(\ee^n\cdot)\|_{H_p^\gamma}^p\\
				&\lesssim_{p,d}\sum_{n\in\bZ}\ee^{nd}\|\big(\zeta_{0,(n)}u\big)(\ee^n\cdot)\|_{H_p^{\gamma+1}}^{p/2}\|\big(\zeta_{0,(n)}u\big)(\ee^n\cdot)\|_{H_p^{\gamma-1}}^{p/2}\\
				&\leq \left(\sum_{n\in\bZ}\ee^{n(d-p)}\|\big(\zeta_{0,(n)}u\big)(\ee^n\cdot)\|_{H_p^{\gamma+1}}^p\right)^{1/2}\left(\sum_{n\in\bZ}\ee^{n(d+p)}\|\big(\zeta_{0,(n)}u\big)(\ee^n\cdot)\|_{H_p^{\gamma-1}}^p\right)^{1/2}\,.
			\end{align*}

			To prove the existence of solutions, it is enough to find a solution in $L_{p,d}(\Omega)\cap H_{p,d-p}^1(\Omega)$.
			Indeed, $L_{p,d}(\Omega)\cap H_{p,d-p}^1(\Omega)$ is continuously embedded into $W_p^1$, and 
			$C_c^{\infty}(\Omega)$ is dense in $L_{p,d}(\Omega)\cap H_{p,d-p}^1(\Omega)$ (see Lemma \ref{21.09.29.4}.(5)).

			Without loss of generality, we assume that $\lambda=0$ or $\lambda=1$ by dilation.
			Note that $\epsilon_0$ in \eqref{22.02.22.1} is invariant even if $\Omega$ is replaced by $r\Omega=\{rx\,:\,x\in\Omega\}$, for any $r>0$.
			
			\textbf{Step 2.1)} Consider the case $\lambda=1$. 
			Since Statement \ref{22.02.19.1} $(\Omega,p,d-p)$ holds, there exists $v\in H_{p,d-p}^2(\Omega)$ such that $\Delta v- v=\trho^{\,-1}f^0$ and 
			\begin{align}\label{240119538}
			\|v\|_{H_{p,d-p}^2(\Omega)}+\|v\|_{L_{p,d+p}(\Omega)} \lesssim_{d,p,\epsilon_0}\left\|\trho^{-1} f^0\right\|_{L_{p,d+p}(\Omega)}\simeq_{p,d} \|f^0\|_p\,.
			\end{align}
			By \eqref{240112315} and \eqref{240119538}, we have
			\begin{equation}\label{230403637}
				\begin{alignedat}{2}
					&\|v\|_{L_{p,d+p}(\Omega)}+\|v\|_{H_{p,d}^1(\Omega)}	\lesssim_{d,p}\|v\|_{H_{p,d-p}^2(\Omega)}+\|v\|_{L_{p,d+p}(\Omega)} \lesssim_{d,p,\epsilon_0} \|f^0\|_p\,.
				\end{alignedat}
			\end{equation}
			Put 
			$$
			\widetilde{f}:=f^0-\Delta(\trho v)+\trho v=-2\Big[\sum_{i= 1}^dD_i\big(vD_i\trho )\Big]+v\Delta \trho\,,
			$$
			and observe that
			\begin{align*}
				\begin{split}
					\big\|\widetilde{f}\big\|_{H_{p,d+p}^{-1}(\Omega)}					&\lesssim_{d,p}\|v\|_{L_{p,d}(\Omega)}\lesssim_{d,p} \|v\|_{H_{p,d}^1(\Omega)}\lesssim_{d,p,\epsilon_0}\left\|f^0\right\|_p\,,
				\end{split}
			\end{align*}
			where the first and third inequalities follow from Lemma \ref{220512433}.(2) and \eqref{230403637}, respectively.
			Since Statement \ref{22.02.19.1} $(\Omega,p,d-p)$ holds, there exists $w\in H_{p,d-p}^1(\Omega)$ such that
			$$
			\Delta w- w=\sum_{i= 1}^dD_if^i+\widetilde{f}
			$$
			and 
			\begin{align}\label{240122813}
			\|w\|_{H_{p,d-p}^1(\Omega)}+\|w\|_{H^{-1}_{p,d+p}(\Omega)}\lesssim_{d,p,\epsilon_0}\sum_{i= 1}^d\|f^i\|_{L_{p,d}(\Omega)}+\big\|\widetilde{f}\big\|_{H_{p,d+p}^{-1}(\Omega)}\lesssim \sum_{i=0}^d\|f^i\|_{p}\,.
			\end{align}
			Therefore, by \eqref{240112315} and \eqref{240122813}, we have
			\begin{equation}\label{2304041137}
				\begin{alignedat}{2}
					\|w\|_{L_{p,d}(\Omega)}+\|w\|_{H_{p,d-p}^1(\Omega)}	&\lesssim_{d,p}&&\|w\|_{H_{p,d-p}^1(\Omega)}+\|w\|_{H^{-1}_{p,d+p}(\Omega)} \lesssim_{d,p,\epsilon_0} \sum_{i\geq 0}\left\|f^i\right\|_p\,.
				\end{alignedat}
			\end{equation}
			Put $u=v\trho+w$.
			Then $u$ is a solution of equation \eqref{231228234} and satisfies
				\begin{alignat}{2}
					&&&\|u_x\|_p+(1+D_\Omega^{-1})\|u\|_p
					\lesssim_{d,p}\|u\|_{L_{p,d}(\Omega)}+\|u\|_{H_{p,d-p}^1(\Omega)}\label{230204149}\\
					&\lesssim_{d,p}&&\|w\|_{L_{p,d}(\Omega)}+\|w\|_{H_{p,d-p}^1(\Omega)}+\|v\|_{L_{p,d+p}(\Omega)}+\|v\|_{H_{p,d}^1(\Omega)}\lesssim_{d,p,\epsilon_0} \sum_{i\geq 0}\|f^i\|_p\,,\nonumber
				\end{alignat}
		 	where the last inequality follows from \eqref{230403637} and \eqref{2304041137}; note that \eqref{230204149} also implies that $u\in L_{p,d}(\Omega)\cap H_{p,d-p}^1(\Omega)$.
			
			\textbf{Step 2.2)} Consider the case $D_{\Omega}<\infty$, and observe that
			\begin{equation}	\label{230404924}
				\begin{alignedat}{2}
					\Big\|f^0+\sum_{i\geq 1}D_if^i\Big\|_{H_{p,d+p}^{-1}(\Omega)}&\lesssim_{d,p}&&\|f^0\|_{L_{p,d+p}(\Omega)}+\sum_{i\geq 1}\|f^i\|_{L_{p,d}(\Omega)}\\
					&\leq&& D_{\Omega}\|f^0\|_p+\sum_{i\geq 1}\|f^i\|_p<\infty\,.
				\end{alignedat}
			\end{equation}
			Since Statement \ref{22.02.19.1} $(\Omega,p,d-p)$ holds, there exists $\widetilde{u}\in H_{p,d-p}^1(\Omega)$ such that
			$$
			\Delta \widetilde{u}-\lambda \widetilde{u}=f^0+\sum_{i\geq 1}D_if^i\,,
			$$
			and
			\begin{align}\label{230404925}
				\|\widetilde{u}\|_{H_{p,d-p}^1(\Omega)}+\lambda \|\widetilde{u}\|_{H_{p,d+p}^{-1}(\Omega)}\lesssim \|f^0+\sum_{i\geq 1}D_if^i\|_{H_{p,d+p}^{-1}(\Omega)}\,.
			\end{align}
			By \eqref{240112315}, \eqref{230404924}, and \eqref{230404925}, we obtain that
				\begin{align}
					&\|\nabla\widetilde{u}\|_{L_p(\Omega)}+D_\Omega^{-1}\|\widetilde{u}\|_{L_p(\Omega)}+\lambda^{1/2}\|\widetilde{u}\|_{L_p(\Omega)}\label{230204150}\\
					\lesssim_{d,p}\,&\|\widetilde{u}\|_{H_{p,d-p}^1(\Omega)}+\lambda \|\widetilde{u}\|_{H_{p,d+p}^{-1}(\Omega)}\lesssim_{d,p,\epsilon_0} D_{\Omega}\|f^0\|_p+\sum_{i\geq 1}\|f^i\|_p\,.\nonumber
				\end{align}
			By \eqref{230204150}, we have $\widetilde{u}\in L_{p,d}(\Omega)\cap H_{p,d-p}^1(\Omega)$.
			
			\textbf{Step 2.3)}
			The existence of solutions is proved in Steps 2.1 and 2.2, for all $\lambda$ and $D_{\Omega}$ satisfying \eqref{2304044112}.
			For the cases where $D_{\Omega}=\infty$ and $\lambda=1$, and $D_{\Omega}<\infty$ and $\lambda=0$, estimate \eqref{230203814} is proved in \eqref{230204149} and \eqref{230204150}, respectively.
			Therefore, it suffices to prove estimate \eqref{230203814} in the remaining case where $D_{\Omega}<\infty$ and $\lambda=1$.
			Since $u$ in Step 2.1 and $\widetilde{u}$ in Step 2.2 are the same (due to the result in Step 1), \eqref{230203814} follows from \eqref{230204149} and \eqref{230204150}.
		\end{proof}

		\subsection{Domain with thin exterior: Aikawa dimension}\label{0062}
		The notion of the Aikawa dimension was originally introduced by Aikawa \cite{aikawa1991}.
		We recall its definition below.
		For a set $E\subset\bR^d$, the Aikawa dimension of $E$, denoted by $\dim_{\cA}(E)$, is defined as
		$$
		\dim_{\cA}(E)=\inf\Big\{\beta\geq 0\,:\,\sup_{p\in E,\,r>0}\frac{1}{r^\beta}\int_{B_{r}(p)}\frac{1}{d(y,E)^{d-\beta}}\dd y<\infty\Big\}
		$$
		with the convention $\frac{1}{0}=\infty$.
		
		In this subsection, we assume that $d\geq 3$, and that $\Omega$ satisfies
		$$
		\beta_0:=\dim_{\cA}(\Omega^c)<d-2\,.
		$$

		\begin{thm}\label{21.10.18.1}
			For a constant $\beta<d-2$, if there exists a constant $A_{\beta}$ such that
			\begin{align}\label{22.02.08.2}
				\sup_{p\in \Omega^c,\,r>0}\frac{1}{r^\beta}\int_{B_{r}(p)}\frac{1}{d(y,\Omega^c)^{d-\beta}}\dd y\leq A_{\beta}<\infty\,,
			\end{align}
			then the function
			$$
			\phi(x):=\int_{\bR^d}|x-y|^{-d+2}\rho(y)^{-d+\beta}\dd y
			$$
			is a superharmonic function on $\bR^d$ with $-\Delta\phi=N(d)\rho^{-d+\beta}$.
			Moreover, we have
			\begin{align}\label{21.11.10.4}
				N^{-1}\rho(x)^{-d+2+\beta}\leq \phi(x)\leq N \rho(x)^{-d+2+\beta}
			\end{align}
			for all $x\in\Omega$, where $N=N(d,\beta,A_\beta)$.
		\end{thm}
		Before proving Theorem \ref{21.10.18.1}, we first look at the corollary of this theorem.
		
		\begin{corollary}\label{22.02.24.2}
			The Hardy inequality \eqref{hardy} holds in $\Omega$, where $\mathrm{C}_0(\Omega)$ depends only on $d$, $\beta_0$, $\{A_{\beta}\}_{\beta> \beta_0}$.
		\end{corollary}
	
			Actually, this corollary follows from the more general result \cite[Theorem 3]{aikawa1991}, and its proof is based on Muckenhoupt's $A_p$ weight theory.
			If we restrict our attention to Corollary \ref{22.02.24.2}, the result can be proved differently, as follows:
	
		\begin{proof}[Proof of Corollary \ref{22.02.24.2}]
			The following inequality can be found in \cite[Lemma 3.5.1]{BEL}: If $f\in C_c^{\infty}(\bR^d)$ and $s>0$ is a smooth superharmonic function on a neighborhood of $\text{supp}(f)$, then
			\begin{align}\label{220829005500}
				\int_{\bR^d}\frac{-\Delta s}{s}|f|^2 \dd x\leq \int_{\bR^d} |\nabla f|^2\dd x\,.
			\end{align}
			(Its proof is based on integrating $\big|\nabla f -(f/s)\nabla s\big|^2$ and performing integration by parts.)
			Take any $\beta\in(\beta_0,d-2)$, and let $\phi$ be the function in Theorem \ref{21.10.18.1}, so that 
			\begin{align}\label{220913400}
				-\Delta \phi\geq N_1\rho^{-2}\phi>0
			\end{align}
			where $N_1=N(d,\beta,A_{\beta})>0$.
			Fix $f\in C_c^{\infty}(\Omega)$.
			For $0<\epsilon<d\big(\text{supp}(f),\partial\Omega\big)$, let $\phi^{(\epsilon)}$ be the mollification of $\phi$ in \eqref{21.04.23.1}.
			Observe that
			$$
			-\Delta\big(\phi^{(\epsilon)}\big)\geq N_1^{-1}\big(\rho^{-2}\phi\big)^{(\epsilon)}\geq N_1^{-1}(\rho+\epsilon)^{-2}\phi^{(\epsilon)}\quad\text{on}\quad \bR^d\,,
			$$
			where $N_1$ is in \eqref{220913400}.
			By appling the monotone convergence theorem to \eqref{220829005500} with $s=\phi^{(\epsilon)}$(see Lemma \ref{21.04.23.5}.(2)), we obtain \eqref{hardy} with $\mathrm{C}_0(\Omega)=N_1$.
		\end{proof}

		\begin{thm}\label{22.02.19.300}
			For any $p\in(1,\infty)$ and $\theta\in\bR$ satisfying
			$$
			\beta_0<\theta<(d-2-\beta_0)p+\beta_0\,,
			$$
			Statement \ref{22.02.19.1} $(\Omega,p,\theta)$ holds.
			In addition, $N_1$ in \eqref{2205241155} depends only on $d$, $p$, $\gamma$, $\theta$, $\beta_0$, $\{\cA_{\beta}\}_{\beta>\beta_0}$.
		\end{thm} 
	
		\begin{proof}
			Take $\beta\in(\beta_0,d-2)$ satisfying
			$$
			\beta<\theta<(d-2-\beta)p+\beta\,.
			$$
			By Corollary \ref{22.02.24.2} and Theorem \ref{21.10.18.1}, $\Omega$ admits the Hardy inequality \eqref{hardy}, and there exists a superharmonic function $\phi$ satisfying $\phi\simeq \rho^{-d+2+\beta}$.
			Therefore by Lemma \ref{220617557}, the proof is completed.
		\end{proof}
		
		\begin{remark}
		Theorem \ref{22.02.19.300} deals with the Poisson equation in $\Omega\subset \bR^d$, $d\geq 3$.
		Moreover, this theorem can also be interpreted as establishing solvability of the Poisson equation $\Delta u-\lambda u=f$ in $\bR^d$, particularly when $f$ blows up near a set $E$ with $\dim_\cA(E)<d-2$.
		In other words, if $u\in H_{p,\theta}^2(\Omega)$ and $f\in L_{p,\theta+2p}(\Omega)$ satisfy equation \eqref{240319417}, then
		$$
		\int_{\bR^d} u(\Delta \phi-\lambda \phi)\dd x=\int_{\bR^d} f\phi\dd x\quad\text{for all}\quad \phi\in C_c^{\infty}(\bR^d)\,.
		$$
		We leave the proof to the reader, noting that one may use the test functions $\phi_k:=\phi\sum_{|n|\leq k}\zeta_{0,(n)}\in C_c^{\infty}(\Omega)$, where $\zeta_{0,(n)}$ is the function defined in \eqref{230130543}.
		\end{remark}

		\begin{proof}[Proof of Theorem \ref{21.10.18.1}]
			We first prove \eqref{21.11.10.4}.
			For a fixed $x\in \bR^d$, put
			$$
			I_j=\int_{E_j}|x-y|^{-d+2}\rho(y)^{-d+\beta}dy\quad\text{for}\,\,\,\,j=0,\,1,\,\ldots\,,
			$$
			where $E_0:=B\big(x,2^{-1}\rho(x)\big)$ and $E_j:=B\big(x,2^{j-1}\rho(x)\big)\setminus B\big(x,2^{j-2}\rho(x)\big)$ for $j=1,\,2,\,\ldots$.
			Then $\phi(x)=\sum_{j\in\bN_0}I_j$.
			If $y\in E_0$ then $\frac{1}{2}\rho(x)\leq \rho(y)\leq 2 \rho(x)$, which implies 
			\begin{align}\label{2401243591}
			I_0\simeq_{d,\beta} \rho(x)^{-d+\beta}\int_{B(x,\rho(x)/2)}|x-y|^{-d+2}\dd y\simeq_d \rho(x)^{-d+2+\beta}\,.
			\end{align}
			For $I_j$, $j\geq 1$, take $p_x\in \partial\Omega$ such that $|x-p_x|=\rho(x)$, and observe that
			\begin{alignat}{2}
				I_j\,&\lesssim_d&&\, \big(2^{j}\rho(x)\big)^{-d+2}\int_{B(p_x,2^j\rho(x))}\rho(y)^{-d+\beta}\dd y\leq N\big(2^{j}\rho(x)\big)^{-d+2+\beta}\,,\label{240124359}
			\end{alignat}
			where $N=N(d,\beta,A_{\beta})$.
			\eqref{2401243591} and \eqref{240124359} imply \eqref{21.11.10.4}.
			
			To prove that $-\Delta \phi=N(d)\,\phi$ in the sense of distributions, recall that 
			$$
			-\Delta_x\Big(|x-y|^{-d+2}\Big)=N(d)\,\delta_0(x-y)
			$$
			in the sense of distributions, where $\delta_0(\cdot)$ is the Dirac delta distribution.
			By \eqref{22.02.08.2} and $\phi\simeq \rho^{-d+2+\beta}$, $\phi$ is locally integrable in $\bR^d$.
			Therefore we obtain that for any $\zeta\in C_c^{\infty}(\bR^d)$, by the Fubini theorem,
			\begin{align*}
				\int_{\bR^d}\phi(x)\big(-\Delta\zeta\big)(x)\dd x\,&=\int_{\bR^d}\Big(\int_{\bR^d}|x-y|^{-d+2}(-\Delta\zeta)(x)\dd x\Big)\rho(y)^{-d+\beta}\dd y\\
				&=N(d) \int_{\bR^d}\zeta(y)\rho(y)^{-d+\beta}\dd y\,.
			\end{align*}
		\end{proof}

		
		\mysection{Application II - Various domains with fat exterior}\label{app2.}
		
		This section is devoted to results concerning the exterior cone condition, convex domains, and the exterior Reifenberg condition.
		Each of these conditions implies the fat exterior condition.
		
		Throughout this section, we consider a domain $\Omega\subsetneq \bR^d$, $d\geq 2$.

		\subsection{Exterior cone condition and exterior line segment condition}\label{0071}
		
		\begin{defn}[Exterior cone condition]
			For $\delta\in [0,\frac{\pi}2)$ and $R\in (0,\infty]$, a domain $\Omega\subset\bR^d$ is said to satisfy the \textit{exterior $(\delta,R)$-cone condition} if for every $p\in\partial\Omega$, there exists a unit vector $e_p\in\bR^d$ such that
			\begin{align}\label{21.08.03.4}
				\{x\in B_R(p)\,:\,(x-p)\cdot e_p\geq |x-p|\cos\delta\}\subset \Omega^c\,. 
			\end{align}
		\end{defn}
 		
		Note that LHS of \eqref{21.08.03.4} is the result of translating and rotating the set
		$$
		\{x=(x_1,\ldots,x_d)\in B_R(0)\,:\,x_1\geq|x|\cos\delta\}\,.
		$$
		
		The exterior $(0,R)$-cone condition is often referred to as \textit{exterior $R$-line segment condition}, since if $\delta=0$, then LHS in \eqref{21.08.03.4} equals $\{p+re_p\,:\,r\in[0,R)\}$.
		See Figure~\ref{230212745} for illustrations of the exterior cone condition and the exterior line segment condition.
		\begin{figure}[h]\centering
			\begin{tikzpicture}[>=Latex]
				\begin{scope}[shift={(-4.5,0)}]
					\clip (-1.5,-1.5) rectangle (1.5,1.5);
					\begin{scope}[name prefix = p-]
						\coordinate (A) at (0,0.7);
						\coordinate (B) at (-1.4,0.35);
						\coordinate (C) at (0,-1.4);
						\coordinate (D) at (1.4,0.35);
					\end{scope}
					
					\begin{scope}
						\draw[fill=gray!20]
						(p-A) .. controls +(-0.5,{sqrt(3)*0.5}) and +(0,0.7)..
						(p-B) .. controls +(0,-0.7) and +(-0.5,0.5)..
						(p-C) .. controls +(0.5,0.5) and +(0,-0.7) ..
						(p-D) .. controls +(0,0.7) and +(0.5,{sqrt(3)*0.5}) .. (p-A);
					\end{scope}
					
				\end{scope}
				\node[align=center] at (-4.5,-2.34) {A. Lipschitz boundary\\condition};

				\begin{scope}
					\clip (-1.5,-1.7) rectangle (1.5,1.5);
					\begin{scope}[name prefix = p-]
						\coordinate (A) at (0,0.7);
						\coordinate (B) at (-1.4,0.35);
						\coordinate (C) at (0,-1.4);
						\coordinate (D) at (1.4,0.35);
					\end{scope}
					
					\begin{scope}
						\draw[fill=gray!20]
						(p-A) .. controls +(-0.5,{sqrt(3)*0.5}) and +(0,0.7)..
						(p-B) .. controls +(0,-1) and +(0,1.4)..
						(p-C) .. controls +(0,1.4) and +(0,-1) ..
						(p-D) .. controls +(0,0.7) and +(0.5,{sqrt(3)*0.5}) .. (p-A);
					\end{scope}
					\draw[dashed] (p-C) circle (0.25);
					
				\end{scope}
				\node[align=center] at (0,-2.4) {B. Exterior\\$(\frac{\pi}{3},\infty)$-cone condition};
				\node[align=center] at (0, -3.5) {(doesn't satisfy Lipschitz\\boundary condition)};
				
				\begin{scope}[shift={(4.5,0)}]
					\clip (-1.5,-1.7) rectangle (1.5,1.5);
					\begin{scope}[name prefix = p-]
						\coordinate (A) at (0,0.35);
						\coordinate (B) at (-1.4,0.45);
						\coordinate (C) at (0,-1.4);
						\coordinate (D) at (1.4,0.45);
					\end{scope}
					
					\begin{scope}
						\draw[fill=gray!20]
						(p-A) ..controls +(0,1.1) and +(0,0.9) ..
						(p-B) .. controls +(0,-1.1) and +(0,1.4)..
						(p-C) .. controls +(0,1.4) and +(0,-1.1) ..
						(p-D) .. controls +(0,0.9) and +(0,1.1) .. (p-A);
					\end{scope}
					\draw[dashed] (p-A) circle (0.25);
					\draw[dashed] (p-C) circle (0.25);
				\end{scope}
				\node[align=center] at (4.5,-2.37) {C. Exterior $\infty$-line\\segment condition};
				\node[align=center] at (4.5, -3.5) {(doesn't satisfy $(\delta,R)$-cone \\condition, $\forall\,\delta,R>0$)};
				
				%
				%
				%
				%
				
			\end{tikzpicture}
			\caption{Examples for the exterior cone condition}\label{230212745}
		\end{figure}
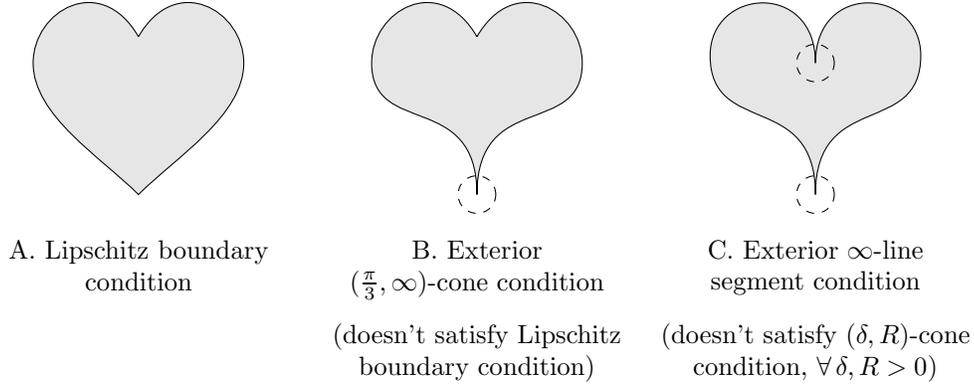

		\begin{example}\label{220816845}
			Suppose that for some constants $K,\,R\in(0,\infty]$, every $p\in\partial\Omega$ admits a function $f_p\in C(\bR^{d-1})$ satisfying
			\begin{align}
				|f_p(y')-f_p(z')|\leq K|y'-z'|\quad\text{for all}\quad y',\,z'\in\bR^{d-1}\,\,\,,\,\,\,\,\,\text{and} \qquad\,\,\, \label{209021056}\\
				\Omega\cap B_R(p)=\big\{y=(y',y_d)\in\bR^{d-1}\times \bR\,:\,y_d>f_p(y')\quad\text{and}\quad |y|<R\big\}\,,\label{209021055}
			\end{align}
			where $(y',y_d)=(y_1,\cdots,y_d)$ in \eqref{209021055} is an orthonormal coordinate system centered at $p$.
			Then $\Omega$ satisfies the exterior $(\delta,R)$-cone condition, where $\delta=\arctan (1/K)\in[0,\pi/2)$.
			
			In addition, if $f\in C(\bR^{d-1})$ satisfies \eqref{209021056} with $f$ in place of $f_p$, then the domain 
			$$
			\big\{(x',x_n)\in\bR^{d-1}\times \bR\,:\,x_n>f(x')\big\}
			$$
			satisfies the exterior $(\delta,\infty)$-cone condition, where $\delta=\arctan(1/K)$.
		\end{example}
		 
		For $\delta\in(0,\pi)$,	let $E_{\delta}:=\{\sigma\in \partial B_1(0)\,:\,\sigma_1>-\cos\delta\}$ (see Figure \ref{2304201145} below).
		By $\Lambda_\delta$, we denote the first Dirichlet eigenvalue of the spherical Laplacian on $E_{\delta}$.
		Alternatively, $\Lambda_{\delta}$ can be expressed as 
		\begin{align}\label{230212803}
			\Lambda_{\delta}=\inf_{f\in F_{\pi-\delta}}\frac{\int_0^{\pi-\delta}|f'(t)|^2(\sin t)^{d-2}  \dd t}{\int_0^{\pi-\delta}|f(t)|^2(\sin t)^{d-2} \dd t}\,,
		\end{align}
		where $F_{\pi-\delta}$ is the set of all non-zero Lipschitz continuous functions $f:[0,\pi-\delta]\rightarrow \bR$ such that $f(\pi-\delta)=0$ (see \cite{FH}).
		We also define
		$$
		\lambda_{\delta}:=-\frac{d-2}{2}+\sqrt{\Big(\frac{d-2}{2}\Big)^2+\Lambda_{\delta}}\,,
		$$
		and when $d=2$, we define $\lambda_0=\frac{1}{2}$.
		
		\begin{figure}[h]
			\begin{tikzpicture}[> = Latex]

				\begin{scope}
					\begin{scope}
						\clip (-0.95,-1.6)--(1.6,-1.6)--(1.6,1.6)--(-0.95,1.6);
						\draw[line width=0.6,fill=gray!40] (0,0) circle (1.5);
					\end{scope}
					\draw[line width=0.6,fill=gray!15] (-0.9*40/41,0) circle [x radius=1.2*9/41, y radius=1.2];
					
					\draw[fill=black] (0,0) circle (0.03);
					\draw[->,line width=0.4] (-2,0)--(2,0);
				\end{scope}

				\begin{scope}
					\draw[dashed] (0,0) -- (-1.2,1.6);
					\draw (-0.4,0.2) node  {$\delta$};
					\draw (-0.3,0) arc(180:180-atan(4/3):0.3);
				\end{scope}

			\end{tikzpicture}
			\caption{$E_{\delta}$}\label{2304201145}
		\end{figure}

		The following quantitative properties of $\Lambda_{\delta}$ and $\lambda_{\delta}$ are given in~\cite{BCG}:
		\begin{prop}\label{230212757}
			Let $\delta\in(0,\pi)$.
			
			\begin{enumerate}
				\item If $d=2$, then $\lambda_{\delta}=\sqrt{\Lambda_{\delta}}=\frac{\pi}{2(\pi-\delta)}>\frac{1}{2}$.
				\item If $d=4$, then $\lambda_{\delta}=-1+\sqrt{1+\Lambda_{\delta}}=\frac{\delta}{\pi-\delta}$.
				\item For $d\geq 3$,
				$$
				\Lambda_{\delta}\geq \left(\int_0^{\pi-\delta}(\sin t)^{-d+2}\Big(\int_0^t(\sin r)^{d-2}\dd r\Big)\,\dd t\right)^{-1}.
				$$
				Moreover, $\Lambda_{\pi/2}= d-1$, $\lim_{\delta\searrow 0}\Lambda_{\delta}=0$, and $\lim_{\delta\nearrow\pi}\Lambda_{\delta}=+\infty$.
			\end{enumerate}
		\end{prop}
		
		Note that when $d=3$, $\Lambda_{\delta}\geq \frac{1}{2}|\log\,\sin\frac{\delta}{2}|^{-1}$.

		\begin{remark}\label{220716637}
			For each $\delta>0$, there exists a function $F\in C\big(\overline{E_{\delta}}\big)\cap C^{\infty}(E_{\delta})$ such that
			$$
			F>0\quad\text{and}\quad   \Delta_\bS F+\Lambda_{\delta}F=0\quad\text{in}\,\,\,\,E_{\delta}\quad;\quad  F=0\quad\text{on}\quad \overline{E_{\delta}}\setminus E_{\delta}
			$$
			(see, \textit{e.g.}, \cite[Section 5]{FH}), where $\Delta_\bS$ is the spherical Laplacian on $\mathbb{S}^{d-1}$.
			Using the representation of the Laplacian on $\bR^d$ in spherical coordinates, 
			$$
			\Delta=D_{rr}+\frac{d-1}{r}D_r+\frac{1}{r^2}\Delta_{\bS}\,,
			$$
			the function
			$v_{\delta}(x):=|x|^{\lambda_{\delta}}F(x/|x|)$
			is a harmonic function on
			$$
			U_{\delta}:=\big\{y\in B_1(0)\,:\,y_1>-|y|\cos\delta \big\}\,,
			$$
			and vanishes on $\partial U_{\delta}\cap B_1(0)$.
		\end{remark}

		With the help of $\lambda_\delta$, we state the main results of this subsection.
		\begin{thm}\label{22.02.18.3}
			Let
			\begin{align*}
				\delta\in[0,\pi/2)\quad\text{if}\,\,\,\, d=2\quad;\quad\delta\in (0,\pi/2)\quad\text{if}\,\,\,\, d\geq 3\,,
			\end{align*}
			and let $\Omega\subset \bR^d$ satisfy the exterior $(\delta,R)$-cone condition, where
			$$
			R\in (0,\infty]\quad\text{if}\,\,\,\,\text{$\Omega$\, is\, bounded}\,\,,\,\,\,\, \text{and}\,\,\quad R=\infty\quad\text{if}\,\,\,\,\text{$\Omega$\, is\, unbounded.}
			$$
			Then $\Omega$ satisfies $\mathbf{LHMD}(\lambda_{\delta})$, where $M_{\lambda_{\delta}}$ in \eqref{21.08.03.1111} depends only on $d$, $\delta$, and also on $\mathrm{diam}(\Omega)/R$ if $\Omega$ is bounded.
		\end{thm}
		
		Before proving Theorem \ref{22.02.18.3}, we present a corollary that directly follows from Theorems \ref{22.02.18.3} and \ref{22.02.19.3}.
		
		\begin{thm}\label{221026914}
			Let $p\in(1,\infty)$.
			Under the same assumption as in Theorem \ref{22.02.18.3}, if $\theta\in\bR$ satisfies
			$$
			-2-(p-1)\lambda_{\delta}<\theta-d<-2+\lambda_{\delta}\,,
			$$
			then Statement \ref{22.02.19.1} $(\Omega,p,\theta)$ holds.
			In addition, $N_1$ in \eqref{2205241155} depends only on $d$, $p$, $\theta$, $\gamma$, $\delta$, and also on $\mathrm{diam}(\Omega)/R$ if $\Omega$ is bounded.
		\end{thm}
		
		To prove Theorem \ref{22.02.18.3}, we use the boundary Harnack principle for Lipschitz domains.

		\begin{lemma}[see Theorem 1 of \cite{Wu}]\label{21.10.18.4}
			Let $D$ be a bounded Lipschitz domain, $A$ be a relatively open subset of $\partial D$, and $U$ be a subdomain of $D$ with $\partial U\cap \partial D\subset A$.
			Then there exists $N=N(D,A,U)>0$ such that if $u, v$ are positive harmonic functions in $D$, and vanish continuously on $A$, then
			$$
			\frac{u(x)}{v(x)}\leq N\frac{u(x_0)}{v(x_0)}\quad\text{for any}\quad x_0,\,x\in U\,.
			$$
		\end{lemma}
		
		\begin{proof}[Proof of Theorem \ref{22.02.18.3}]
			By Lemma \ref{22.02.18.4}, it is sufficient to prove that there exists a constant $M>0$ such that 
			\begin{align*}
				w(x,p,r)\leq M\left(\frac{|x-p|}{r}\right)^{\lambda_{\delta}}\quad\text{for all}\,\,\,x\in\Omega\cap B(p,r)
			\end{align*}
			whenever $p\in\partial\Omega$ and $r\in(0,R)$.
			For any $p\in\partial\Omega$, there exists a unit vector $e_p\in\bR^d$ such that
			$$
			C_{p}:=\{y\in B_R(p)\,:\,(y-p)\cdot e_p\geq |y-p|\cos\delta\}\subset \Omega^c\,.
			$$
			Since
			$$
			\Omega\cap B_r(p)\subset B_r(p)\setminus C_{p}\quad\text{and}\quad \Omega\cap \partial B_r(p)\subset \partial B_r(p)\setminus C_{p}\,,
			$$
			we have
			\begin{align}\label{2301111002}
				w(x,p,r)\leq w\big(x,\,B_r(p)\setminus C_{p}\,,\,\partial B_r(p)\setminus C_{p}\,\big),
			\end{align}
			by directly applying the definition of $w(\cdot,p,r)$ (see \eqref{2301051056}).
			Consider the rotation $T$ satisfying $T(e_p)=(-1,0,\ldots,0)$, and set $T_0(x)=r^{-1}T(x-p)$. 
			Then
			\begin{align}\label{2301111003}
				w\big(x,\,B_r(p)\setminus C_{p}\,,\,\partial B_r(p)\setminus C_{p}\,\big)=w\big(T_0(x),U_{\delta},E_{\delta}\big),
			\end{align}
			where
			$$
			U_{\delta}=\{y\in B_1(0)\,:\,y_1> -|y|\cos\delta\}\,\,\,\,\text{and}\,\,\,\, E_{\delta}=\{y\in \partial B_1(0)\,:\,y_1>- |y|\cos\delta\}\,.
			$$
			By \eqref{2301111002} and \eqref{2301111003}, it is sufficient to show that there exists a constant $M>0$ depending only on $d$ and $\delta$ such that 
			\begin{align}\label{2304131242}
				w(x,U_{\delta},E_{\delta})\leq M|x|^{\lambda_{\delta}}\quad\text{for all}\quad x\in U_{\delta}\,,
			\end{align}
			
			\textbf{Case 1:} $\delta>0$\textbf{.}
			Put $v(x)=|x|^{\lambda_{\delta}}F_0(x/|x|)$
			where $F_0$ is the first Dirichlet eigenfunction of spherical Laplacian on $E_{\delta}\subset \partial B_1(0)$, with $\sup_{E_{\delta}} F_0=1$ (see Remark \ref{220716637}).
			Note that $U_{\delta}$ is a bounded Lipschitz domain, and $w(\,\cdot\,,U_{\delta},E_{\delta})$ and $v$ are positive harmonic functions in $U_{\delta}$, and vanish on $\partial U_{\delta}\cap B_{1}$.
			Applying Lemma \ref{21.10.18.4} with $D=U_{\delta}$, $A=(\partial U_{\delta})\cap B_{1}(0)$, and $U=U_{\delta}\cap B_{1/2}(0)$, we obtain that there exists a constant $N_0=N_0(d,\delta)>0$ such that 
			$$
			w(x,U_{\delta},E_{\delta})\leq N_0 v(x)\leq N_0|x|^{\lambda_{\delta}}\quad\text{for}\quad x\in U_{\delta}\cap B_{1/2}(0)\,.
			$$
			Therefore \eqref{2304131242} is obtained, where $M_0=\max \big(N_0, 2^{\lambda_0}\big)$.
			
			\textbf{Case 2:} $\delta=0$ and $d=2$\textbf{.} 
			We identify $\bR^2$ with $\bC$.
			Note 
			$$
			U_0=\{r\ee^{i\theta}\,:\,r\in(0,1),\,\theta\in(-\pi,\pi)\}\,\,,\,\,\,\, E_0=\{\ee^{i\theta}\,:\,\theta\in(-\pi,\pi)\}\,.
			$$
			Observe that a function $s$ is a classical superharmonic function on $U_0$ if and only if $s(z^2)$ is a classical superharmonic function on $B_{1}(0)\cap \bR_+^2$.
			By the definition of PWB solutions (see \eqref{2301051056}), we have
			$$
			w(z^2,U_0,E_0)=w\big(z,B_1(0)\cap \bR_+^2,\partial B_1(0)\cap \bR_+^2\big)\,.
			$$
			Since the map $z=(z_1,z_2)\mapsto z_1$ is harmonic on $B_1(0)\cap \bR^2_+$, by Lemma \ref{21.10.18.4} with $D=B_1(0)\cap \bR^2_+$, we obtain that
			\begin{align}\label{220906722}
				w\big(z,B_1(0)\cap \bR_+^2,\big(\partial B_1(0)\big)\cap \bR_+^2\big)\leq N |z|\quad\text{for}\,\,\,z\in B_{1/2}(0)\cap \bR_+^2\,,
			\end{align}
			where $N$ is an absolute constant.
			This completes the proof.
		\end{proof}

		\subsection{Convex domains}\label{convex}
		Recall that a set $E\subset \bR^d$ is said to be \textit{convex} if $(1-t)x+ty\in E$ for any $x,\,y\in E$ and $t\in [0,1]$.
		
		\begin{lemma}\label{2209071233}
			For an open set $\Omega\subset \bR^d$, $\Omega$ is convex if and only if for any $p\in\partial\Omega$, there exists a unit vector $e_p\in \bR^d$ such that
			\begin{align}\label{21.08.17.1}
				\Omega\subset \{x\,:\,(x-p)\cdot e_p<0\}=:U_p\,.
			\end{align}
		\end{lemma}
	\begin{proof}
			Let $\Omega$ be a convex domain, and fix $p\in \partial\Omega$.
			Since the set $\{p\}$ is convex and disjoint from $\Omega$, the hyperplane separation theorem (see, \textit{e.g.}, \cite[Theorem 3.4.(a)]{rudin}) implies that there exists a unit vector $e_p\in\bR^d$ such that \eqref{21.08.17.1} holds.
			Conversely, suppose that for any $p\in\partial\Omega$, there exists a unit vector $e_p$ satisfying \eqref{21.08.17.1}.
			Then $E:=\bigcap_{p\in\partial\Omega}U_p$ is convex, $\Omega\subset E$, and $E\cap \partial\Omega=\emptyset$. 
			This implies $E=\Omega$; otherwise, $E\cap\partial\Omega\ne\emptyset$, which is a contradiction.
			This completes the proof.
	\end{proof}

		\begin{thm}\label{22.02.18.3111}
			Let $\Omega\subset \bR^d$ be a convex domain.
			Then $\Omega$ satisfies $\mathbf{LHMD}(1)$, where $M_1$ in \eqref{21.08.03.1111} depends only on $d$.
		\end{thm}
		\begin{proof}		
		The argument used to obtain \eqref{220906722} also implies that for any $d\in\bN$,
		$$
		w\big(x,B_1(0)\cap \bR_+^d,(\partial B_1(0))\cap \bR_+^d\big)\leq N(d) |x|\quad \text{for all}\,\,\,x\in B_1(0)\cap \bR^d_+\,.
		$$
		By translation, dilation, and rotation, we obtain that for any convex domain $\Omega$ and $p\in\partial\Omega$, 
		$$
		w(x,p,r)\leq w\big(x,B_r(p)\cap U_p,\big(\partial B_r(p)\big)\cap U_p\big)\leq N(d)\frac{|x-p|}{r}
		$$
		for all $x\in B_r(p)\cap\Omega$, where $U_p$ denotes the set on RHS of \eqref{21.08.17.1}.
		This completes the proof.
		\end{proof}
		
		This result also implies that the Hardy inequality \eqref{hardy} holds in $\Omega$, where $\mathrm{C}_0(\Omega)$ depends only on $d$ (see Lemma \ref{240320301}).
		It is worth noting that Marcus, Mizel, and Pinchover \cite[Theorem 11]{MMP} proved that for a convex domain $\Omega$, \eqref{hardy} holds with $\mathrm{C}_0(\Omega)=4$, and that this constant cannot be improved.
		
		Combining Theorems \ref{22.02.19.3} and \ref{22.02.18.3111}, we obtain the following result:
		
		\begin{thm}\label{2208131026}
			Let $\Omega\subset \bR^d$ be a convex domain.
			For any $p\in(1,\infty)$ and $\theta\in\bR$ with 
			$$
			-p-1<\theta-d<-1\,,
			$$ 
			Statement \ref{22.02.19.1} $(\Omega,p,\theta)$ holds.
			In addition, $N_1$ in \eqref{2205241155} depends only on $d$, $p$, $\gamma$, $\theta$.
			In particular, $\Omega$ is not necessarily bounded, and $N_1$ is independent of $\Omega$.
		\end{thm}
 		
		\subsection{Exterior Reifenberg condition}\label{ERD}
		The notion of the vanishing Reifenberg condition was introduced by Reifenberg \cite{Reifcondition} and has since been extensively studied in the literature (see, \textit{e.g.}, \cite{Relliptic, CKL, KenigToro3, TT} and Section \ref{0003}.3 of this paper).
		The following definition appears in \cite{Relliptic, KenigToro3}:
		For $\delta\in(0,1)$ and $R>0$, a domain $\Omega\subset \bR^d$ is said to satisfy the $(\delta,R)$-\textit{Reifenberg condition}, if for every $p\in\partial\Omega$ and $r\in(0,R]$, there exists a unit vector $e_{p,r}\in\bR^d$ such that
		\begin{align}\label{22.02.26.41}
			\begin{split}
				&\Omega\cap B_r(p)\subset \{x\in B_r(p)\,:\,(x-p)\cdot e_{p,r}<\delta r\}\quad\text{and}\\
				&\Omega\cap B_r(p)\supset \{x\in B_r(p)\,:\,(x-p)\cdot e_{p,r}>-\delta r\}\,.
			\end{split}
		\end{align}
		In addition, $\Omega$ is said to satisfy the \textit{vanishing Reifenberg condition} if for every $\delta\in(0,1)$, there exists $R_{\delta}>0$ such that $\Omega$ satisfies the $(\delta,R_{\delta})$-Reifenberg condition.
		Note that the vanishing Reifenberg condition is strictly weaker than the $C^1$ boundary condition (see Examples \ref{220910305}.(2) and (3)).
				
		In this subsection, we introduce the totally vanishing exterior Reifenberg condition, which generalizes the vanishing Reifenberg condition.
		We also establish a solvability result for the Poisson equation in domains satisfying the totally vanishing exterior Reifenberg condition (see Theorem \ref{22.07.17.109}).
		
		\begin{defn}[Exterior Reifenberg condition]\label{2209151117}\,\,
			
			\begin{enumerate}
				
				\item
				By $\mathbf{ER}_{\Omega}$ we denote the set of all $(\delta,R)\in[0,1]\times\bR_+$ satisfying the following: For each $p\in\partial\Omega$, and each connected component $\Omega_{p,R}^{(i)}$ of $\Omega\cap B(p,R)$, there exists a unit vector $e_{p,R}^{(i)}\in\bR^d$ such that
				\begin{align}\label{22.02.26.4}
					\Omega_{p,R}^{(i)}\subset \{x\in B_R(p)\,:\,(x-p)\cdot e_{p,R}^{(i)}<\delta R\}\,.
				\end{align}
				By $\delta(R):=\delta_{\Omega}(R)$ we denote the infimum of $\delta$ such that $(\delta,R)\in \mathbf{ER}_{\Omega}$.
				
				\item For $\delta\in[0,1]$, we say that $\Omega$ satisfies the \textit{totally $\delta$-exterior Reifenberg condition} (abbreviate to `$\langle\mathrm{TER}\rangle_\delta$'), if there exist constants $0<R_{0}\leq R_{\infty}<\infty$ such that
				\begin{align}\label{220916111}
					\delta_{\Omega}(R)\leq \delta\quad\text{whenever}\quad R\leq R_0\,\,\,\text{or}\,\,\,R\geq R_\infty\,.
				\end{align}
				
				\item We say that $\Omega$ satisfies the \textit{totally vanishing exterior Reifenberg condition} (abbreviate to `$\langle\mathrm{TVER}\rangle$'), if
				$\Omega$ satisfies $\langle\mathrm{TER}\rangle_\delta$ for all $\delta\in(0,1]$. In other words,
				$$
				\lim_{R\rightarrow 0}\delta_{\Omega}(R)=\lim_{R\rightarrow \infty}\delta_{\Omega}(R)=0\,.
				$$
			\end{enumerate}
		\end{defn}
		 
		The main theorem in this subsection is concerned with domains satisfying $\langle\mathrm{TER}\rangle_\delta$ for sufficiently small $\delta>0$.
		The primary focus, however, lies in the condition $\langle\mathrm{TVER}\rangle$.
		For a comparison between the Reifenberg condition and $\langle\mathrm{TVER}\rangle$, see Figure \ref{230212856} and Example \ref{220910305} below.
			
		\begin{figure}[h]\centering
			\begin{tikzpicture}
				\begin{scope}[shift={(0,4.5)}]
					\begin{scope}[shift={(0,2)}]
						\draw (-5,-1) rectangle (5,0.6);
						\clip (-5,-1) rectangle (5,0.6);
						\draw[decoration={Koch, Koch angle=12, Koch order=4}] 
						decorate {(-5,-0) -- (-2,-0.2)};
						\draw[decoration={Koch, Koch angle=13.5, Koch order=4}] 
						decorate {(1.7,-0.3)--(-2,-0.2)};
						\draw[decoration={Koch, Koch angle=15, Koch order=4}] 
						decorate {(1.7,-0.3) -- (5,-0.6)};
						
					\end{scope}
					
					\begin{scope}[shift={(0,-0.5)}]
						\draw[dashed] (0,0) circle (0.35);
						\draw[->] (0,0.45) -- (0,1.25); 
						\draw (-5,-1) rectangle (5,0.6);
						\clip (-5,-1) rectangle (5,0.6);
						\draw[decoration={Koch, Koch angle=30, Koch order=4}] 
						decorate {(-5,-0.6) -- (-3,-0.3)};
						\draw[decoration={Koch, Koch angle=27.5, Koch order=4}] 
						decorate {(-1.8,-0.1)--(-3,-0.3)};
						\draw[decoration={Koch, Koch angle=21.25, Koch order=4}] 
						decorate {(-1.8,-0.1) -- (-0.3,0)};
						\draw[decoration={Koch, Koch angle=22.5, Koch order=4}] 
						decorate {(1.5,-0.1)--(-0.3,0)};
						\draw[decoration={Koch, Koch angle=28.75, Koch order=4}] 
						decorate {(3.5,-0.3)--(1.5,-0.1)};
						\draw[decoration={Koch, Koch angle=24.75, Koch order=4}] 
						decorate {(3.5,-0.3) -- (5,-0.6)};
						
						
					\end{scope}
					
				\end{scope}
				
				\begin{scope}[scale=0.7, shift={(-6.5,0)}]
					\draw[decorate, decoration={random steps,segment length=2, amplitude=0.5}, fill=gray!20]
					(2,0) arc(0:180:2) .. controls +(0,-1) and +(-0.5,0) .. (-1,-1.5)..controls +(0.5,0) and +(-0.5,0) .. (0,-1).. controls +(0.5,0) and +(-0.5,0) .. (1,-1.5)..controls +(0.5,0) and +(0,-1)..(2,0);
					
					\begin{scope}
						\draw[dashed] (0,2) circle (0.5);
						\draw[->] (0.4,2.6) -- (1,3.8);
							%
							%
							%
					\end{scope}
					
					\node[align=center] at (0,-3.5) {Vanishing\\Reifenberg condition};
				\end{scope}

				\begin{scope}[scale=0.7]
					\begin{scope}
						\draw[decorate, decoration={random steps,segment length=2, amplitude=0.5}, fill=gray!20]
						(2,0) arc(0:180:2);
						\fill[gray!20] (-2,0)--(0,-2) -- (2,0)--(0,2);
						\fill[white]
						(0,-2) ..controls +(0,2) and +(-0.5,0) .. (2,0)--(2,-2);
						\draw (0,-2) ..controls +(0,2) and +(-0.5,0) .. (2,0);
						\fill[white]
						(0,-2) ..controls +(0,2) and +(0.5,0) .. (-2,0)--(-2,-2);
						\draw (0,-2) (0,-2) ..controls +(0,2) and +(0.5,0) .. (-2,0);
					\end{scope}
					\draw[dashed] (0,2) circle (0.5);
					\draw[->] (0,2.8) -- (0,3.8); 
					
					\node[align=center] at (0,-3.5) {$\la \mathrm{TVER}\ra^\ast$\\(Definition \ref{221013228})};
					
				\end{scope}

				\begin{scope}[scale=0.7, shift={(6.5,0)}]
					\begin{scope}
						\draw[decorate, decoration={random steps,segment length=2, amplitude=0.5}, fill=gray!20]
						(2,0) arc(0:180:2);
						\fill[gray!20] (-2,0)--(0,-2) -- (2,0)--(0,2);
						\fill[white]
						(0,-2) ..controls +(0,2) and +(-0.5,0) .. (2,0)--(2,-2);
						\draw (0,-2) ..controls +(0,2) and +(-0.5,0) .. (2,0);
						\fill[white]
						(0,-2) ..controls +(0,2) and +(0.5,0) .. (-2,0)--(-2,-2);
						\draw (0,-2) (0,-2) ..controls +(0,2) and +(0.5,0) .. (-2,0);

						\draw[fill=white, shift={(0,0.5)}] (0,0) arc(0:180:0.5) .. controls +(0,-0.5) and +(-0.5,0) .. (0,0) .. controls +(0.5,0) and +(0,-0.5) .. (1,0) arc(0:180:0.5);
					\end{scope}
					\draw[dashed] (0,2) circle (0.5);
					\draw[->] (-0.4,2.6) -- (-1,3.8); 
					
					\node[align=center] at (0,-3.5) {$\la \mathrm{TVER}\ra$\\(Definition \ref{2209151117})};
					
				\end{scope}
			\end{tikzpicture}
			
			\caption{Totally vanishing exterior Reifenberg condition}\label{230212856}
		\end{figure}
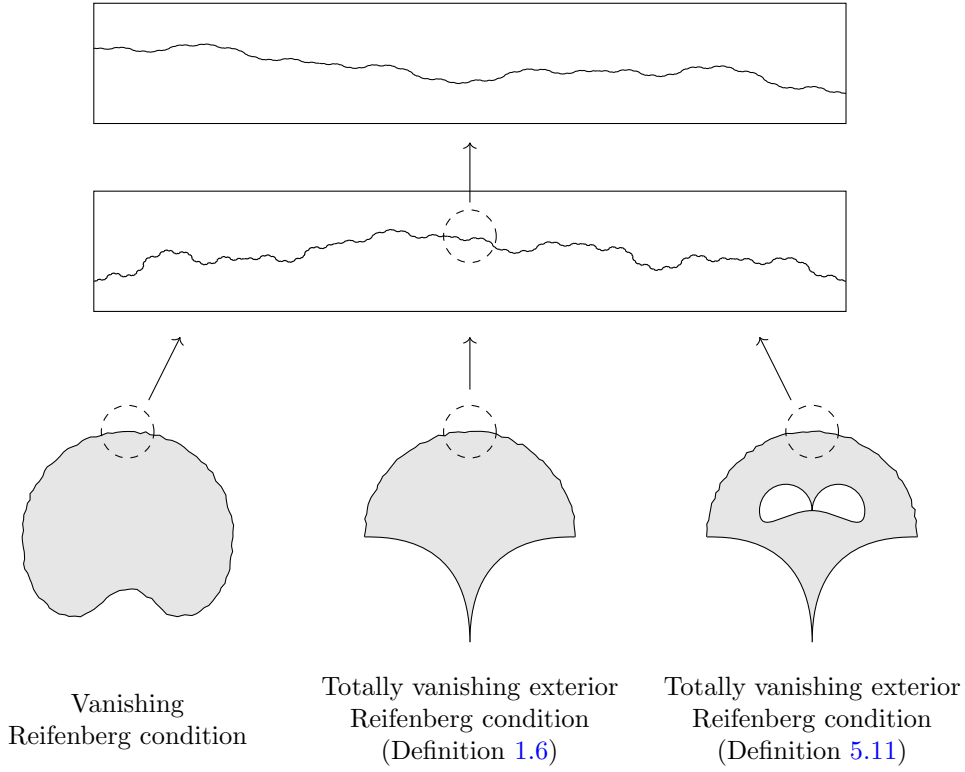
		
		\begin{lemma}\label{220918306}
			For any $R>0$, $\big(\delta(R),R\big)\in\mathbf{ER}_\Omega$. 
		\end{lemma}
		\begin{proof}Take a sequence $\{\delta_n\}_{n\in\bN}$ such that $(\delta_n,R)\in \mathbf{ER}_\Omega$ and $\delta_n\rightarrow \delta(R)$ as $n\rightarrow \infty$.
			Since $(\delta_n,R)\in \mathbf{ER}_\Omega$, for any $p\in\partial\Omega$ and any connected component of $\Omega\cap B(p,R)$, denoted by $\Omega_{p,R}$, there exists a unit vector $e_n$ such that 
			\begin{align}\label{230115252}
				\Omega_{p,R}\subset \{x\in B_R(p)\,:\,(x-p)\cdot e_n<\delta_n R\}\,.
			\end{align}
			Since $\{e_n\}_{n\in\bN}\subset \partial B(0,1)$, there exists a subsequence $\{e_{n_k}\}_{k\in\bN}$ such that $e_p:=\lim_{k\rightarrow \infty}e_{n_k}$ exists in $\partial B(0,1)$.
			It is implied by \eqref{230115252} that 
			\begin{align*}
				\Omega_{p,R}\subset \{x\in B_R(p)\,:\,(x-p)\cdot e_p<\delta(R) R\}\,.
			\end{align*}
			Therefore $\big(\delta(R),R\big)\in \mathbf{ER}_\Omega$.
		\end{proof}

		\begin{example}\label{220910305}
			\,\,
			
			\begin{enumerate}
				\item
				If $\Omega$ satisfies the $(\delta,R_1)$-Reifenberg condition,
				then $\delta(R)\leq \delta$ for all $R\leq R_1$; indeed, the first line of \eqref{22.02.26.41} implies \eqref{22.02.26.4} with $e_{p,r}^{(i)}=e_{p,r}$.
				Moreover, if $\Omega$ is bounded, then Proposition \ref{220918248} implies that $\delta(R)\leq \mathrm{diam}(\Omega)/R$.
				Therefore, if $\Omega$ is a bounded domain satisfying the vanishing Reifenberg condition, then $\Omega$ also satisfies $\langle\mathrm{TVER}\rangle$.
				
				\item 
				By $\lambda_*(\bR^{d-1})$, we denote the little Zygmund class, which is the set of all $f\in C(\bR^{d-1})$ such that
				$$
				\lim_{h\rightarrow 0}\sup_{x\in\bR^{d-1}}\frac{|f(x+h)-2f(x)+f(x-h)|}{|h|}=0\,.
				$$
				For $f\in\lambda_*(\bR^{d-1})$, put
				$$
				\Omega=\{(x',x_d)\in\bR^{d-1}\times \bR\,:\,x_d>f(x')\}\,.
				$$
				Then, as mentioned in \cite[Example 1.4.3]{CKL} (see also \cite[Theorem 6.3]{RA}), $\Omega$ satisfies the vanishing Reifenberg condition, which implies $\lim_{R\rightarrow 0}\delta_{\Omega}(R)=0$.
				Moreover, Proposition \ref{220918248} implies that $\delta(R)\leq \frac{2\|f\|_{C(\bR^{d-1})}}{R}$. (Recall that $\|f\|_{C(\bR^{d-1})}<\infty$.)
				Therefore $\Omega$ satisfies $\langle\mathrm{TVER}\rangle$.
				
				\item Suppose that $\Omega$ is bounded, and for any $p\in\partial\Omega$ there exists $R>0$ and $f\in\lambda_{\ast}(\bR^{d-1})$ such that
				$$
				\Omega\cap B(p,R)=\big\{y=(y',y_n)\in\bR^{d-1}\times \bR\,:\,|y|<R\,\,\,\text{and}\,\,\,y_n>f(y')\big\}\,,
				$$
				where $(y',y_n)$ is an orthonormal coordinate system centered at $p$. Then $\Omega$ satisfies the vanishing Reifenberg condition, and therefore $\Omega$ satisfies $\langle\mathrm{TVER}\rangle$.
				
				\item Let $\Omega$ satisfy the exterior $R_0$-ball condition, \textit{i.e.}, there exists $R_0>0$ such that for any $p\in\partial\Omega$, there exists $q\in\bR^d$ satisfying $|p-q|=R_0$ and $B(q,R_0)\subset \Omega^c$.
				Then $\delta(R)\leq \frac{R}{2R_0}$, and therefore $\lim_{R\rightarrow 0}\delta(R)=0$.
				
				\item If a domain $\Omega$ is an intersection of domains satisfying the totally vanishing Reifenberg condition, then $\Omega$ satisfies $\langle\mathrm{TVER}\rangle$.
			\end{enumerate}
		All of the above examples remain valid if $\langle\mathrm{TVER}\rangle$ is replaced by $\langle\mathrm{TVER}\rangle^\ast$ as defined in Definition \ref{221013228}.
		\end{example}

		A sufficient condition for $\lim_{R\rightarrow \infty}\delta_{\Omega}(R)=0$ is that $\delta_{\Omega}(R)\lesssim 1/R$.
		We characterize domains $\Omega$ satisfying $\delta_{\Omega}(R)\lesssim 1/R$.
		
		\begin{prop}\label{220918248}
			\begin{align*}
				\sup_{R>0}R\,\delta_{\Omega}(R)=\sup_{p\in\partial\Omega}d\big(p,\partial(\Omega_{\mathrm{c.h.}})\big)\,,
			\end{align*}
			where $\Omega_{\mathrm{c.h.}}$ is the convex hull of $\Omega$, \textit{i.e.},
			$$
			\Omega_{\mathrm{c.h.}}:=\big\{(1-t)x+ty\,:\,x,\,y\in\Omega\,\,,\,\,t\in[0,1]\big\}\,.
			$$
		\end{prop}
		
		\begin{remark}
			It follows from the definition of $\delta_{\Omega}(R)$ that $R\delta_{\Omega}(R)$ increases as $R\rightarrow \infty$.
			Therefore, if $\delta_{\Omega}(r_0)>0$ for some $r_0>0$, then $\delta_{\Omega}(R)\gtrsim 1/R$ as $R\rightarrow \infty$.
			As a result, by Proposition \eqref{220918248}, an equivalent condition for $\delta_{\Omega}(R)$ to exhibit minimal nontrivial decay (\textit{i.e.}, $\delta_{\Omega}(R)\simeq 1/R$) is $\sup_{p\in\partial\Omega}d\big(p,\partial(\Omega_{\mathrm{c.h.}})\big)<\infty$.
		\end{remark}
		
		\begin{proof}[Proof of Proposition \ref{220918248}]
			We only need to prove that for any $N_0>0$,
			\begin{align}\label{230104440}
				\sup_{R>0}R\,\delta_{\Omega}(R)\leq N_0\quad  \Longleftrightarrow \quad \sup_{p\in\partial\Omega}d\big(p,\partial(\Omega_{\mathrm{c.h.}})\big)\leq N_0\,.
			\end{align}
			
			\textbf{Step 1.}
			We first claim that LHS of \eqref{230104440} holds
			if and only if for any $p\in\partial\Omega$, there exists a unit vector $e_p$ such that
			\begin{align}\label{220918300}
				\Omega\subset \{x\in\bR^d\,:\,(x-p)\cdot e_p<N_0\}\,.
			\end{align}
			The `if' part is obvious.
			Hence, it remains to prove the `only if' part.
			Assume that LHS of \eqref{230104440} holds.
			Fix $p\in\partial\Omega$, and take $\{\widetilde{\Omega}_n\}_{n\in\bN}$ such that $\widetilde{\Omega}_n$ is a connected component of $\Omega\cap B_n(p)$, and 
			$$
			\widetilde{\Omega}_1\subset \widetilde{\Omega}_2\subset \cdots\,.
			$$
			(More precisely, let $\widetilde{\Omega}_1$ be any connected component of $\Omega\cap B_1(p)$.
			For each $n\in\bN$, since $\widetilde{\Omega}_n$ is a connected open subset of $\Omega\cap B_{n+1}(p)$, there exists a connected component $\widetilde{\Omega}_{n+1}$ of $\Omega\cap B_{n+1}(p)$ that contains $\widetilde{\Omega}_n$.)
			Since $\Omega$ is a domain, $\Omega$ is path connected.
			This implies that
			\begin{align}\label{2209201131}
				\bigcup_{n\in\bN}\widetilde{\Omega}_n=\Omega\,.
			\end{align}
			Since $R\delta(R)\leq N_0$, for each $n\in\bN$, there exists $e_n\in\partial B_1(0)$ such that
			\begin{align}\label{2209201230}
				\widetilde{\Omega}_n\subset \{x\in\bR^d\,:\,(x-p)\cdot e_n<n\delta_\Omega(n)\}\subset \{x\in\bR^d\,:\,(x-p)\cdot e_n<N_0\} 
			\end{align}
		(see Lemma \ref{220918306}).
			Since $\partial B_1(0)$ is compact, there exists a subsequence $\{e_{n_k}\}$ which converges to a certain point, $e_p\in\partial B_1(0)$.
			By \eqref{2209201131} and \eqref{2209201230}, we obtain that \eqref{220918300} holds for this $e_p$.
			
			\textbf{Step 2.}
			By \eqref{230104440}, it suffices to prove the following: For $p\in\partial\Omega$, \eqref{220918300} holds for some $e_p\in\partial B_1(0)$ if and only if $d\big(p,\partial(\Omega_{\mathrm{c.h.}})\big)\leq N_0$.

			To prove the `only if' part, suppose that \eqref{220918300} holds and observe that
			$$
			p\in \partial\Omega\subset\overline{\Omega_{\mathrm{c.h.}}}\subset \{x\in\bR^d\,:\,(x-p)\cdot e_p\leq N_0\}\,.
			$$
			Put
			$\alpha_0:=\sup\{\alpha\geq 0\,:\,p+\alpha e_p\in\overline{\Omega_{\mathrm{c.h.}}}\}$.
			Then $p+\alpha_0e_p\in \partial(\Omega_{\mathrm{c.h.}})$, and therefore $d\big(p,\partial(\Omega_{\mathrm{c.h.}})\big)\leq \alpha_0\leq N_0$.
			
			To prove the `if' part, suppose that there exists $q\in\partial(\Omega_{\mathrm{c.h.}})$ such that
			$$
			|p-q|=d\big(p,\partial(\Omega_{\mathrm{c.h.}})\big)\leq N_0\,.
			$$
			By Lemma \ref{2209071233} and the convexity of $\Omega_{\mathrm{c.h.}}$, there is a unit vector $\widetilde{e}_q$ such that
			$$
			\Omega_{\mathrm{c.h.}}\subset \{x\in\bR^d\,:\,(x-q)\cdot \widetilde{e}_q<0\}\,.
			$$
			This implies that for any $x\in\Omega\subset \Omega_{\mathrm{c.h.}}$, 
			$$
			(x-p)\cdot \widetilde{e}_q<(q-p)\cdot\widetilde{e}_q\leq |p-q|\leq  N_0\,.
			$$
			Therefore \eqref{220918300} holds with $e_p:=\widetilde{e}_q$.
		\end{proof}
		
		\begin{remark}
			It follows from Step 1 in the proof of Proposition \ref{220918248} that the proposition remains valid if we redefine $\delta_{\Omega}(R)$ as the infimum of all $\delta>0$ such that, for every $p\in\partial\Omega$, there exists a unit vector $e_{p,R}$ satisfying \eqref{230203624} with $r=R$.
		\end{remark}
		  
		We now state the main result of this subsection.
		We assume Theorem \ref{231217511} for the moment (it will be proved at the end of this subsection) and proceed to prove Theorem \ref{22.07.17.109}.
		
		\begin{thm}\label{231217511}
				For any $\epsilon\in(0,1)$, there exists $\delta>0$ depending only on $d$ and $\epsilon$ such that, if $\Omega$ satisfies $\langle\mathrm{TER}\rangle_\delta$, then $\Omega$ satisfies $\mathbf{LHMD}(1-\epsilon)$, with $M_{1-\epsilon}$ in \eqref{21.08.03.1111} depending only on $d$, $\epsilon$, $\delta$, and $R_{0}/R_{\infty}$ (where $R_0$ and $R_\infty$ are the constants in \eqref{220916111}).
		\end{thm}

		\begin{thm}\label{22.07.17.109}
			For any $p\in(1,\infty)$ and $\theta\in\bR$ with
			$-p-1<\theta-d<-1$, there exists $\delta>0$, depending only on $d$, $p$, and $\theta$, such that if $\Omega$ satisfies $\langle\mathrm{TER}\rangle_\delta$, then Statement \ref{22.02.19.1},$(\Omega,p,\theta)$ holds.
			In addition, $N_1$ in \eqref{2205241155} depends only on $d$, $p$, $\gamma$, $\theta$, and $R_0/R_\infty$ (where $R_0$ and $R_\infty$ are the constants in \eqref{220916111}).
			In particular, if $\Omega$ satisfies $\langle\mathrm{TVER}\rangle$, then Statement \ref{22.02.19.1} $(\Omega,p,\theta)$ holds for all $p\in(1,\infty)$ and $\theta\in\bR$ with $-p-1<\theta-d<-1$.
			
		\end{thm}
		\begin{proof}
			Take $\epsilon\in(0,1)$ such that
			\begin{align}\label{231217531}
			-p-1+(p-1)\epsilon<\theta-d<-1-\epsilon\,,
			\end{align}
			and put $\delta$ as the constant in Theorem \ref{231217511} for this $\epsilon$.
			Consider a domain $\Omega$ satisfying $\langle\mathrm{TER}\rangle_\delta$.
			By Theorem \ref{231217511}, this $\Omega$ satisfies $\mathbf{LHMD}(1-\epsilon)$.
			Therefore Theorem \ref{22.02.19.3} and \eqref{231217531} imply that Statement \ref{22.02.19.1} $(\Omega,p,\theta)$ holds with $N_1=N(d,p,\gamma,\theta,R_0/R_\infty)$.
		\end{proof}
	
		\begin{remark}
		Kenig and Toro \cite[Lemma 2.1]{KenigToro} proved that, if a bounded domain satisfies the vanishing Reifenberg condition, then it also satisfies $\mathbf{LHMD}(1-\epsilon)$ for all $\epsilon\in(0,1)$.
		\end{remark}
 		
		To prove Theorem \ref{231217511}, we need the following lemma:
		
		\begin{lemma}\label{21.08.24.1}
			If $(\delta,R)\in\mathbf{ER}_{\Omega}$, then there exists a continuous function $w_{p,R}:\Omega\rightarrow (0,1]$ satisfying the following:
			\begin{enumerate}
				\item $w_{p,R}$ is a classical superharmonic function on $\Omega$\,.
				\item $w_{p,R}=1\text{ on }\{x\in\Omega\,:\,|x-p|>(1-\delta)R\}$\,.
				\item $w_{p,R}\leq M\delta$ in $\Omega\cap B(p,\delta R)$\,.
			\end{enumerate}
			Here, $M$ is a constant depending only on $d$ (in particular, independent of $\delta$).
		\end{lemma}
		\begin{proof}[Proof of Lemma \ref{21.08.24.1}]
			If $\delta>1/8$, then by putting $w_{p,R}\equiv 1$ and $M=8$, this lemma is proved.
			Therefore it suffices to consider the case $\delta\leq 1/8$.
			For a fixed $p\in\partial\Omega$, let $\big\{\Omega^{(i)}_{p,R}\big\}$ be the set of all connected components of $\Omega\cap B(p,R)$.
			For each $i$, take a unit vector $e_{p,R}^{(i)}$ satisfying \eqref{22.02.26.4}.
			Put
			\begin{align}\label{230116526}
				q=p+R(\delta+1/4)e^{(i)}_{p,R}\,,
			\end{align}
			so that 
			\begin{align}\label{220920414}
				|p-q|=R(\delta+1/4)\quad\text{and}\quad \Omega_{p,R}^{(i)}\cap B(q,R/4)\neq \emptyset
			\end{align}
			(see Figure \ref{230116528} below).
			
			\begin{figure}[h]
				\begin{tikzpicture}[>=Latex]
					
					\begin{scope}[scale=0.7]
						\clip (0,0) circle (3.3);		
						
						\begin{scope}[rotate=90]	
							\begin{scope}
								\clip[shift={(-0.15,0)}](-0.2,-{sqrt(10)}) .. controls +(0.1,0.5) and +(0,-.5) .. (0.3,-1.5) .. controls +(0,0.5) and +(0.1,-0.5) .. (0.06,-0.6) .. controls +(-0.1,0.5) and +(0,-0.5) .. (0,0) .. controls +(0,0.5) and +(-0.05,-0.5) .. (0.06,0.6).. controls +(0.05,0.5) and +(0,-0.4) .. (0.25,1.5).. controls +(0,+0.4) and +(-0.2,-0.75) .. (.25,{sqrt(10)}) arc({atan(3)}:270:{sqrt(11)});
								\clip (0,0) circle (3);
								\foreach \i in {-7,-6.5,...,0}	
								{\draw (\i,-3.5)--(\i+4,3.5);}
							\end{scope}
							
							\draw[decoration={Koch, Koch angle=20, Koch order=3}] 
							decorate {(0,0) -- (0.06,0.6)};
							\draw[decoration={Koch, Koch angle=15, Koch order=3}] 
							decorate {(0.25,1.5)--(0.06,0.6)};
							\draw[decoration={Koch, Koch angle=20, Koch order=3}] 
							decorate {(0.25,1.5) -- (0.25,{sqrt(10)})};
							\draw[decoration={Koch, Koch angle=24, Koch order=3}] 
							decorate {(0.06,-0.6)--(0,0)};
							\draw[decoration={Koch, Koch angle=20, Koch order=3}] 
							decorate {(0.06,-0.6) -- (0.3,-1.5)};
							\draw[decoration={Koch, Koch angle=15, Koch order=3}] 
							decorate {(0.3,-1.5) -- (-0.2,-{sqrt(10)})};
							
							\draw (0,0) circle ({sqrt(10)});

							\draw[fill=black] (0,0) circle (0.06);
							
							\draw (1.15,0) circle (0.75);
							\draw[fill=black] (1.15,0) circle (0.06);	

						\end{scope}
						

						
					\end{scope}
					
					\begin{scope}

						\fill[white] (-1,-1.2) circle (0.5);
						\draw (-1,-1.2) node {$\Omega_{p,R}^{(i)}$};
						
						\fill[white] (-0.14,-0.25) circle (0.23);
						
						\draw (-0.1,-0.25) node {$p$};
						
						\draw (-0.14,1) node {$q$};
						
						\draw[dashed] (-4,0.28) -- (0,0.28);
						\draw[dashed] (4,0.28) -- (0,0.28);	
						\draw[dashed] (4,0)-- (0,0) ;	
						\draw[->] (3,0)--(3,1.7);
						\draw (3.6,1) node {$e_{p,R}^{(i)}$};
					\end{scope}
				\end{tikzpicture}
				\caption{$q$ and $B(q,R/4)$  in \eqref{230116526}, \eqref{220920414}}\label{230116528}
			\end{figure}
			Put $W^{(i)}(x)=F_0\big(4R^{-1}|x-q|\big)/F_0(2)$, where
			\begin{align}\label{230104814}
			\begin{split}
			F_0(t)=\log(t)\quad\text{if}\quad d=2\quad;\quad F_0(t)=1-t^{2-d}\quad\text{if}\quad d\geq 3\,,
			\end{split}
			\end{align}
			so that $\Delta W^{(i)}=0$ on $\bR^d\setminus \{q\}$.
			Observe that
			\begin{alignat*}{2}
				0\leq W^{(i)}(x)\,&\leq M_0\big(4R^{-1}|x-q|-1\big)\quad&&\text{if}\,\,\,|x-q|\geq R/4\,;\\
				W^{(i)}(x)\,&\geq 1\quad&&\text{if}\,\,\,|x-q|\geq R/2\,,
			\end{alignat*}
			where $M_0$ is a constant depending only on $d$.
			By \eqref{220920414} and that $\delta<\frac{1}{8}$, for $x\in\Omega_{p,R}^{(i)}$,
			\begin{alignat*}{2}
				&\text{if\quad $|x-p|\leq\delta R$}\,\,,\,\,\,&&\text{then\quad$\frac{R}{4}\leq |x-q|\leq \frac{R}{4}+2\delta R$}\,;\\
				&\text{if\quad$|x-p|\geq (1-\delta)R$}\,\,,\,\,\,&&\text{then\quad$|x-q|\geq \frac{(3-8\delta)R}{4}\geq \frac{R}{2}$}\,.
			\end{alignat*}
			Therefore we obtain that
			\begin{alignat*}{2}
				 0\leq\, &W^{(i)}(x)\leq 8M_0\delta\quad&&\text{if}\,\,\,\,|x-p|\leq \delta R\,;\\
				 &W^{(i)}(x)\geq 1\quad &&\text{if}\,\,\,\,|x-p|\geq (1-\delta)R\,.
			\end{alignat*}
			Put
			\begin{align*}
				w_{p,R}(x)=
				\begin{cases}
					W^{(i)}(x)\wedge 1&\text{if}\,\,\,x\in\Omega^{(i)}_{p,R}\\
					1&\text{if}\,\,\,x\in\Omega\setminus B(p,R)\,.
				\end{cases}
			\end{align*}
			Then $w_{p,R}$ is continuous in $\Omega$, and satisfies (2) and (3) of this lemma.
			(1) of this lemma follows from \eqref{230104814} and Lemma \ref{21.05.18.1}.
		\end{proof}
		
		\begin{proof}[Proof of Theorem \ref{231217511}]
		Let $M>0$ be the constant in Lemma \ref{21.08.24.1}.
		For given $\epsilon\in (0,1)$, take small enough $\delta\in(0,1)$ such that $M\delta< \delta^{1-\epsilon}$.
		We assume that $\Omega$ satisfies \eqref{220916111} for this $\delta$.
		Using dilation and Lemma \ref{220918306}, without loss of generality, we assume that $(\delta,R)\in\mathbf{ER}_{\Omega}$ whenever $R\leq \widetilde{R}_0:=R_0/R_\infty\leq 1$ or $R\geq 1$.

		Note that for $(\delta,R)\in\mathbf{ER}_{\Omega}$, by
		Lemma~\ref{21.08.24.1} and the definition of the PWB solution \eqref{2301051056}, $w(\cdot,p,R)\leq M\delta\leq \delta^{1-\epsilon}$ on $\Omega\cap \partial B_{\delta R}(p)$.
		Therefore, by Lemma \ref{220621237}.(3),
		\begin{align}\label{240317424}
		w(\cdot,p,R)\leq \delta^{1-\epsilon}w(\cdot,p,\delta R)\quad\text{on}\quad \Omega\cap  B_{\delta R}(p)\,.
		\end{align}

		The proof is completed by establishing \eqref{21.08.03.1111} for $\alpha:=1-\epsilon$ and $M_{1-\epsilon}$ depending only on $\delta$ and $\widetilde{R}_0$.
		We prove \eqref{21.08.03.1111} by dividing $r$ and $|x-p|$ into the following five cases:
		
		\textbf{Case 1:} $r\leq \widetilde{R}_0$\textbf{.}
		Take $n_0\in\bN_0$ such that $\delta^{n_0+1}r\leq |x-p|< \delta^{n_0}r$.
		Since $(\delta,\delta^kr)\in \mathbf{ER}_{\Omega}$ for all $k\geq 0$, it follows from \eqref{240317424} that
		$$
		w(x,p,r) \leq \delta^{n_0(1-\epsilon)}w(x,p,\delta^{n_0} r)\leq \delta^{n_0(1-\epsilon)}\leq \bigg(\frac{|x-p|}{\delta \,r}\bigg)^{1-\epsilon}\,.
		$$
		
		\textbf{Case 2:} $|x-p|<\widetilde{R}_0< r\leq 1$\textbf{.}		
		By Lemmas \ref{220621237}.(1) and (3) and the result of Case 1, we have
		$$
		w(x,p,r)\leq w(x,p,\widetilde{R}_0)\lesssim_{\delta,\widetilde{R}_0} |x-p|^{1-\epsilon}\leq  \Big(\frac{|x-p|}{ r}\Big)^{1-\epsilon}
		$$
		
		\textbf{Case 3:} $\widetilde{R}_0\leq |x-p|<r\leq 1$\textbf{.}		
		It directly follows that
		$$
		w(x,p,r)\leq 1\leq \bigg(\frac{|x-p|}{\widetilde{R}_0\, r}\bigg)^{1-\epsilon}\,.
		$$
		
		\textbf{Case 4:} $|x-p|<1<r$\textbf{.}
		Take $n_0\in\bN_0$ such that $\delta^{n_0+1}r\leq 1< \delta^{n_0}r$.
		Then $(\delta,\delta^kr)\in \mathbf{ER}_{\Omega}$ for all $k=0,\,1,\,\ldots,\,n_0-1$.
		Therefore we have
		\begin{align*}
		w(x,p,r)\leq \delta^{n_0(1-\epsilon)}w(x,p,\delta^{n_0} r)\leq \delta^{n_0(1-\epsilon)}w(x,p,1)\lesssim_{\delta,\widetilde{R}_0} \delta^{n_0(1-\epsilon)}|x-p|^{1-\epsilon}\,,
		\end{align*}
		where the first inequality follows from \eqref{240317424}, the second follows from Lemmas \ref{220621237}.(1) and (3), and the last follows from the result in Cases 2 and 3.
		Since $\delta^{n_0}\leq  1/(\delta r)$, we have $w(x,p,r)\lesssim_{\delta,\widetilde{R}_0}\big(|x-p|/r\big)^{1-\epsilon}$.
		
		\textbf{Case 5:} $1\leq |x-p|<r$\textbf{.}		
		Take $n_0\in\bN_0$ such that $\delta^{n_0+1}r\leq |x-p|< \delta^{n_0}r$. Since $1<|x-p|$, we have $(\delta,\delta^kr)\in \mathbf{ER}_{\Omega}$ for all $k=0,\,1,\,\ldots,\,n_0-1$.
		This implies that
		$$
		w(x,p,r)\leq \delta^{n_0(1-\epsilon)}w(x,p,\delta^{n_0} r)\leq \delta^{n_0(1-\epsilon)}\leq \bigg(\frac{|x-p|}{\delta \,r}\bigg)^{1-\epsilon}\,.
		$$
		\end{proof}

		\appendix

		\mysection{Auxiliary results}\label{008}

		\begin{lemma}\label{21.04.23.4}
			Let $p\in(1,\infty)$ and $u\in C(\bR^d)$ satisfy \eqref{22.01.25.2}.
			
			\begin{enumerate}
				\item $|u|^{p/2-1}u\in W_2^1(\bR^d)$ and $D_i(|u|^{p/2-1}u)=\frac{p}{2}|u|^{p/2-1}(D_iu)1_{\{u\neq 0\}}$.
				
				\item  $|u|^p\in W_1^2(\bR^d)$ and 
				\begin{align}\label{240123949}
					\begin{split}
					D_i\big(|u|^p\big)\,&=p|u|^{p-2}uD_iu 1_{\{u\neq 0\}}\,\,;\\
					D_{ij}\big(|u|^p\big)\,&=\big(p|u|^{p-2}uD_{ij}u+p(p-1)|u|^{p-2}D_iuD_ju\big)\,1_{\{u\neq 0\}}\,.
					\end{split}	
			\end{align}
			\end{enumerate}
		\end{lemma}
		
		\begin{proof}
			This proof is a variant of \cite[Lemma 2.17]{Krylov1999-1}.
			Take a sequence of nonnegative functions $\{g_n\}_{n\in\bN}\subset C(\bR)$ such that $g_n=0$ on a neighborhood of $0$ for each $n\in\bN$, and $g_n(s) \nearrow |s|^{p/2-1}1_{s\neq 0}$ for all $s\in\bR$.
			Put 
			\begin{align*}
				F_n(t):=\,\,\int^t_0g_n(s)\dd s\quad,\quad 	G_n(t):=\int_0^t\big(g_n(s)\big)^2\dd s\,.
			\end{align*}
			
			Recall assumption \eqref{22.01.25.2}, and denote $A=\sup|u|$.
			Since $0\leq g_n(s)\leq |s|^{p/2-1}$, the Lebesgue dominated convergence theorem implies that $F_n(t)\rightarrow \,\frac{2}{p}\,|t|^{p/2-1}t$ and $G_n(t)\rightarrow  \,\frac{1}{p-1}|t|^{p-2}t$
			uniformly for $t\in[-A,A]$.
			Furthermore, their absolute values increase as $n\rightarrow \infty$.
			Since $F_n\big(u(\,\cdot\,)\big)$ and $G_n\big(u(\,\cdot\,)\big)$ vanish on a neighborhood of $\{u=0\}$, they are supported in a compact subset of $\{u\neq 0\}$, and continuously differentiable with
			$$
			D_i\big(F_n(u)\big)=g_n(u)D_iu\,1_{\{u\neq 0\}}\quad\text{and}\quad D_i\big(G_n(u)\big)=\big(g_n(u)\big)^2D_iu\,1_{\{u\neq 0\}}\,.
			$$
			
			(1) Integrate by parts to obtain
			\begin{align*}
				\int_{\bR^d}|g_n(u)\nabla u \,1_{\{u\neq 0\}}|^2\dd x\,&=-\int_{\bR^d}G_n(u)\Delta u \,1_{\{u\neq 0\}}\dd x\\
				&\leq \frac{1}{p-1}\int_{\{u\neq 0\}}|u|^{p-1}|\Delta u|\dd x\,.
			\end{align*}
			By the monotone convergence theorem, we have $|u|^{p/2-1}|\nabla u|\in L_2(\bR^d)$.
			We denote $v=\frac{2}{p}|u|^{p/2-1}u$.
			For any $\zeta\in C_c^{\infty}(\bR^d)$, we have
			\begin{align*}
				-\int_{\bR^d} v\cdot D_i\zeta \dd x&=-\lim_{n\rightarrow\infty}\int_{\bR^d} F_n(u)\cdot D_i\zeta\dd x\\
				&=\lim_{n\rightarrow\infty}\int_{\{u\neq 0\}} g_n(u)D_iu\cdot \zeta\dd x=\int_{\{u\neq 0\}} |u|^{p/2-1}D_iu\cdot \zeta \dd x\,.
			\end{align*}
			Here, the first and the last equalities follow from the Lebesgue dominated convergence theorem, because 
			$|F_n(u)|\leq |v|$ and $|g_n(u)D_i u|\leq |u|^{p/2-1}|\nabla u|\in L_2(\bR^d)$.
			Therefore $v\in W_2^1(\bR^d)$ and $D_i v=|u|^{p/2-1}D_i u\,1_{\{u\neq 0\}}$.
			
			(2) It follows from (1) of this lemma that $|u|^p\in W_1^1(\bR^d)$ with $D_i\big(|u|^p\big)=p|u|^{p-2}u D_iu 1_{u\neq 0}$.
			For any $\zeta\in C_c^{\infty}(\Omega)$, we have
			\begin{align*}
				&\frac{1}{p-1}\int_{\{u\neq 0\}}|u|^{p-2}uD_iu \cdot D_j\zeta \dd x\\
				=&\,\lim_{n\rightarrow \infty}\int_{\bR^d} G_n(u)D_iu \cdot D_j\zeta \dd x\\
				=&\,-\lim_{n\rightarrow \infty}\int_{\bR^d} \Big(|g_n(u)|^2 D_iu D_ju+G_n(u)D_{ij}u\Big)\zeta \dd x\\
				=&\,-\int_{\{u\neq 0\}} \Big(|u|^{p-2}D_i u D_ju+\frac{1}{p-1}|u|^{p-2}u D_{ij}u1_{\{u\neq 0\}}\Big)\zeta \dd x\,.
			\end{align*}
			Here, the first and last equalities follow from the Lebesgue dominated convergence theorem, since $|G_n(u)|\leq \frac{1}{p-1}|u|^{p-1}$ and $|g_n(u)|\leq |u|^{p/2-1}$ (see the assumption on $u$ and (1) of this lemma). 
			Therefore $|u|^p\in W_1^2(\bR^d)$ with \eqref{240123949}.
\end{proof}

		\begin{lemma}\label{22.02.16.1}
	There exist linear maps 
	$$
	\Lambda_0\,:\,\Psi H_{p,\theta}^{\gamma}\rightarrow \Psi H_{p,\theta}^{\gamma+1}(\Omega)\quad\text{and}\quad \,\Lambda_1,\,\ldots,\,\Lambda_d:\Psi H_{p,\theta}^{\gamma}\rightarrow \Psi H_{p,\theta-p}^{\gamma+1}(\Omega)
	$$
	such that for any $f\in \Psi H_{p,\theta}^{\gamma}(\Omega)$, $f=\Lambda_0 f+\sum_{i=1}^dD_i(\Lambda_if)$ and 
	\begin{align}\label{2301021031}
		\begin{gathered}
			\|\Lambda_0 f\|_{\Psi H_{p,\theta}^{\gamma+1}(\Omega)}+\sum_{i=1}^d\|\Lambda_i f\|_{\Psi H_{p,\theta-p}^{\gamma+1}(\Omega)}\leq N\|f\|_{\Psi H_{p,\theta}^{\gamma}(\Omega)}\,,
		\end{gathered}
	\end{align}	
where $N=N(d,p,\gamma,\theta,\mathrm{C}_2(\Psi))$.
\end{lemma}
\begin{proof}
	\textbf{Step 1.} We first prove the case $\Psi\equiv 1$.
	Consider linear operators from $H_{p}^{\gamma}$ to $H_{p}^{\gamma+1}$ defined by $L_0:=(1-\Delta)^{-1}$ and $L_i:=-D_i(1-\Delta)^{-1}$ for $i=1,\,\ldots,\,d$.
	They satisfy that for any $g\in H_p^{\gamma}$,
	\begin{align}\label{2205101114}
		L_0g+\sum_{i=1}^dD_iL_ig=g\quad\text{and}\quad 
		\sum_{i=0}^d\|L_ig\|_{H_p^{\gamma+1}}\lesssim_{d,p,\gamma}\|g\|_{H_p^{\gamma}}\,.
	\end{align}
	We denote $\zeta_1(t)=\zeta_0(\ee^{-1}t)+\zeta_0(t)+\zeta_0(\ee t)$ and $\zeta_{1,(n)}(x):=\zeta_1\big(\ee^{-n}\trho(x)\big)$.
	Put
	\begin{align*}
		\widetilde{\Lambda}_0f(x):=\,&\sum_{n\in\bZ}\zeta_{1,(n)}(x)\, L_0\Big[\big(\zeta_{0,(n)}f\big)(\ee^n\cdot)\Big](\ee^{-n}x)\\
		&-\sum_{k=1}^d\sum_{n\in\bZ}\ee^n\big(D_k\zeta_{1,(n)}\big)(x)\, L_k\Big[\big(\zeta_{0,(n)}f\big)(\ee^n\cdot)\Big](\ee^{-n}x)\,,\\
		\widetilde{\Lambda}_if(x):=\,&\sum_{n\in\bZ}\ee^n\zeta_{1,(n)}(x)\, L_i\Big[\big(\zeta_{0,(n)}f\big)(\ee^n\cdot)\Big](\ee^{-n}x)\,,
	\end{align*}
	for $i=1,\,\ldots,\,d$.
	By \eqref{2205101114}, we have 
	\begin{align*}
		\widetilde{\Lambda}_0f+\sum_{i=1}^dD_i\widetilde{\Lambda}_if=\,&\sum_{n\in\bZ}\left(\zeta_{1,(n)}(\,\cdot\,)\times\Big[\big(L_0+\sum_{i=1}^dD_iL_i\big)\big[(\zeta_{0,(n)}f)(\ee^n\cdot)\big]\Big](\ee^{-n}\,\cdot\,)\right)\\
		=\,&\sum_{n\in\bZ}\Big[\zeta_{1,(n)}\zeta_{0,(n)}f\Big]=\sum_{n\in\bZ}\zeta_{0,(n)}f=f\,.
	\end{align*}
	In addition, we also obtain
	\begin{align}
			&\big\|\big(\zeta_{0,(n)}\widetilde{\Lambda}_0f\big)(\ee^n\cdot)\big\|_{H_{p}^{\gamma+1}}^p+\sum_{i=1}^d\ee^{-np}\big\|\big(\zeta_{0,(n)}\widetilde{\Lambda}_if\big)(\ee^n\cdot)\big\|_{H_{p}^{\gamma+1}}^p\label{2401021036}\\
			\lesssim_N&\sum_{i=0}^d\sum_{|k|\leq 2}\Big\|\big(\zeta_{0,(n)} \zeta_{1,(n+k)}\big)(\ee^{n}\cdot)\times  L_i\Big[\big(\zeta_{0,(n+k)}f\big)(\ee^{n+k}\cdot)\Big](\ee^{-k}\cdot)\Big\|_{H_p^{\gamma+1}}^p\nonumber\\
			&+\ee^{np}\sum_{i=0}^d\sum_{|k|\leq 2}\Big\|\big(\zeta_{0,(n)} D\zeta_{1,(n+k)}\big)(\ee^n\cdot)\times L_i\Big[\big(\zeta_{0,(n+k)}f\big)(\ee^{n+k}\cdot)\Big](\ee^{-k}\cdot)\Big\|_{H_p^{\gamma+1}}^p\nonumber\\
			\lesssim_N&\sum_{|k|\leq 2}\sum_{i=0}^d\Big\|L_i\Big[\big(\zeta_{0,(n+k)}f\big)(\ee^{n+k}\cdot)\Big]\Big\|_{H_p^{\gamma+1}}^p\lesssim_{d,p,\gamma}\sum_{|k|\leq 2}\big\|\big(\zeta_{0,(n+k)}f\big)(\ee^{n+k}\cdot)\big\|_{H_{p}^\gamma}^p\,,\nonumber
	\end{align}
	Here, the first and second inequalities follow from that
	\begin{align*}
	\big\|\big(\zeta_{0,(n)} \zeta_{1,(n+k)}\big)(\ee^{n}\cdot)\big\|_{C^m(\bR^d)}+\ee^n\big\|\big(\zeta_{0,(n)} \cdot D\zeta_{1,(n+k)}\big)(\ee^n\cdot)\big\|_{C^m(\bR^d)}\leq N(d,k,m)\,,
	\end{align*}
	where $N(d,k,l)=0$ if $|k|\geq 3$ (see \eqref{230130542} and recall that $\text{supp}(\zeta_1)\subset [\ee^{-2},\ee^2]$).
	\eqref{2401021036} implies \eqref{2301021031} for $\Psi\equiv 1$ and $\widetilde{\Lambda}_i$ instead of $\Lambda_i$.
	
	\textbf{Step 2.} 
	For $f\in \Psi H_{p,\theta}^{\gamma}$ (equivalently, $\Psi^{-1}f\in H_{p,\theta}^{\gamma}(\Omega)$), put
	\begin{align*}
		\Lambda_0f=\Psi \widetilde{\Lambda}_0(\Psi^{-1}f)-\sum_{k=1}^d\big(D_k\Psi\big)\cdot \widetilde{\Lambda}_k(\Psi^{-1}f)\quad;\quad \Lambda_if=\Psi \widetilde{\Lambda}_i(\Psi^{-1}f)\,.
	\end{align*}
	for $i=1,\,\cdots,\,d$.
	Then we have
	$$
	\Big(\Lambda_0+\sum_{i=1}^dD_i\Lambda_i\Big)f=\Psi\Big(\widetilde{\Lambda}_0+\sum_{i=1}^dD_i\widetilde{\Lambda}_i\Big)\big(\Psi^{-1}f\big)=f\,.
	$$
	Moreover, Lemma \ref{21.09.29.4}.(3) and \eqref{2401021036} imply that
	\begin{align*}
		&\|\Psi^{-1}\Lambda_0f\|_{H_{p,\theta}^{\gamma+1}(\Omega)}+\sum_{i=1}^d\|\Psi^{-1}\Lambda_if\|_{H_{p,\theta-p}^{\gamma+1}(\Omega)}\\
		\lesssim\,\,&\|\widetilde{\Lambda}_0(\Psi^{-1} f)\|_{H_{p,\theta}^{\gamma+1}(\Omega)}+\sum_{i=1}^d\|\widetilde{\Lambda}_i(\Psi^{-1}f)\|_{H_{p,\theta-p}^{\gamma+1}(\Omega)}\lesssim\|\Psi^{-1}f\|_{H_{p,\theta}^{\gamma}(\Omega)}\,.
	\end{align*}
	Therefore, the proof is completed.
\end{proof}

	\begin{lemma}\label{22.04.11.3}
	Let $\eta\in C_c^{\infty}(\bR^d)$ satisfy $\eta=1$ on $B_1(0)$ and $\mathrm{supp}(\eta)\subset B_2(0)$.
	For each $i\in\bN$, let $N(i)\in\bN$ be a constant satisfying
	$$
	\mathrm{supp}\Big(\sum_{|n|\leq i}\zeta_{0,(n)}\Big)\subset \big\{x\in\Omega\,:\,\big(N(i)/2\big)^{-1}\leq \rho(x)\leq N(i)/2\big\}\,.
	$$
	Let $\Lambda_i$, $\Lambda_{i,j}$, and $\Lambda_{i,j,k}$ be linear functionals on $\cD'(\Omega)$ defined as
	\begin{align*}
		\Lambda_{i}f:=\Big(\sum_{|n|\leq i}\zeta_{0,(n)}\Big)f\,\,\,,\,\,\,\,\Lambda_{i,j}f=\eta(j^{-1}\cdot\,)\Lambda_if\,\,\,,\,\,\,\,\Lambda_{i,j,k}f=\big(\Lambda_{i,j}f\big)^{(N(i)^{-1}k^{-1})}\,,
	\end{align*}
	where $\big(\Lambda_{i,j}f\big)^{(\epsilon)}$ is defined in the same way as in \eqref{21.04.23.1}.
	Then for any $p\in(1,\infty)$, $\gamma$,\,$\theta\in\bR$, and regular Harnack function $\Psi$, the following hold:
	\begin{enumerate}
		\item For any $f\in \cD'(\Omega)$, $\Lambda_{i,j,k}f\in C_c^{\infty}(\Omega)$\,.
		
		\item  For any $f\in \Psi H_{p,\theta}^{\gamma}(\Omega)$,
		\begin{align}
			&\sup_i\|\Lambda_if\|_{\Psi H_{p,\theta}^{\gamma}(\Omega)}\leq N_1\|f\|_{\Psi	H_{p,\theta}^{\gamma}(\Omega)}\nonumber\\
			&\,\sup_j\|\Lambda_{i,j}f\|_{\Psi H_{p,\theta}^{\gamma}(\Omega)}\leq N_2\|f\|_{\Psi H_{p,\theta}^{\gamma}(\Omega)}\label{2401021134}\\
			&\,\sup_k\|\Lambda_{i,j,k}f\|_{\Psi H_{p,\theta}^{\gamma}(\Omega)}\leq N_3\|f\|_{\Psi H_{p,\theta}^{\gamma}(\Omega)}\,,\nonumber
		\end{align}
		where $N_1$, $N_2$, $N_3$ are constants independent of $f$.
		
		\item For any $f\in \Psi H_{p,\theta}^{\gamma}(\Omega)$,
		\begin{align}\label{2401021135}
			\lim_{k\rightarrow \infty}\Lambda_{i,j,k}f=\Lambda_{i,j}f\,\,\,,\,\,\,\lim_{j\rightarrow \infty}\Lambda_{i,j}f=\Lambda_{i}f\,\,\,,\,\,\,\lim_{i\rightarrow \infty}\Lambda_i f=f\quad\text{in}\quad \Psi H_{p,\theta}^{\gamma}(\Omega)\,.
		\end{align}
	\end{enumerate}
\end{lemma}
\begin{proof}
	(1) It follows directly from the properties of distributions.
	
	(2), (3) \textbf{Step 1: $\Lambda_i$.} 
	Let $f\in H_{p,\theta}^{\gamma}(\Omega)$.
	From \eqref{230130542}, one can observe that
	\begin{align*}
		\|f-\Lambda_i f\|_{\Psi H_{p,\theta}^{\gamma}(\Omega)}^p\lesssim_N \sum_{|n|\geq i-1}\ee^{n\theta}\big\|\big(\Psi^{-1}f\zeta_{0,(n)}\big)(\ee^n\cdot)\big\|_{H_p^{\gamma}}^p\leq \|f\|_{\Psi H_{p,\theta}^{\gamma}(\Omega)}^p\,,
	\end{align*}
	where $N=N(d,p,\gamma,\theta)$.
	Therefore we have
	$$
	\sup_i\|\Lambda_if\|_{\Psi H_{p,\theta}^{\gamma}(\Omega)}\leq N\|f\|_{\Psi H_{p,\theta}^{\gamma}(\Omega)}\quad\text{and}\quad \lim_{i\rightarrow\infty}\|f-\Lambda_if\|_{\Psi H_{p,\theta}^{\gamma}(\Omega)}=0\,.
	$$
	
	\textbf{Step 2: $\Lambda_{i,j}$.} 
	The definition of $H_{p,\theta}^\gamma(\Omega)$ implies that for any $A>1$, if $F\in \cD'(\Omega)$ or $F\in\cD'(\bR^d)$, and $F$ is supported in $\{x\in\Omega\,:\,A^{-1}\leq \rho(x)\leq A\}$, then
	\begin{align}\label{2401021132}
		\|F\|_{H_{p,\theta}^{\gamma}(\Omega)}\simeq_N\|F\|_{H_{p}^\gamma}\,,
	\end{align}
	where $N=N(d,p,\theta,\gamma,A)$.
	For each $i\in\bN$, $\Psi^{-1}\Lambda_if$ and $\Psi^{-1}\Lambda_{i,j}f$ are supported in
	$$
	\big\{x\in\Omega\,:\,N(i)^{-1}\leq \rho(x)\leq N(i)\big\}\,.
	$$
	Therefore $\Psi^{-1}\Lambda_i f\in H_p^\gamma$.
	Since $\Psi^{-1}\Lambda_{i,j}f=\eta(j^{-1}\,\cdot)\Psi^{-1}\Lambda_{i}f$, we obtain that 
	\begin{align*}
		\begin{gathered}
			\lim_{j\rightarrow \infty}\big\|\Psi^{-1}\Lambda_i f-\Psi^{-1}\Lambda_{i,j}f\big\|_{H_{p}^{\gamma}}= 0\quad \text{and}\quad 
			\big\|\Psi^{-1}\Lambda_{i,j}f\big\|_{H_{p}^{\gamma}}\lesssim_{N_2}\big\|\Psi^{-1}\Lambda_{i}f\big\|_{H_{p}^{\gamma}}\,,
		\end{gathered}
	\end{align*}
	where $N_2=N(d,p,\gamma,\theta,i,\eta)$.
	Thus, \eqref{2401021134} and \eqref{2401021135} for $\Lambda_{i,j}$ follow from \eqref{2401021132}.

	\textbf{Step 3: $\Lambda_{i,j,k}$.} 
	Put
	$$
	K_{i,j}=\{x\in\Omega\,:\,N(i)^{-1}\leq \rho(x)\leq N(i)\,,\,\,|x|\leq 2j\}\,,
	$$
	which is a compact subset of $\Omega$, and $\Lambda_{i,j}f$ and $\Lambda_{i,j,k}f$ are supported in there.
	Since $\Psi$ and $\Psi^{-1}$ belong to $C^{\infty}(\Omega)$, we obtain that
	\begin{align*}
		\|\Lambda_{i,j}f\|_{\Psi H_{p,\theta}^{\gamma}(\Omega)}:=\|\Psi^{-1} \Lambda_{i,j}f\|_{H_{p,\theta}^{\gamma}(\Omega)}\simeq_N \|\Psi^{-1} \Lambda_{i,j}f\|_{H_{p}^{\gamma}}\simeq_N\|\Lambda_{i,j}f\|_{H_{p}^{\gamma}},
	\end{align*}
	where $N=N(d,p,\gamma,\theta,i,j,\Psi)$; it also holds for $\Lambda_{i,j,k}f$ and $\Lambda_{i,j}f-\Lambda_{i,j,k}f$, instead of $\Lambda_{i,j}f$.
	
	Since $\Lambda_{i,j,k}f$ is a mollification of $\Lambda_{i,j}f$, we have
	\begin{align*}
		\|\Lambda_{i,j,k}f\|_{H_p^{\gamma}}\lesssim_{N_3} \|\Lambda_{i,j}f\|_{H_p^{\gamma}}\quad\text{and}\quad 
		\lim_{k\rightarrow \infty}\big\|\Lambda_{i,j}f-\Lambda_{i,j,k}f\big\|_{H_{p}^{\gamma}}=0\,,
	\end{align*}
	where $N_3=N(d,p,\theta,\gamma,\Psi,i,j,\eta)$.
\end{proof}

		\section*{Acknowledgment}
		The author wishes to express deep appreciation to Prof. Kyeong-Hun Kim for valuable comments and the referee for the careful reading.
		The author is sincerely grateful to Dr. Jin Bong Lee for the careful assistance in proofreading and editing this paper.

		

\begin{thebibliography}{10}
	
	\bibitem{convexAdo}
	{\sc V.~Adolfsson}, {\em {$L^p$}-integrability of the second order derivatives
		of {G}reen potentials in convex domains}, Pacific J. Math., 159 (1993),
	pp.~201--225.
	
	\bibitem{aikawa1991}
	{\sc H.~Aikawa}, {\em Quasiadditivity of {R}iesz capacity}, Math. Scand., 69
	(1991), pp.~15--30.
	
	\bibitem{aikawa1997}
	{\sc H.~Aikawa}, {\em Capacity and
		{H}ausdorff content of certain enlarged sets}, Mem. Fac. Sci. Eng. Shimane
	Univ. Ser. B Math. Sci., 30 (1997), pp.~1--21.
	
	\bibitem{aikawa2002}
	{\sc H.~Aikawa}, {\em H\"{o}lder
		continuity of the {D}irichlet solution for a general domain}, Bull. London
	Math. Soc., 34 (2002), pp.~691--702.
	
	\bibitem{Aikawa2009}
	{\sc H.~Aikawa}, {\em Boundary {H}arnack
		principle and the quasihyperbolic boundary condition}, in Sobolev spaces in
	mathematics. {II}, vol.~9 of Int. Math. Ser. (N. Y.), Springer, New York,
	2009, pp.~19--30.
	
	\bibitem{AA}
	{\sc A.~Ancona}, {\em On strong barriers and an inequality of {H}ardy for
		domains in {${\bf R}^n$}}, J. London Math. Soc. (2), 34 (1986), pp.~274--290.
	
	\bibitem{AG}
	{\sc D.~H. Armitage and S.~J. Gardiner}, {\em Classical potential theory},
	Springer Monographs in Mathematics, Springer-Verlag London, Ltd., London,
	2001.
	
	\bibitem{BEL}
	{\sc A.~A. Balinsky, W.~D. Evans, and R.~T. Lewis}, {\em The analysis and
		geometry of {H}ardy's inequality}, Universitext, Springer, Cham, 2015.
	
	\bibitem{BCG}
	{\sc C.~Betz, G.~A. C\'{a}mera, and H.~Gzyl}, {\em Bounds for the first
		eigenvalue of a spherical cap}, Appl. Math. Optim., 10 (1983), pp.~193--202.
	
	\bibitem{BK2006}
	{\sc M.~Borsuk and V.~Kondratiev}, {\em Elliptic boundary value problems of
		second order in piecewise smooth domains}, vol.~69 of North-Holland
	Mathematical Library, Elsevier Science B.V., Amsterdam, 2006.
	
	\bibitem{Relliptic}
	{\sc S.-S. Byun and L.~Wang}, {\em Elliptic equations with {BMO} coefficients
		in {R}eifenberg domains}, Comm. Pure Appl. Math., 57 (2004), pp.~1283--1310.
	
	\bibitem{CKL}
	{\sc L.~Capogna, C.~E. Kenig, and L.~Lanzani}, {\em Harmonic measure. Geometric
		and analytic points of view}, vol.~35 of University Lecture Series, American
	Mathematical Society, Providence, RI, 2005.
	
	\bibitem{Reifweight2}
	{\sc J.~Choi and D.~Kim}, {\em Weighted {$L_{p,q}$}-estimates for higher order
		elliptic and parabolic systems with {${\rm BMO}_x$} coefficients on
		{R}eifenberg flat domains}, Calc. Var. Partial Differential Equations, 58
	(2019), pp.~Paper No. 90, 29.
	
	\bibitem{FDB}
	{\sc G.~M. Constantine and T.~H. Savits}, {\em A multivariate {F}a\`a di
		{B}runo formula with applications}, Trans. Amer. Math. Soc., 348 (1996),
	pp.~503--520.
	
	\bibitem{DD_2008}
	{\sc D.~Daners}, {\em Domain perturbation for linear and semi-linear boundary
		value problems}, in Handbook of differential equations: stationary partial
	differential equations. {V}ol. {VI}, Handb. Differ. Equ.,
	Elsevier/North-Holland, Amsterdam, 2008, pp.~1--81.
	
	\bibitem{RA}
	{\sc R.~A. DeVore and R.~C. Sharpley}, {\em Maximal functions measuring
		smoothness}, Mem. Amer. Math. Soc., 47 (1984), pp.~viii+115.
	
	\bibitem{DK2015}
	{\sc H.~Dong and D.~Kim}, {\em Elliptic and parabolic equations with measurable coefficients in weighted {S}obolev spaces}, Adv. Math., 274 (2015),
	pp.~681--735.
	
	\bibitem{DongKim}
	{\sc H.~Dong and D.~Kim}, {\em On {$L_p$}-estimates for elliptic and parabolic
		equations with {$A_p$} weights}, Trans. Amer. Math. Soc., 370 (2018),
	pp.~5081--5130.
	
	\bibitem{EvansPDE}
	{\sc L.~C. Evans}, {\em Partial differential equations}, vol.~19 of Graduate Studies in
	Mathematics, American Mathematical Society, Providence, RI, 1998.
	
	\bibitem{FH}
	{\sc S.~Friedland and W.~K. Hayman}, {\em Eigenvalue inequalities for the
		{D}irichlet problem on spheres and the growth of subharmonic functions},
	Comment. Math. Helv., 51 (1976), pp.~133--161.
	
	\bibitem{convexFromm}
	{\sc S.~J. Fromm}, {\em Potential space estimates for {G}reen potentials in
		convex domains}, Proc. Amer. Math. Soc., 119 (1993), pp.~225--233.
	
	\bibitem{nonsmoothGris}
	{\sc P.~Grisvard}, {\em Elliptic problems in nonsmooth domains}, vol.~24 of
	Monographs and Studies in Mathematics, Pitman (Advanced Publishing Program),
	Boston, MA, 1985.
	
	\bibitem{Iwaniec1983}
	{\sc T.~Iwaniec}, {\em Projections onto gradient fields and ${L}^{p}$-estimates
		for degenerated elliptic operators}, Studia Mathematica, 75 (1983),
	pp.~293--312.
	
	\bibitem{Pseudodiff}
	{\sc N.~Jacob}, {\em Pseudo differential operators and {M}arkov processes.
		{V}ol. {III}}, Imperial College Press, London, 2005.
	
	\bibitem{kenig}
	{\sc D.~Jerison and C.~E. Kenig}, {\em The inhomogeneous {D}irichlet problem in
		{L}ipschitz domains}, J. Funct. Anal., 130 (1995), pp.~161--219.
	
	\bibitem{KenigToro3}
	{\sc C.~E. Kenig and T.~Toro}, {\em Harmonic measure on locally flat domains},
	Duke Math. J., 87 (1997), pp.~509--551.
	
	\bibitem{KenigToro}
	{\sc C.~E. Kenig and T.~Toro}, {\em Poisson kernel
		characterization of {R}eifenberg flat chord arc domains}, Ann. Sci. \'{E}cole
	Norm. Sup. (4), 36 (2003), pp.~323--401.
	
	\bibitem{KilKos1994}
	{\sc T.~Kilpeläinen and P.~Koskela}, {\em Global integrability of the
		gradients of solutions to partial differential equations}, Nonlinear
	Analysis: Theory, Methods \& Applications, 23 (1994), pp.~899--909.
	
	\bibitem{doyoon}
	{\sc D.~Kim}, {\em Elliptic equations with nonzero boundary conditions in
		weighted {S}obolev spaces}, J. Math. Anal. Appl., 337 (2008), pp.~1465--1479.
	
	\bibitem{KK2007_variable}
	{\sc D.~Kim and N. V. Krylov}, {\em Elliptic differential equations with coefficients measurable
		with respect to one variable and {VMO} with respect to the
		others}, SIAM J. Math. Anal., 39 (2007), no. 2, pp.~489--506.
	
	\bibitem{Kim2014}
	{\sc K.-H. Kim}, {\em A weighted {S}obolev space theory of parabolic stochastic
		{PDE}s on non-smooth domains}, J. Theoret. Probab., 27 (2014), pp.~107--136.
	
	\bibitem{KK2004}
	{\sc K.-H. Kim and N.~V. Krylov}, {\em On the {S}obolev space theory of
		parabolic and elliptic equations in {$C^1$} domains}, SIAM J. Math. Anal., 36
	(2004), pp.~618--642.
	
	
	\bibitem{KL2013}
	{\sc K.-H. Kim and K.~Lee}, {\em A weighted $L_p$-theory for parabolic PDEs with BMO coefficients on $C^1$-domains}, J. Differential Equations, 254 (2013),
	pp.~368--407.	
	
	
	\bibitem{ConicPDE}
	{\sc K.-H. Kim, K.~Lee, and J.~Seo}, {\em A weighted {S}obolev regularity
		theory of the parabolic equations with measurable coefficients on conic
		domains in {$\bR^d$}}, J. Differential Equations, 291 (2021),
	pp.~154--194.
	
	\bibitem{kinnunen2021}
	{\sc J.~Kinnunen, J.~Lehrb{\"a}ck, and A.~V{\"a}h{\"a}kangas}, {\em Maximal
		function methods for Sobolev spaces}, vol.~257, American Mathematical Soc.,
	2021.
	
	\bibitem{Kinhardy}
	{\sc J.~Kinnunen and O.~Martio}, {\em Hardy’s inequalities for {S}obolev
		functions}, Mathematical Research Letters, 4 (1997), pp.~489--500.
	
	\bibitem{KZ}
	{\sc P.~Koskela and X.~Zhong}, {\em Hardy's inequality and the boundary size},
	Proc. Amer. Math. Soc., 131 (2003), pp.~1151--1158.
	
	\bibitem{Krylov1994}
	{\sc N.~V. Krylov}, {\em A {$W^n_2$}-theory of the {D}irichlet problem for {SPDE}s in
		general smooth domains}, Probab. Theory Related Fields, 98
	(1994), pp.~389--421.
	
	
	\bibitem{Krylov1999-0}
	{\sc N.~V. Krylov}, {\em Some properties of weighted {S}obolev spaces in $\bR^d_+$}, Ann. Scuola Norm. Sup. Pisa Cl. Sci. (4), 28	(1999), pp.~675--693.
	
	
	\bibitem{Krylov1999-1}
	{\sc N.~V. Krylov}, {\em Weighted {S}obolev spaces and {L}aplace's equation and
		the heat equations in a half space}, Comm. Partial Differential Equations, 24
	(1999), pp.~1611--1653.
	
	
	\bibitem{Krylov2001}
	{\sc N.~V. Krylov}, {\em Some properties of traces for stochastic and deterministic parabolic weighted {S}obolev spaces}, J. Funct. Anal., 183 (2001), pp.~1--41.
	
	
	\bibitem{Krylov2007}
	{\sc N.~V. Krylov}, {\em Parabolic and elliptic equations with {VMO} coefficients}, Comm. Partial Differential Equations, 32 (2007), no. 1--3, pp.~453--475.
	
	
	\bibitem{Krylov2008}
	{\sc N.~V. Krylov}, {\em Lectures on elliptic
		and parabolic equations in {S}obolev spaces}, vol.~96 of Graduate Studies in
	Mathematics, American Mathematical Society, Providence, RI, 2008.
	
		
	\bibitem{Krylov2009}
	{\sc N. V.~Krylov}, {\em Second-order
		elliptic equations with variably partially {VMO} coefficients}, J. Funct.
	Anal., 257 (2009), pp.~1695--1712.


	\bibitem{lehr}
	{\sc J.~Lehrb\"{a}ck}, {\em Weighted {H}ardy inequalities and the size of the
		boundary}, Manuscripta Math., 127 (2008), pp.~249--273.
	
	\bibitem{MR3168477}
	{\sc J.~Lehrb\"{a}ck}, {\em Weighted {H}ardy
		inequalities beyond {L}ipschitz domains}, Proc. Amer. Math. Soc., 142 (2014),
	pp.~1705--1715.
	
	\bibitem{LT}
	{\sc J.~Lehrb\"{a}ck and H.~Tuominen}, {\em A note on the dimensions of
		{A}ssouad and {A}ikawa}, J. Math. Soc. Japan, 65 (2013), pp.~343--356.
	
	\bibitem{lewis}
	{\sc J.~L. Lewis}, {\em Uniformly fat sets}, Trans. Amer. Math. Soc., 308
	(1988), pp.~177--196.
	
	
	\bibitem{Lo0}
	{\sc S.~V. Lototsky}, {\em Dirichlet problem for stochastic parabolic equations in smooth domains}, Stochastics Stoch. Rep., 68 (1999), pp.~145--175.


	\bibitem{Lo1}
	{\sc S.~V. Lototsky}, {\em Sobolev spaces with weights in domains and boundary
		value problems for degenerate elliptic equations}, Methods Appl. Anal., 7
	(2000), pp.~195--204.
	
	\bibitem{MMP}
	{\sc M.~Marcus, V.~J. Mizel, and Y.~Pinchover}, {\em On the best constant for
		{H}ardy's inequality in {$\bR^n$}}, Trans. Amer. Math. Soc., 350 (1998),
	pp.~3237--3255.
	
	\bibitem{Mattila}
	{\sc P.~Mattila}, {\em Geometry of sets and measures in {E}uclidean spaces},
	vol.~44 of Cambridge Studies in Advanced Mathematics, Cambridge University
	Press, Cambridge, 1995.
	
	\bibitem{MR0492821}
	{\sc V.~G. Maz'ja and B.~A. Plamenevski\u{\i}}, {\em Estimates in {$L\sb{p}$}
		and in {H}\"{o}lder classes, and the {M}iranda-{A}gmon maximum principle for
		the solutions of elliptic boundary value problems in domains with singular
		points on the boundary}, Math. Nachr., 81 (1978), pp.~25--82.
	
	\bibitem{MNP}
	{\sc V.~Maz'Ya, S.~Nazarov, and B.~Plamenevskij}, {\em Asymptotic theory of
		elliptic boundary value problems in singularly perturbed domains}, vol.~1,
	Springer Science \& Business Media, 2000.
	
	\bibitem{MR}
	{\sc V.~Maz'ya and J.~Rossmann}, {\em Elliptic equations in polyhedral
		domains}, vol.~162 of Mathematical Surveys and Monographs, American
	Mathematical Society, Providence, RI, 2010.
	
	\bibitem{MBX2022}
	{\sc M.~Mourgoglou, B.~Poggi, and X.~Tolsa}, {\em Solvability of the
		{P}oisson-{D}irichlet problem with interior data in ${L}^{p'}$-{C}arleson
		spaces and its applications to the ${L}^{p}$-regularity problem}, J. Eur. Math. Soc., published online first (2025).
	
	\bibitem{Reifcondition}
	{\sc E.~R. Reifenberg}, {\em Solution of the {P}lateau {P}roblem for
		{$m$}-dimensional surfaces of varying topological type}, Acta Math., 104
	(1960), pp.~1--92.
	
	\bibitem{rudin}
	{\sc W.~Rudin}, {\em Functional analysis}, International Series in Pure and
	Applied Mathematics, McGraw-Hill, Inc., New York, second~ed., 1991.
	
	\bibitem{Seo202304}
	{\sc J.~Seo}, {\em Sobolev space theory for {P}oisson's and the heat equations
		in non-smooth domains via superharmonic functions and {H}ardy's inequality},
	arXiv preprint arXiv:2304.10451v1.
	
	
	\bibitem{Seo202411}
	{\sc J.~Seo}, {\em Weighted Sobolev space theory for the heat equation and the time-fractional heat equation in non-smooth domains},
	arXiv preprint arXiv:2411.06761.
	
	
	\bibitem{TT}
	{\sc T.~Toro}, {\em Doubling and flatness: geometry of measures}, Notices Amer.
	Math. Soc., 44 (1997), pp.~1087--1094.
	
	\bibitem{TF}
	{\sc F.~Tr\`eves}, {\em Basic linear partial differential equations}, Pure and
	Applied Mathematics, Vol. 62, Academic Press [Harcourt Brace Jovanovich,
	Publishers], New York-London, 1975.
	
	\bibitem{VM}
	{\sc M.~Vuorinen}, {\em On the {H}arnack constant and the boundary behavior of
		{H}arnack functions}, Ann. Acad. Sci. Fenn. Ser. A I Math., 7 (1982),
	pp.~259--277.
	
	\bibitem{ward}
	{\sc A.~D. Ward}, {\em On essential self-adjointness, confining potentials
		{$\&$} the {$L_p$}-{H}ardy inequality}, PhD thesis, Massey University,
	Albany, 2014.
	
	\bibitem{Wu}
	{\sc J.~M.~G. Wu}, {\em Comparisons of kernel functions, boundary {H}arnack
		principle and relative {F}atou theorem on {L}ipschitz domains}, Ann. Inst.
	Fourier (Grenoble), 28 (1978), pp.~147--167, vi.
	
\end{thebibliography}
		
%
%

	\end{document}